\documentclass[a4paper,11pt,oneside,reqno]{amsart}
\usepackage{silence}
\allowdisplaybreaks
\usepackage[utf8]{inputenc}
\usepackage[pdftex]{graphicx}
\usepackage[foot]{amsaddr}
\usepackage{amssymb}
\usepackage{esint}
\usepackage{amsmath}
\usepackage{amsfonts}
\usepackage{amsthm}
\usepackage{bm}
\usepackage{mathrsfs}
\usepackage{mathtools}
\usepackage{comment}
\mathtoolsset{showonlyrefs=true}
\usepackage{appendix}
\usepackage{enumerate}
\usepackage[a4paper, total={6.6in, 9.7in}]{geometry}
\usepackage[stretch=30,shrink=30]{microtype}
\usepackage{stmaryrd}
\usepackage[normalem]{ulem} 
\usepackage{xcolor}
\definecolor{antiquefuchsia}{rgb}{0.57, 0.36, 0.51}
\definecolor{azure}{rgb}{0.0, 0.5, 1.0}

\usepackage[backref=page, colorlinks = true, linkcolor = black, citecolor = antiquefuchsia]{hyperref}
\hypersetup{
    pdftitle={High-Dimensional Families of Minimal Surfaces of Arbitrary Genus in Round Spheres},
    pdfauthor={Riccardo Caniato},
    pdfsubject={Minimal surfaces and twistor geometry},
    pdfkeywords={minimal surfaces, superminimal immersions, twistor spaces, Riemann surfaces}
}
\renewcommand*{\backref}[1]{}
\renewcommand*{\backrefalt}[4]{%
    \ifcase #1 (Not cited.)%
    \or        (Cited on page~#2.)%
    \else      (Cited on pages~#2.)%
    \fi}

\makeatletter
\def\th@plain{%
	\thm@notefont{}
	\itshape 
}
\def\th@definition{%
	\thm@notefont{}
	\normalfont 
}
\makeatother

\numberwithin{equation}{section}

\newtheorem{theorem}{Theorem}[section]
\newtheorem{lemma}[theorem]{Lemma}
\newtheorem{proposition}[theorem]{Proposition}
\newtheorem{corollary}[theorem]{Corollary}

\newtheorem{conjecture}[theorem]{Conjecture}

\theoremstyle{definition}
\newtheorem{definition}[theorem]{Definition}

\newtheorem{remark}[theorem]{Remark}

\newcommand{\n}{\mathbb{N}}
\newcommand{\s}{\mathbb{S}}

\newcommand{\z}{\mathbb{Z}}
\renewcommand{\r}{\mathbb{R}}
\renewcommand{\c}{\mathbb{C}}

\DeclareMathOperator{\HH}{HH}

\DeclareMathOperator{\dist}{dist}

\DeclareMathOperator{\vol}{vol}

\DeclareMathOperator{\Id}{Id}

\renewcommand{\H}{\mathrm{H}}

\newcommand{\D}{\mathbb{D}}
\newcommand{\cp}{\mathbb{CP}}

\title[High-Dimensional Families of Minimal Surfaces]{High-Dimensional Families of Minimal Surfaces\\of Arbitrary Genus in Round Spheres}

\author[Riccardo Caniato]{Riccardo Caniato}
\address{Mathematics Institute, University of Warwick,
Coventry, United Kingdom}
\email{riccardo.caniato@warwick.ac.uk}
\date{\today} 

\begin{document}
\begin{abstract}
    We prove the existence of arbitrarily high-dimensional families of minimal surfaces of any prescribed genus and conformal structure in even-dimensional round spheres. More precisely, let $n\geq2$ and let $\Sigma$ be any closed Riemann surface of genus $g$. We construct a sequence of degrees $d_\ell\to+\infty$ such that, for every $\ell$, there exists a complex manifold of complex dimension $2d_\ell+n^2(1-g)$ consisting of linearly full branched superminimal immersions of $\Sigma$ into $\mathbb S^{2n}$ of degree $d_\ell$. In particular, this yields parametrised families of linearly full branched minimal immersions whose dimensions tend to infinity. 
\end{abstract}
\maketitle
\tableofcontents
\section{Introduction}
\subsection{Minimal immersions and harmonic maps}
The study of minimal surfaces in round spheres has known a spectacular development through several complementary viewpoints, combining differential geometry, complex analysis, spectral theory, integrable systems, and variational methods. Let $\Sigma$ be a closed smooth surface and let $\varphi:\Sigma\rightarrow\s^m$ be a smooth immersion. We denote by $g_{\varphi}:=\varphi^*g_{\s^m}$ the metric induced on $\Sigma$ by $\varphi$ from the round metric $g_{\s^m}$ on $\s^m$. We recall that the symmetric bilinear form
\begin{align*}
    A_\varphi(X,Y):=\big(\nabla^{\mathbb S^m}_{d\varphi(X)}d\varphi(Y)\big)^\perp=\nabla^{\mathbb S^m}_{d\varphi(X)}d\varphi(Y) -d\varphi\big(\nabla^{g_\varphi}_{X}Y\big)\qquad\forall\,X,Y\in\mathfrak{X}(\Sigma)
\end{align*}
is called the \emph{second fundamental form} of $\varphi$ and encodes the extrinsic curvature of $\varphi(\Sigma)$ in the ambient $\s^m$.\footnote{\,Here, $\nabla^{\s^m}$ is the Levi--Civita connection on $(\s^m,g_{\s^m})$ and $(\,\cdot\,)^\perp$ denotes the orthogonal projection onto the normal bundle of $\varphi$.} 
The \emph{mean curvature vector} of $\varphi$ is given by
\begin{align*}
    H_\varphi:=\frac{1}{2}\operatorname{tr}_{g_\varphi}A_\varphi,
\end{align*}
and the immersion $\varphi$ is said to be \emph{minimal} if $H_\varphi\equiv0$. Equivalently, $\varphi$ is minimal if it is a critical point of the \textit{area functional}
\begin{align*}
    \operatorname{Area}(\varphi):=\int_\Sigma d\vol_{g_\varphi}
\end{align*}
with respect to arbitrary smooth variations. When $\Sigma$ is endowed with a conformal structure, the preceding variational characterization is closely related to the theory of harmonic maps. Let $[g]$ be a conformal class of Riemannian metrics on $\Sigma$ and let $\varphi:(\Sigma,[g])\rightarrow \s^m$ be a smooth map. Given any representative $g\in[g]$, the \emph{Dirichlet energy} of $\varphi$ is defined by
\begin{align*}
    \operatorname{E}(\varphi):=\frac{1}{2}\int_\Sigma \lvert d\varphi\rvert_g^2\,d\vol_g.
\end{align*}
Since the domain is two-dimensional, this quantity is invariant under conformal changes of $g$ and therefore depends only on the conformal structure of $\Sigma$. The map $\varphi$ is said to be \emph{harmonic} if it is a critical point of $\operatorname{E}$ with respect to arbitrary smooth variations. The link between minimal immersions and harmonic maps is given by the notion of conformality. We say that $\varphi$ is \emph{weakly conformal} if there exists a non-negative function $\lambda:\Sigma\to[0,+\infty)$ such that
\begin{align*}
    \varphi^*g_{\mathbb S^m}=\lambda^2g.
\end{align*}
%
%
%
It is a classical fact that a conformal immersion of a Riemann surface is harmonic if and only if it is minimal. More generally, a nonconstant weakly conformal harmonic map has at most isolated points at which its differential vanishes. These points are called \emph{branch points}, and away from them the map is a conformal minimal immersion. Accordingly, a \emph{branched minimal immersion}
\begin{align*}
    \varphi:(\Sigma,[g])\rightarrow\mathbb S^m
\end{align*}
is defined as a nonconstant weakly conformal harmonic map. Equivalently, there exists a discrete subset $\mathcal B_\varphi\subset\Sigma$ such that $\varphi$ restricts to a conformal minimal immersion on $\Sigma\smallsetminus\mathcal B_\varphi$, while
\begin{align*}
    d\varphi(x)=0 \qquad\forall\,x\in\mathcal B_\varphi.
\end{align*}
The set $\mathcal B_\varphi$ is called the \emph{branch locus} of $\varphi$. For the classical theory of branch points and the regularity theory of harmonic maps on two-dimensional domains, we refer the reader to \cite{GulliverOssermanRoyden1973,SacksUhlenbeck1981,Helein1991,Riviere2007}, among others.

\subsection{The codimension-one case}
The codimension-one theory of minimal surfaces in the round three-sphere has played a central role in the development of the subject. While Almgren proved that every minimal two-sphere in $\s^3$ is totally geodesic (see \cite{Almgren1966}), Lawson showed that this rigidity disappears once the topology is allowed to vary, constructing closed embedded minimal surfaces of arbitrary genus in $\s^3$ (see \cite{Lawson1970}). His work initiated a vast literature on the construction of embedded examples, including symmetry-based methods of Hsiang--Lawson and Karcher--Pinkall--Sterling, as well as more recent gluing and desingularization constructions of Kapouleas--Yang, Choe--Soret, and Kapouleas--Wiygul (see \cite{HsiangLawson1971,KarcherPinkallSterling1988,KapouleasYang,ChoeSoret,KapouleasWiygul}). At the same time, fundamental rigidity and compactness results were obtained by Choi--Schoen, Urbano, and Brendle (see \cite{ChoiSchoen1985,Urbano1990,Brendle2013}). 

A different variational approach is provided by min-max theory, whose geometric-measure-theoretic foundations were laid by Almgren and Pitts, while a closely related three-dimensional theory was developed by Simon--Smith (see \cite{Almgren1962,Pitts1981,Smith1982,ColdingDeLellis2003}). Building on these foundations, Marques--Neves proved the Willmore conjecture and established the existence of infinitely many embedded minimal hypersurfaces under positive Ricci curvature (see \cite{MarquesNeves2014,MarquesNeves2017}). A central subsequent development was Zhou's proof of the multiplicity-one conjecture of Marques--Neves, showing that for bumpy metrics the minimal hypersurfaces realizing the volume spectrum can be chosen two-sided and with multiplicity one (see \cite{ZhouMultiplicityOne}). Together with the Weyl law, density and equidistribution results of Liokumovich--Marques--Neves, Irie--Marques--Neves, and Marques--Neves--Song, these advances established a remarkably detailed picture of the min-max minimal spectrum for generic metrics; Song ultimately removed the genericity and curvature assumptions and proved the existence of infinitely many closed embedded minimal hypersurfaces in every closed Riemannian manifold of dimension between three and seven (see \cite{LiokumovichMarquesNeves2018,IrieMarquesNeves2018,MarquesNevesSong2019,Song2023}).

Minimal tori in $\s^3$ also admit a powerful description through
integrable systems. Beginning with Hitchin's spectral-curve construction for harmonic tori and related developments by Bobenko and Burstall--Ferus--Pedit--Pinkall, this approach translates the harmonic map equation into algebro-geometric data (see \cite{Hitchin1990,Bobenko1991,BurstallFerusPeditPinkall1993}). Particularly relevant to the present work is Carberry's construction of real $N$-dimensional families of linearly full minimal tori in $\s^3$, for every $N\geq1$, including positive-dimensional families with fixed rectangular conformal structure (see \cite{Carberry}). Thus, arbitrarily high-dimensional families already occur in codimension one, although Carberry's construction is confined to genus one and special conformal structures.

A further source of minimal immersions in spheres comes from spectral geometry. A classical theorem of Takahashi characterizes minimal immersions into round spheres in terms of Laplace eigenfunctions (see \cite{Takahashi1966}). Building on this connection, the relation between extremal Laplace eigenvalues and minimal immersions was developed in work of Yang--Yau, Li--Yau, El Soufi--Ilias, and Montiel--Ros, and subsequently exploited by Nadirashvili in his solution of Berger's problem for the two-torus (see \cite{YangYau1980,LiYau1982,ElSoufiIlias1986,MontielRos1986,Nadirashvili1996}). The existence and structure of extremal metrics were further studied by Nadirashvili--Sire, Petrides, and Karpukhin--Nadirashvili--Penskoi--Polterovich (see \cite{NadirashviliSire,Petrides2014,KarpukhinNadirashviliPenskoiPolterovich}). More recently, Karpukhin--Stern related conformal eigenvalues to a min-max theory for sphere-valued harmonic maps, while Karpukhin--Kusner--McGrath--Stern used equivariant eigenvalue
optimization to construct new embedded minimal surfaces in $\s^3$;
their large-topology geometry was subsequently analyzed by Karpukhin--McGrath--Stern (see \cite{KarpukhinStern,KarpukhinKusnerMcGrathStern,KarpukhinMcGrathStern2026}). Related developments of Petrides and Karpukhin--Petrides--Stern have recently completed the existence theory for metrics maximizing the first normalized Laplace eigenvalue on closed surfaces (see \cite{Petrides2024,KarpukhinPetridesStern}).
\subsection{Higher codimension and twistor geometry}
Despite the extensive understanding achieved in codimension one, comparatively little is known about minimal surfaces in high-dimensional round spheres. Indeed, higher codimension comes with a substantial loss of rigidity, allowing minimal surfaces to display a far richer and less constrained geometry. On the other hand, higher codimension also brings into play additional complex-analytic structures which have no direct analogue in $\s^3$, and which can be effectively exploited to explore this new landscape.

The first non-totally-geodesic linearly full examples of minimal surfaces in higher-dimensional round spheres go back to Bor{\accent23 u}vka, who constructed minimal two-spheres of constant Gaussian curvature in $\s^{2n}$ (see \cite{Boruvka1933}).\footnote{\,Given a Hilbert space $H$, we denote by $\s_H$ the unit sphere in $H$ centred at $0\in H$. We say that a smooth map $\varphi:\Sigma\to\s_H$ is \textit{linearly full} if $\varphi(\Sigma)$ is not contained in any proper closed linear subspace of $H$.} A systematic study of minimal two-spheres in round spheres of arbitrary dimension was subsequently initiated by Calabi, who proved, in particular, that a linearly full minimal immersion of $\s^2$ into a round sphere can occur only when the ambient dimension is even, and uncovered the strong isotropy properties satisfied by such maps (see \cite{Calabi1967}). These ideas were further developed by Chern and Barbosa through the association of a minimal two-sphere with a totally isotropic holomorphic directrix curve in complex projective space (see \cite{Chern1970, Barbosa1975}). This isotropy phenomenon leads naturally to the distinguished class of superminimal surfaces. Given a branched minimal immersion $\varphi:\Sigma\to\s^{2n}$, let $z$ be a local holomorphic coordinate on $\Sigma\smallsetminus\mathcal B_\varphi$. For every $k\in\n$, we denote by
\begin{align*}
    \mathbf{O}_\varphi^{k}(x):=\operatorname{span}_{\c}\big\{\partial_{\bar z}\varphi(x),\partial_{\bar z}^2\varphi(x),\ldots,\partial_{\bar z}^k\varphi(x)\big\}\subset\c^{2n+1}
\end{align*}
the $k$-th \textit{osculating space} of $\varphi$ at $x$. The immersion $\varphi$ is said to be \emph{superminimal} if $\mathbf{O}_\varphi^{k}(x)$ is totally isotropic for every $x\in\Sigma\smallsetminus\mathcal B_\varphi$ and every $k\in\n$.\footnote{\,Recall that a complex linear subspace $P\subset\c^m$ is called \textit{totally isotropic} if $(v,w)=0$ for every $v,w\in P$, where $(\,\cdot\,,\,\cdot)$ denotes the bilinear extension to $\c^m$ of the Euclidean inner product on $\r^m$.} For $\Sigma=\s^2$, the results of Calabi and Chern imply that every linearly full minimal immersion into $\s^{2n}$ is superminimal (see \cite{Calabi1967,Chern1970}). The importance of superminimal immersions is that they can be described in terms of holomorphic maps. More precisely, if
\begin{align*}
    \mathscr Z_n:=\operatorname{SO}(2n+1)/\operatorname{U}(n)
\end{align*}
denotes the twistor space of $\s^{2n}$ (see Definition \ref{Definition: twistor bundle}), then a linearly full superminimal immersion admits a canonical holomorphic lift to $\mathscr Z_n$, satisfying a suitable \textit{horizontality} condition (see Definition \ref{Definition: horizontality}). Conversely, holomorphic horizontal curves in $\mathscr Z_n$ project to superminimal surfaces in $\s^{2n}$.

The theory of superminimal surfaces beyond the genus-zero setting was pioneered by Bryant (see \cite{Bryant1982,BryantOctonions1982}). In the particular case $n=2$, we have $\mathscr Z_2\cong\mathbb{CP}^3$, and the horizontal distribution of this twistor space is the standard holomorphic contact distribution on $\cp^3$. Exploiting this structure, Bryant proved that every compact Riemann surface admits a conformal minimal---in fact superminimal---immersion into $\s^4$. Hano subsequently extended this twistor-theoretic existence theory to every even-dimensional sphere, proving that every compact Riemann surface admits a linearly full conformal minimal immersion into $\s^{2n}$ for every $n\geq2$ (see \cite{Hano1996}). The underlying superminimal system was later described explicitly by Chi--Fern\'andez--Wu and Fern\'andez, who obtained a formulation of the problem in terms of normalized potentials and constructions of superminimal surfaces of arbitrarily large area for any prescribed compact Riemann surface (see \cite{ChiFernandezWu1999}, \cite{Fernandez2003}).

Beyond the superminimal setting, several complementary approaches have led to substantial new developments on higher-codimensional minimal surfaces. From a variational perspective, Colding--Minicozzi developed a min-max construction for sweepouts by two-spheres based on conformal reparametrization and harmonic replacement, and Zhou subsequently extended this approach first to tori and then to surfaces of arbitrary genus, incorporating the variation and possible degeneration of the conformal structure and obtaining bubble-tree limits of branched minimal surfaces (see \cite{ColdingMinicozzi2008,Zhou2010,Zhou2017}). A different parametric min-max theory for minimal surfaces in arbitrary codimension was developed by Rivi\`ere using a viscosity method (see \cite{Riviere2017}). The resulting parametrized stationary varifolds were further analyzed by Pigati--Rivi\`ere, who established their regularity and proved the multiplicity-one property of the associated min-max minimal surfaces (see \cite{PigatiRiviere2020a,PigatiRiviere2020b}). Very recently, Cheng--Zhou returned to the harmonic-replacement framework and, combining it with a construction of non-trivial sweepouts by higher-genus surfaces, proved that every closed Riemannian manifold of dimension four or five contains a closed branched minimal immersion (see \cite{ChengZhou2026}). In a different direction, using Atiyah's structure theory of holomorphic vector bundles over elliptic curves, Fraser--Schoen obtained systolic rigidity results for stable and covering-stable minimal tori in arbitrary codimension (see \cite{FraserSchoen2025}). On the other hand, a new high-codimensional scenario has emerged from the work of Song, connecting harmonic maps and minimal surfaces in spheres with representation theory, geometric group theory, and random matrix theory. Building on his formulation of the spherical Plateau problem (see \cite{SongSphericalPlateau}), Song studied equivariant harmonic maps associated with high-dimensional unitary representations and showed that their geometry exhibits a remarkable asymptotic probabilistic rigidity (see \cite{SongRandomHarmonicMaps}). In particular, for random unitary representations, the induced metrics of the corresponding harmonic maps concentrate around a hyperbolic metric; as an application, he constructed sequences of closed branched minimal surfaces in spheres of increasing dimension which converge locally in the Benjamini--Schramm sense to the hyperbolic plane and whose Gaussian curvatures converge to a negative constant in an averaged sense. Motivated by this asymptotic picture, in recent joint work with Li and Song we studied the corresponding limiting problem in the infinite-dimensional Hilbert sphere. For representations weakly equivalent to the regular representation, we showed that equivariant area minimization rigidly determines not only the intrinsic, but the extrinsic geometry of the minimizers (see \cite{CaniatoLiSong2025}). A central ingredient in this rigidity theorem is a classification of hyperbolic minimal surfaces in the infinite-dimensional Hilbert sphere, which extends to infinite dimensions the classical results of Calabi, Kenmotsu, and Bryant for minimal surfaces of constant Gaussian curvature in finite-dimensional round spheres (see \cite{Calabi1967,Kenmotsu1976,Bryant3}).
\subsection{Main results and methods}
We first fix some notation. Recall that the Picard group $\operatorname{Pic}(\mathscr Z_n)$ of $\mathscr{Z}_n$ is isomorphic to $\mathbb Z$. We denote by $\mathcal O_{\mathscr Z_n}(1)$ its unique ample generator. Accordingly, we define the \emph{degree} of a holomorphic map $\psi:\Sigma\to\mathscr Z_n$ from a closed Riemann surface $\Sigma$ by
\begin{align*}
    \deg(\psi)
    :=
    \big\langle
        \psi^*c_1\big(\mathcal O_{\mathscr Z_n}(1)\big),
        [\Sigma]
    \big\rangle
    =
    \big\langle
        c_1\big(\mathcal O_{\mathscr Z_n}(1)\big),
        \psi_*[\Sigma]
    \big\rangle.
\end{align*}
If $\varphi:\Sigma\to\s^{2n}$ is a linearly full branched superminimal
immersion, we define the \emph{degree} of $\varphi$ to be the degree of its unique twistor lift (see Definition~\ref{Definition: degree}). For every
$d\in\mathbb N\smallsetminus\{0\}$, we denote by
\begin{align*}
    \operatorname{SM}_{d}^{f}(\Sigma,\s^{2n})
\end{align*}
the set of all linearly full branched superminimal immersions
$\varphi:\Sigma\to\s^{2n}$ of degree $d$. Note that $\operatorname{SM}_{d}^{f}(\Sigma,\s^{2n})$ is a set of parametrized maps, and no quotient by reparametrizations of $\Sigma$ or by ambient isometries is taken.

The main result of this paper is the following.
\begin{theorem}\label{Theorem: main statement 1}
    Let $n\geq2$ and $g\in\mathbb N$, and let $\Sigma$ be any closed
    Riemann surface of genus $g$. Then there exists a sequence $\{d_\ell\}_{\ell\ge 1}\subset\mathbb N\smallsetminus\{0\}$ such that
    \begin{align*}
        d_\ell\rightarrow+\infty \qquad\mbox{ as }\, \ell\to+\infty
    \end{align*}
    and, for every $\ell\ge 1$, there exists a complex manifold
    \begin{align*}
        X_{\Sigma,d_\ell}\subset\operatorname{SM}_{d_\ell}^{f}(\Sigma,\s^{2n})
    \end{align*}
    of complex dimension
    \begin{align*}
        \dim_{\mathbb C}X_{\Sigma,d_\ell}=2d_\ell+n^2(1-g).
    \end{align*}
\end{theorem}
The strength of Theorem~\ref{Theorem: main statement 1} lies not merely in the existence of branched superminimal immersions, but in the size of the families in which they occur. Indeed, we have already observed that classical works \cite{Bryant1982,Hano1996} established the existence of linearly full superminimal immersions of an arbitrary closed Riemann surface into even-dimensional round spheres. Theorem~\ref{Theorem: main statement 1} shows that, for every prescribed conformal structure and every fixed $n\geq2$, such immersions occur, along a sequence of degrees tending to infinity, in families whose dimensions become arbitrarily large. Since the dimension of the isometry group of $\s^{2n}$ depends only on $n$, this abundance persists even after accounting for ambient congruences: in particular, the dimension of the families modulo ambient congruences remains unbounded. This should be contrasted with the considerable difficulty of producing comparably large families in the codimension-one setting and, more generally, in situations where no twistor correspondence is available.

The precise dimension obtained in Theorem~\ref{Theorem: main statement 1} also places our result naturally between two previously understood regimes. In the case $n=2$, Chi--Mo proved that, for sufficiently large degree $d$, the dimension of every irreducible component of the space of branched superminimal immersions of a fixed Riemann surface of genus $g$ into $\s^4$ lies between $2d+4(1-g)$ and $2d-g+4$ (see \cite{ChiMo1996}), while Chi subsequently showed that the lower value
\begin{align*}
2d+4(1-g)
\end{align*}
is attained by a non-totally-geodesic component for infinitely many degrees (see \cite{Chi2000}). At the other end, when $g=0$ and the target dimension is arbitrary, Fern\'andez proved that the space of linearly full harmonic two-spheres of degree $d$ in $\s^{2n}$ has pure dimension
\begin{align*}
2d+n^2
\end{align*}
(see \cite{Fernandez2012}). The dimension
\begin{align*}
2d+n^2(1-g)
\end{align*}
appearing in Theorem~\ref{Theorem: main statement 1} simultaneously recovers Chi's formula when $n=2$ and Fern\'andez's formula when $g=0$. In this sense, our result extends the arbitrary-genus picture from $\s^4$ to every even-dimensional sphere $\s^{2n}$, while at the same time extending the high-codimensional picture from the two-sphere to every prescribed closed Riemann surface. Related second-variation results have recently been obtained by Ball--Madnick~\cite[Theorems~1.3 and~1.4]{BallMadnick2026SecondVariation}, who establish upper and lower bounds for the Morse index and nullity
of linearly full branched superminimal immersions of compact Riemann surfaces into $\s^{2n}$. Their nullity estimates provide complementary infinitesimal information to the families constructed here. 

The mechanism underlying this extension is the holomorphic nature of the superminimal problem. Under the twistor correspondence, a superminimal immersion $\varphi:\Sigma\to\s^{2n}$ is encoded by its holomorphic horizontal lift. From the analytic point of view, this passage substantially changes the character of the deformation problem. Viewed directly in the sphere, minimal immersions solve a non-linear second-order elliptic system and, after quotienting out the degeneracy arising from reparametrizations, their infinitesimal deformations are governed by a self-adjoint Fredholm Jacobi operator. Since every self-adjoint Fredholm operator has index zero, there is no index-theoretic reason forcing the existence of non-trivial deformations; indeed, whenever the Jacobi operator is non-degenerate modulo the Jacobi fields induced by ambient symmetries, the corresponding minimal immersion is locally rigid up to congruence. The twistor correspondence replaces this problem by a first-order elliptic system of Cauchy--Riemann type, together with the first-order horizontality constraint. The deformation theory of holomorphic maps is of a fundamentally different nature: infinitesimal deformations are described by spaces of holomorphic sections, whose dimensions are governed by the Riemann--Roch theorem and can grow linearly with the degree, while the corresponding obstructions are encoded by higher cohomology groups and can, in favourable regimes, be forced to vanish. Thus, the twistor correspondence does considerably more than provide a convenient reformulation of the minimal surface equation: it transforms a second-order variational problem naturally prone to rigidity into a first-order complex-geometric problem for which non-trivial, and potentially very large, deformation spaces are built into the underlying elliptic theory. The main challenge is to show that this holomorphic flexibility survives the additional horizontality constraint.

It is precisely at this point that the simultaneous passage to arbitrary genus and arbitrary target dimension differs substantially from the previously understood cases. In the case $n=2$, Chi--Mo are able to exploit the contact structure of $\mathscr Z_2\cong\cp^3$ to reduce horizontality to a scalar condition admitting an explicit algebraic description in terms of pairs of pencils and their branch data. This reduction has no direct analogue for $n>2$, where the vertical distribution has complex rank $n(n-1)/2$ and horizontality becomes a coupled system. On the genus side, the passage from $\cp^1$ to an arbitrary Riemann surface is equally significant. The genus-zero theory benefits from the exceptional holomorphic geometry of $\cp^1$: every holomorphic vector bundle splits as a direct sum of line bundles, and the homogeneous Bor{\accent23 u}vka curve can be analyzed globally through explicit algebraic and representation-theoretic data. For a prescribed surface of positive genus there is no comparable splitting theorem, and the relevant deformation problem therefore retains genuinely global bundle-theoretic information which cannot be reduced to line-bundle degree computations. To preserve an explicit high-degree reference geometry on such a surface, we pull back the Bor{\accent23 u}vka curve by holomorphic maps $f:\Sigma\to\cp^1$. By the Riemann--Hurwitz formula, every such map is necessarily ramified when $g>0$, and precisely at these forced branch points the pointwise non-degeneracy of the linearized horizontality equation deteriorates. Thus, the degeneration along the branch locus is not an artifact of the construction, but the analytic manifestation of the passage from the homogeneous genus-zero setting to arbitrary topology. The key new achievement of this work is therefore to solve the coupled linearized horizontality equation globally, with quantitative control across the branch locus. 
More precisely, for every closed Riemann surface $\Sigma$, we construct holomorphic covers
\begin{align*}
    f:\Sigma\rightarrow\cp^1
\end{align*}
of arbitrarily large degree $\ell$ whose geometry remains quantitatively controlled as $\ell\to+\infty$. We call such maps \emph{balanced covers} (see Section~3). At the natural scale $\ell^{-\frac12}$, their differentials are uniformly controlled from above and their energy has a uniform averaged lower bound. Their first two derivatives also satisfy a joint polynomial lower bound. In particular, all branch points are simple. Composing with the twistor lift $\tilde\varphi_n:\cp^1\to\mathscr Z_n$ of a Bor{\accent23 u}vka sphere gives the holomorphic horizontal curves
\begin{align*}
    \psi:=\tilde\varphi_n\circ f,
\end{align*}
which we call \emph{balanced Bor{\accent23 u}vka covers}.

To convert this finite-order non-degeneracy into surjectivity of the linearized horizontality operator, we develop two further quantitative ingredients. The first is a theory of localized holomorphic sections for high-degree line bundles. Building on Tian's holomorphic peak-section construction \cite[Lemma~1.2]{Tian1990} and the approximately holomorphic methods of Donaldson and Auroux \cite{Donaldson1996,Auroux1997}, in Section~2 we prove a holomorphic peak-section theorem under averaged curvature positivity and uniform natural-scale curvature bounds. The resulting holomorphic sections realize arbitrary prescribed finite jets while retaining quantitative Gaussian localization. The same construction is used in Section~3 to build quantitative holomorphic pencils and, after a transversality argument, the balanced covers described above; in particular, their branching is controlled by a global polynomial non-degeneracy estimate for the Wronskian and its first derivative. The second ingredient is a polynomially bounded holomorphic right inverse for the linearized horizontality operator (Lemma~\ref{Lemma: polynomial holomorphic right inverse}). Along the Bor{\accent23 u}vka curve, we construct a homogeneous filtration whose quotient operators are explicit first-order operators on jet bundles. After pullback by $f$, these quotient equations can be solved recursively away from the branch locus. The apparent poles at each branch point are removed by finitely many exact conditions on the jets of a scalar section. A global $\bar\partial$-correction preserves these jets, and holomorphic lifting through the filtration yields the required right inverse. The coercivity estimates of Section~2 and the balanced-cover bounds ensure that its norm grows at most polynomially with $\ell$. Its existence proves surjectivity, while duality gives the corresponding projected spectral gap. The holomorphic implicit function theorem and Riemann--Roch then yield
\begin{align*}
    \dim_{\mathbb C}X_{\Sigma,d_\ell}=2d_\ell+n^2(1-g),
\end{align*}
as asserted in Theorem~\ref{Theorem: main statement 1}.
\subsection{Related open problems}
We conclude the introduction by discussing two open problems naturally suggested by our results.

The first concerns the global structure of the space of superminimal
immersions. By Proposition~\ref{Proposition: correspondence holo+hor with supermin},
every linearly full branched superminimal immersion has a unique
horizontal holomorphic lift whose projection is either that immersion
or its antipodal map. Since our parametrized maps are not identified
with their antipodal maps, the resulting bijection is
\begin{align*}
    \operatorname{HH}_d^f(\Sigma,\mathscr Z_n)\times\{+1,-1\}
    &\rightarrow\operatorname{SM}_d^f(\Sigma,\s^{2n}),\\
    (\psi,\varepsilon)&\mapsto\varepsilon\,\pi\circ\psi.
\end{align*}
Here $\operatorname{HH}_d^f(\Sigma,\mathscr Z_n)$ denotes the space of
linearly full holomorphic horizontal maps of degree $d$.
We endow $\operatorname{SM}_d^f(\Sigma,\s^{2n})$ with the complex-algebraic
structure transported from this disjoint union of two copies of
$\operatorname{HH}_d^f(\Sigma,\mathscr Z_n)$.
The two copies have identical component dimensions, so the dimension
conjecture below is equivalently formulated on either side. 
The dimension appearing in Theorem~\ref{Theorem: main statement 1} suggests that the families constructed there should not be exceptional components, but should instead reflect the expected dimension of the entire linearly full locus, at least for $d$ sufficiently large. This leads us to the following higher-genus version of the Bolton--Woodward conjecture (see \cite{BoltonWoodward1992,BoltonWoodward1993}).
\begin{conjecture}\label{Conjecture: higher genus Bolton Woodward}
Let $n\geq2$ and let $\Sigma$ be a closed Riemann surface of genus $g$. Then, for every sufficiently large degree $d$, the space
\begin{align*}
    \operatorname{SM}_d^f(\Sigma,\s^{2n})
\end{align*}
is of pure complex dimension
\begin{align*}
    2d+n^2(1-g).
\end{align*}
\end{conjecture}
The restriction to linearly full maps is essential. Already for $n=2$, Chi--Mo showed that the full moduli space contains a totally geodesic component of dimension $2d-g+4$, whereas the expected dimension of its non-totally-geodesic part is $2d-4g+4$ (see \cite{ChiMo1996}). In $\s^4$, a superminimal surface is linearly full precisely when it is non-totally-geodesic, while in higher dimensions a non-totally-geodesic superminimal surface may still be contained in a proper even-dimensional subsphere. Thus, linear fullness is the natural higher-dimensional replacement of the condition appearing in the conjecture of Chi--Mo. Indeed, Conjecture~\ref{Conjecture: higher genus Bolton Woodward} recovers their conjecture that the non-totally-geodesic locus should have pure dimension $2d-4g+4$ when $n=2$ (see \cite{ChiMo1996}). At the opposite end, when $g=0$, it reduces, through the twistor correspondence, to the conjecture of Bolton--Woodward, proved by Fern\'andez, according to which the space of linearly full harmonic two-spheres of degree $d$ in $\s^{2n}$ has dimension $2d+n^2$ (see \cite{Fernandez2012}). Theorem~\ref{Theorem: main statement 1} provides evidence in the remaining regime: for arbitrary $n$ and $g$, and along an unbounded sequence of degrees, it produces smooth families having exactly the dimension predicted above. The problem left open is therefore, in essence, to rule out the appearance in high degree of linearly full components of excess dimension.

A second question concerns the curvature of the surfaces produced by Theorem~\ref{Theorem: main statement 1}. This is closely related to Problem~101 in Yau's celebrated list of open problems, which asks for the existence of closed minimal surfaces with strictly negative Gaussian curvature in round spheres of dimension at least four (see \cite{Yau1982}). Recent work of Ancona--Labourie--Roig-Sanchis--Toulisse, building on Song's asymptotic construction, gives such surfaces in every sufficiently high-dimensional round sphere (see \cite{AnconaLabourieRoigSanchisToulisse2025}). It remains particularly interesting to understand whether negatively curved minimal surfaces can be found in each fixed ambient dimension.

Our construction suggests the following more precise conjecture.
\begin{conjecture}\label{Conjecture: negative curvature}
    Let $n\geq3$. For every sufficiently large $g$ and every closed
    Riemann surface $\Sigma$ of genus $g$, the families
    $X_{\Sigma,d_\ell}$ in Theorem~\ref{Theorem: main statement 1}
    can be chosen so that, for some $\ell=\ell(\Sigma)$,
    there exists $\varphi\in X_{\Sigma,d_\ell}$ satisfying
    \begin{align*}
        K_\varphi<0 \qquad\mbox{ on }\,\Sigma\smallsetminus\mathcal B_\varphi.
    \end{align*}
\end{conjecture}
The genus and degree must grow comparably in the above conjecture. Indeed, the dimension formula in Theorem~\ref{Theorem: main statement 1} and the non-emptiness of $X_{\Sigma,d_\ell}$ give
\begin{align*}
    d_\ell\geq\frac{n^2}{2}(g-1).
\end{align*}
On the other hand, nonpositive Gaussian curvature away from the branch locus forces $d_\ell\leq C_n(g-1)$, by Proposition~\ref{Proposition: degree bound under nonpositive curvature}.

There are several reasons to investigate this regime of simultaneous growth of genus and degree. Our first motivation comes from Song's probabilistic construction (see \cite{SongRandomHarmonicMaps}). Consider the thrice-punctured sphere, whose Teichm\"uller space is trivial and whose fundamental group is the free group $F_2$ with two generators. A representation
\begin{align*}
    \rho:F_2\rightarrow U(N)
\end{align*}
is determined by a pair of unitary matrices, so that
$\operatorname{Hom}(F_2,U(N))\cong U(N)^2$ carries a natural probability measure induced by the product Haar measure and has real dimension $2N^2$. To such a representation, Song associates an equivariant harmonic map into the corresponding unit sphere by minimizing the renormalized energy. As $N\to+\infty$, the geometry of the harmonic representative associated with a random representation concentrates, with probability tending to one, around a rescaled hyperbolic metric. More precisely, the convergence is locally smooth on compact subsets away from the cusp regions; no analogous uniform statement holds near the images of the cusps. The closed branched minimal surfaces obtained from this construction nevertheless converge globally in the Benjamini--Schramm sense to a rescaled hyperbolic plane, while their Gaussian curvatures converge to the corresponding negative constant in an averaged sense. 

Although Song's asymptotic regime involves increasing target dimension, it suggests a broader principle which is one of the motivations for the present work: in sufficiently large natural families of minimal surfaces, a surface chosen generically or at random should exhibit an increasingly hyperbolic geometry as the number of available parameters grows. The families constructed in Theorem~\ref{Theorem: main statement 1} provide a particularly interesting setting in which to test this heuristic without allowing the ambient sphere to vary. It is therefore natural to ask whether the geometry of typical
members of $X_{\Sigma,d_\ell}$ becomes increasingly hyperbolic
along sequences for which $d_\ell\asymp g\to+\infty$ and
$2d_\ell+n^2(1-g)\to+\infty$, and Conjecture~\ref{Conjecture: negative curvature} may be viewed as a first deterministic manifestation of this principle.

A second, and more directly geometric, indication comes from the work of Mohsen on negatively curved complete intersections (see \cite{Mohsen2022}). Given a fixed complex projective manifold of complex dimension at least three, equipped with a fixed Hermitian metric and an ample line bundle, Mohsen showed that for every sufficiently large degree there exists a complete-intersection curve whose induced Gaussian curvature is strictly negative. Thus, in his setting the ambient geometry remains fixed throughout: it is precisely the passage to high degree which creates enough local flexibility to impose a pointwise curvature inequality. The mechanism is based on Donaldson--Auroux asymptotic techniques and on the avoidance of suitable algebraic subsets in finite jet spaces.

The parallel with our setting is particularly suggestive. In addition to the growth of $\dim_{\mathbb C}X_{\Sigma,d_\ell}$, the analytic tools developed here provide localized holomorphic sections with prescribed finite jets, while the deformation argument gives quantitative solvability of the linearized horizontality equation. To use these ingredients for curvature control, one would need to construct horizontal deformations with suitable prescribed second-order geometry and control the passage to genuine horizontal maps. Since Gaussian curvature is determined by the second-order geometry of the immersion, this suggests studying whether the jets for which $K_\varphi<0$ fails can be avoided within suitably chosen regular families. As in Mohsen's construction, increasing the degree may provide the local flexibility needed to impose a pointwise curvature inequality inside a fixed ambient space.

The restriction $n\geq3$ in Conjecture~\ref{Conjecture: negative curvature} is sharp, since a nonconstant branched superminimal immersion of a closed Riemann surface into $\s^4$ cannot have nonpositive Gaussian curvature throughout its regular locus. In particular, allowing branch points does not remove the obstruction in $\s^4$ (see Remark~\ref{Remark: nonpositive curvature obstruction in S4}).
\begin{remark}\label{Remark: nonpositive curvature obstruction in S4}
    In $\s^4$, nonpositive Gaussian curvature is impossible even when branching is allowed. Indeed, suppose that $\varphi:\Sigma\to\s^4$ is a nonconstant branched superminimal immersion of a closed Riemann surface and that $K_\varphi\leq0$ on its regular locus. Since $K_{\varphi}\leq0$, the Gauss equation implies that $A_{\varphi}$ is nowhere zero on the regular locus. Superminimality means that the curvature ellipse is a circle, so we may choose a local positively oriented orthonormal tangent frame $(e_1,e_2)$ and a local orthonormal normal frame $(n_1,n_2)$ such that
    \begin{align*}
        A_{\varphi}(e_1,e_1)=a n_1,\qquad
        A_{\varphi}(e_1,e_2)=a n_2,\qquad
        A_{\varphi}(e_2,e_2)=-a n_1,
    \end{align*}
    where $a>0$ and the Gauss equation gives $1-K_{\varphi}=2a^2$.
    Define the tangent and normal local connection 1-forms $\omega$ and $\eta$ by
    \begin{align*}
        \nabla^{g_{\varphi}}e_1=\omega\otimes e_2,
        \qquad
        \nabla^\perp n_1=\eta\otimes n_2.
    \end{align*}
    The Codazzi equations yield
    \begin{align*}
        (\eta-2\omega)(e_1)=-e_2(\log a),
        \qquad
        (\eta-2\omega)(e_2)=e_1(\log a),
    \end{align*}
    or equivalently $\eta-2\omega=*d\log a$, where $*$ is the
    Hodge star of $g_{\varphi}$. On the other hand, the Gauss and Ricci equations give
    \begin{align*}
        d\omega=-K_{\varphi}\,d\vol_{g_{\varphi}},
        \qquad
        d\eta=-2a^2\,d\vol_{g_{\varphi}},
    \end{align*}
    where $d\vol_{g_{\varphi}}$ is the oriented area form of $g_{\varphi}$. Taking the exterior derivative of the Codazzi identity and using $d\hspace{-0.5mm}\ast\hspace{-0.5mm}df=(\Delta_{g_{\varphi}} f)\,d\vol_{g_{\varphi}}$, with
    $\Delta_{g_{\varphi}}=\operatorname{div}_{g_{\varphi}}\nabla_{g_{\varphi}}$, we obtain
    \begin{align*}
        \Delta_{g_{\varphi}}\log a=2K_{\varphi}-2a^2=3K_{\varphi}-1.
    \end{align*}
    Consequently, the function
    \begin{align*}
        v:=-\log(1-K_\varphi)\leq0
    \end{align*}
    satisfies
    \begin{align*}
        \Delta_{g_\varphi}v=2-6K_\varphi>0
        \qquad\mbox{on }\,\Sigma\smallsetminus\mathcal B_\varphi.
    \end{align*}
    Subharmonicity is conformally invariant in dimension two.
    Since $v$ is bounded above, it extends across the isolated branch points as a subharmonic function on $\Sigma$, by assigning the upper-limit value at each branch point. This extension is not identically $-\infty$. Compactness and the maximum principle therefore force it to be constant, contradicting the strict inequality above.
\end{remark}
Starting with $\s^6$, the situation appears considerably more flexible, both from the point of view of curvature and of embeddedness. There is no analogous reason forcing the second fundamental form of an unbranched superminimal surface to vanish somewhere, and we expect that the same high-degree local flexibility underlying Conjecture~\ref{Conjecture: negative curvature} could also be used to eliminate branch points. At the same time, embeddedness should be favoured by a simple dimension count. Indeed, self-intersections of a map $\varphi:\Sigma\to\s^{2n}$ are detected by the two-point evaluation map
\begin{align*}
(x,y)\mapsto\big(\varphi(x),\varphi(y)\big),
\qquad x\neq y,
\end{align*}
through its intersection with the diagonal in $\s^{2n}\times\s^{2n}$. Since the diagonal has real codimension $2n$, whereas $\Sigma\times\Sigma$ has real dimension $4$, transversality would rule out double points whenever $n\geq3$. Thus, if the horizontal deformation spaces constructed here enjoy the corresponding two-point transversality property, one should expect an unbranched generic member to be embedded. The difficulty is that the available deformations are constrained simultaneously by holomorphicity and horizontality, so ordinary general-position arguments cannot be applied directly; to our knowledge, the corresponding generic embeddedness statement is not presently known even for the classical moduli spaces of linearly full minimal two-spheres in $\s^{2n}$, $n\geq3$.

We therefore expect a sufficiently strong refinement of the large-degree deformation theory to allow one simultaneously to eliminate branch points, impose negative Gaussian curvature, and achieve embeddedness. Such a strengthening of Conjecture~\ref{Conjecture: negative curvature} would produce embedded negatively curved minimal surfaces in every $\s^{2n}$, $n\geq3$. Since the equatorial inclusion $\s^{2n}\hookrightarrow\s^{2n+1}$ is totally geodesic, this would in turn yield embedded negatively curved minimal surfaces in every round sphere of dimension at least six. In this way, the large-degree twistor deformation theory developed here may provide a new route toward many of the remaining finite-dimensional cases of Yau's Problem~101, even within the embedded class.
\section*{Acknowledgements}
I am especially grateful to Antoine Song for many stimulating discussions and for drawing my attention to the notion of superminimal surfaces.

I would also like to thank Gavin Ball, Renato Ghini Bettiol, Mark Haskins, Yang Li, Jesse Madnick, and Lorenzo Sarnataro for their interesting comments related to this work. 

Finally, I am grateful to 
Aditya Kumar, Antoine Song, Daniel Stern, and Xin Zhou for their careful reading of the final manuscript and their suggestions for improvement.
\section{Holomorphic peak sections of high-degree line bundles}
Throughout the present paper, $(\Sigma,j)$ denotes a closed Riemann surface
of genus $g\in\mathbb N$. We fix a constant-curvature Riemannian metric
$g_0$ compatible with $j$, normalized so that
\begin{align*}
    K_{g_0}=\begin{cases}
        1,  & g=0,\\
        0,  & g=1,\\
        -1, & g\geq2,
    \end{cases}
\end{align*}
and, when $g=1$, so that $\operatorname{vol}_{g_0}(\Sigma)=1$.
We denote by
\begin{align*}
    \omega_0:=g_0(j\,\cdot\,,\,\cdot\,)
\end{align*}
the K\"ahler form of $(\Sigma,j,g_0)$.
If $(L,h)\to\Sigma$ is a holomorphic Hermitian line bundle, we denote by 
\begin{align*} 
    \deg(L):= \big\langle c_1(L),[\Sigma]\big\rangle 
\end{align*} 
its degree, by $\nabla=\nabla_h$ its Chern connection, and by 
\begin{align*} 
    \bar\partial_L: \Gamma(L)\rightarrow\Omega^{0,1}(L)
\end{align*}
the Dolbeault operator induced by its holomorphic structure. We write $\bar\partial_{h}^*$ for the formal adjoint of $\bar\partial_L$ with respect to $g_0$ and $h$. When no confusion can arise, the subscript $h$ will be omitted. All pointwise norms, covariant derivatives, $L^p$-norms, balls, distances, and volume forms are computed using $g_0$ and $h$, unless otherwise specified. 
Inequalities between real $(1,1)$-forms are always understood pointwise.
%
%
For $x\in\Sigma$ and $\ell\in\mathbb N$, we denote by
\begin{align*} 
    J_{x,\ell}(L) := \bigoplus_{q=0}^{\ell} \operatorname{Sym}^{q} \big(T_x^{*\,1,0}\Sigma\big)\otimes L_x 
\end{align*} 
the space of holomorphic $\ell$-jets of sections of $L$ at $x$. If a section $\sigma$ is holomorphic in a neighbourhood of $x$, its $\ell$-jet is denoted by 
\begin{align*} 
    j_x^\ell(\sigma)\in J_{x,\ell}(L). 
\end{align*} 
The metrics $g_0$ and $h$ induce a Hermitian norm $\lvert\,\cdot\,\rvert_{g_0,h}$ on $J_{x,\ell}(L)$.

\bigskip
The purpose of this section is to establish the analytic tool used
later in the deformation theory of balanced Bor{\accent23 u}vka
covers. Starting from a polynomial spectral gap for the Dolbeault
operator, we prove a weighted $\bar\partial$-solvability result with
prescribed jet vanishing and use it to construct global holomorphic
sections which realize arbitrary finite jets while remaining localized near a prescribed point.
\subsection{Weighted \texorpdfstring{$\bar\partial$}{del-bar}-solvability with prescribed jet vanishing}
The following lemma provides the correction mechanism underlying the peak-section construction. The coercivity estimate determines the relevant spectral scale, the singular weight ensures that the correction preserves a prescribed finite jet at $x$, and the exponential weight controls its localization away from $x$.

Fix $r_0>0$ sufficiently small. For each $x\in\Sigma$, choose a
uniformly controlled holomorphic coordinate $z:U_x\to\{z\in\c:|z|<2r_0\}$ centred at $x$, and set
\begin{align*}
    V_x:=\{y\in U_x:|z(y)|<r_0\}.
\end{align*}
Fix smooth functions $w_x$ equal to $|z|^2$ on $V_x$ and positive
on $\Sigma\smallsetminus\{x\}$, with a uniform upper bound on $\Sigma$ and a uniform positive lower bound on $\Sigma\smallsetminus V_x$.
\begin{lemma}\label{Lemma: del bar solution with weighted norm}
    Let $(L,h)\to\Sigma$ be a holomorphic Hermitian line bundle of
    degree $d$. Assume that there are positive numbers $A_d\to+\infty$
    and smooth functions $b_d\geq0$ such that, for all sufficiently
    large $d$,
    \begin{align}\label{Equation: 6}
        \|\beta\|_{L^2(h)}
        \leq\frac{C}{\sqrt{A_d}}
        \|\bar\partial_h^*\beta\|_{L^2(h)}
        \qquad\forall\,\beta\in\Omega^{0,1}(L),
    \end{align}
    and
    \begin{align}\label{Equation: scalar coercivity for weighted correction}
        \Lambda_{\omega_0}(iF_h)&\geq b_d-C_0,\\
        \int_\Sigma\bigl(|du|_{g_0}^2+b_d|u|^2\bigr)\,d\vol_{g_0}
        &\geq\lambda A_d\int_\Sigma|u|^2\,d\vol_{g_0}
        \qquad\forall\,u\in C^\infty(\Sigma,\r),
        \nonumber
    \end{align}
    where $C,\lambda>0$ and $C_0\geq0$ are independent of $d$.
    Fix $\,\ell\geq1$. There exist $d_0\in\n$ and $\delta_0,K>0$
    such that for every $d\geq d_0$, $x\in\Sigma$, $0<\delta\leq\delta_0$, and $\alpha\in\Omega^{0,1}(L)$ satisfying
    \begin{align}\label{Equation: 7}
        \int_\Sigma|\alpha|_{g_0,h}^2w_x^{-\ell-1}\,d\vol_{g_0}<+\infty
    \end{align}
    we can find a solution $\,\xi\in\Gamma(L)$ of $\bar\partial\xi=\alpha$ with
    \begin{align}\label{Equation: weighted correction local curvature gain}
        \int_\Sigma|\xi|_h^2w_x^{-\ell-1}
        e^{\delta A_d\dist_{g_0}^2(x,\,\cdot\,)}\,d\vol_{g_0}&\leq K\int_\Sigma\frac{|\alpha|_{g_0,h}^2}{A_d+b_d}w_x^{-\ell-1}
        e^{\delta A_d\dist_{g_0}^2(x,\,\cdot\,)}\,d\vol_{g_0}\\
        &\leq\frac{K}{A_d}\int_\Sigma
        |\alpha|_{g_0,h}^2w_x^{-\ell-1}
        e^{\delta A_d\dist_{g_0}^2(x,\,\cdot\,)}\,d\vol_{g_0}.
        \nonumber
    \end{align}
    If $\alpha$ vanishes near $x$, then $\xi$ is holomorphic near $x$
    and $j_x^\ell(\xi)=0$.
    The constants $K$ and $\delta_0$ depend only on the fixed
    background geometry, $\ell$, $C$, $\lambda$, and $C_0$.
    The threshold $d_0$ may additionally depend on the sequence
    $A_d$ and on the threshold from which the hypotheses hold.
    All constants are uniform in $x$.
\end{lemma}
\begin{proof}
Fix $x\in\Sigma$, set $q:=\ell+1$, and write $U:=U_x$, $V:=V_x$, and $w:=w_x$. The proof proceeds in four main steps: establishing a coercivity estimate for $(0,1)$-forms with values in $L$ twisted by a fixed Hermitian line bundle, deriving the corresponding exponentially weighted estimate, solving the weighted $\bar\partial$-equation after twisting by a point-divisor bundle, and finally proving the vanishing of the prescribed jet.

\medskip
\noindent
\textbf{Step 1: coercivity estimate for twisted bundles}. 
Let $(M,h_M) \to \Sigma$ be a fixed Hermitian holomorphic line bundle, and define
\begin{align*}
     E := L \otimes M, \qquad h_{E} := h \otimes h_M,
\end{align*}
We first claim that for $d$ sufficiently large, there exists a constant $C_1 > 0$, independent of $d$, such that
\begin{equation}\label{Equation: fixed twist adjoint coercivity}
    \|\beta\|_{L^2(h_E)} \le \frac{C_1}{\sqrt{A_d}} \big\|\bar\partial_{h_E}^*\beta\big\|_{L^2(h_E)} \qquad\forall \beta \in \Omega^{0,1}(E).
\end{equation}
Choose a finite open cover $\{U_j\}_{j=1}^N$ of $\Sigma$, smooth unitary frames $m_j$ for $M$ over $U_j$, and a smooth partition of unity $\{\chi_j^2\}_{j=1}^N$ subordinate to $\{U_j\}$ such that $\sum_{j=1}^N \chi_j^2 = 1$ on $\Sigma$. Let $\beta \in \Omega^{0,1}(E)$. On each $U_j$, we can write $\beta = b_j \otimes m_j$ for some $b_j \in \Omega^{0,1}(L|_{U_j})$. Since $m_j$ is unitary, $|b_j|_{g_0,h} = |\beta|_{g_0,h_E}$ on $U_j$. Hence,
\begin{equation}\label{Equation: beta partition norm}
    \|\beta\|_{L^2(h_E)}^2 = \sum_{j=1}^N \|\chi_j b_j\|_{L^2(h)}^2.
\end{equation}
Extending each $\chi_j b_j$ by zero to a smooth $L$-valued $(0,1)$-form on $\Sigma$ and applying \eqref{Equation: 6}, we obtain
\begin{equation}\label{Equation: apply L coercivity locally}
    \|\beta\|_{L^2(h_E)}^2
    \leq\frac{C^2}{A_d}
    \sum_{j=1}^N
    \|\bar\partial_h^*(\chi_jb_j)\|_{L^2(h)}^2.
\end{equation}
We now relate the right-hand side to ${\bar\partial}_{h_E}^*\beta$. Since $\nabla_{h_E}=\nabla_{h}\otimes 1 + 1 \otimes \nabla_{h_M}$ and $m_j$ is smooth, the local formula for the formal adjoint yields
\begin{align}\label{Equation: 8}
    {\bar\partial}_{h_E}^*(b_j \otimes m_j) = ({\bar\partial}_h^* b_j) \otimes m_j + T_j b_j,
\end{align}
where $T_j$ is a smooth zero-order operator depending only on $(M,h_M)$ and the frame $m_j$. Furthermore,
\begin{align}\label{Equation: 9}
    \bar\partial_h^*(\chi_j b_j) = \chi_j \bar\partial_h^*b_j + S_j b_j,
\end{align}
where $S_j$ is a smooth zero-order operator depending only on $\chi_j$ and $g_0$. Combining \eqref{Equation: 8} and \eqref{Equation: 9} yields the pointwise bound
\begin{align*}
    \big|\bar\partial_h^*(\chi_j b_j)\big|_h\le C_2\left(\chi_j\big|\bar\partial_{h_E}^*\beta\big|_{h_E} + |\beta|_{g_0,h_E}\right),
\end{align*}
with $C_2 > 0$ independent of $d$. Integrating over $\Sigma$ and summing over $j$ gives
\begin{equation}\label{Equation: adjoint comparison}
    \sum_{j=1}^N\big\|{\bar\partial}_h^*(\chi_j b_j)\big\|_{L^2(h)}^2 \le C_3 \left(\|{\bar\partial}_{h_E}^*\beta\big\|_{L^2(h_E)}^2 + \|\beta\|_{L^2(h_E)}^2 \right)
\end{equation}
with $C_3 > 0$ independent of $d$. Substituting \eqref{Equation: adjoint comparison} into \eqref{Equation: apply L coercivity locally}, we find
\begin{align}\label{Equation: 16}
    \|\beta\|_{L^2(h_E)}^2
    \leq\frac{C^2C_3}{A_d}
    \left(
        \|\bar\partial_{h_E}^*\beta\|_{L^2(h_E)}^2
        +\|\beta\|_{L^2(h_E)}^2
    \right).
\end{align}
Since $A_d\to+\infty$, we may take $d$ sufficiently large that
$C^2C_3A_d^{-1}\leq1/2$. Absorbing the last term into the
left-hand side and taking square roots gives
\eqref{Equation: fixed twist adjoint coercivity} with
$C_1:=C\sqrt{2C_3}$.

\medskip
\noindent
\textbf{Step 2: coercivity estimate for twisted bundles with exponential weight}.
For any $y\in\Sigma$, let
\begin{align*}
    r_y^2:=\dist_{g_0}^2(y,\,\cdot\,).
\end{align*}
Since $(\Sigma,g_0)$ is compact, the functions $\{r_y^2\}_{y\in\Sigma}$ are uniformly semiconcave. More precisely, there exists a constant $C_{\mathrm{sc}}>0$, depending only on the geometry of $(\Sigma,g_0)$, such that
\begin{align*}
    \nabla^2r_y^2\le C_{\mathrm{sc}}g_0 \qquad\forall\,x\in\Sigma
\end{align*}
in the barrier, and hence distributional, sense. By Greene--Wu smoothing \cite[Proposition~2.2]{GreeneWu1979}, for every $d\in\n$ there exists a function $\rho_{x,d}\in C^\infty(\Sigma)$
such that
\begin{align}\label{Equation: regularized distance uniform approximation}
    \left|\rho_{x,d}-\dist_{g_0}^2(x,\,\cdot\,)\right|\le A_d^{-1} \qquad\mbox{ on }\,\Sigma,
\end{align}
and, after increasing $C_{\mathrm{sc}}$ if necessary,
\begin{align}\label{Equation: regularized distance upper Hessian}
    \nabla^2\rho_{x,d}\le C_{\mathrm{sc}}g_0 \qquad\mbox{ on }\,\Sigma.
\end{align}
Moreover, the approximation can be chosen so that
\begin{align*}
    \rho_{x,d}=\dist_{g_0}^2(x,\,\cdot\,)
\end{align*}
on a fixed smaller neighbourhood of $x$ contained in $U$. Fix $\delta>0$, to be chosen sufficiently small independently of $d$, and define
\begin{align*}
    \Phi_d:=-\frac{\delta}{2}A_d\rho_{x,d}.
\end{align*}
By \eqref{Equation: regularized distance upper Hessian}, we have
\begin{align}\label{Equation: lower Hessian Phi}
    \nabla^2\Phi_d\ge -\frac{\delta C_{\mathrm{sc}}}{2}A_dg_0.
\end{align}
We establish the weighted coercivity estimate by changing the Hermitian
metric. Set
\begin{align*}
    t_d:=\delta A_d\rho_{x,d}, \qquad h_{E,d}:=e^{t_d}h_E.
\end{align*}
Throughout the following computation, $\nabla^{h_{E,d}}$ denotes the connection induced on $E$-valued forms by the Chern connection of $h_{E,d}$ and the Levi--Civita connection of $g_0$. Since $h_E=h\otimes h_M$, the curvature of the new metric satisfies
\begin{align*}
    F_{h_{E,d}}=F_h+F_{h_M}-\partial\bar\partial t_d.
\end{align*}
Moreover, with the convention $\Delta_{g_0}=\operatorname{div}_{g_0}\nabla_{g_0}$, we have $\Lambda_{\omega_0}(i\partial\bar\partial f)=\frac12\Delta_{g_0}f$. Taking the trace in \eqref{Equation: regularized distance upper Hessian} therefore gives
\begin{align*}
    \Lambda_{\omega_0}(i\partial\bar\partial t_d)=\frac{\delta A_d}{2}\Delta_{g_0}\rho_{x,d}\leq C_{\mathrm{sc}}\delta A_d.
\end{align*}
Choose $C_M\geq0$ such that
\begin{align*}
    C_M\geq C_0+\big\|\Lambda_{\omega_0}(iF_{h_M})+K_{g_0}\big\|_{L^\infty}.
\end{align*}
This constant can be chosen uniformly for the fixed compact family of twists equipped with the smoothly varying Hermitian metrics used below. The curvature hypothesis on $h$ now yields
\begin{align*}
    \Lambda_{\omega_0}(iF_{h_{E,d}})+K_{g_0}&=\Lambda_{\omega_0}(iF_h)+\Lambda_{\omega_0}(iF_{h_M})+K_{g_0}-\Lambda_{\omega_0}(i\partial\bar\partial t_d)\\
    &\geq b_d-C_M-C_{\mathrm{sc}}\delta A_d.
\end{align*}
For $\gamma\in\Omega^{0,1}(E)$, define
\begin{align*}
    Q_{d}(\gamma):=\|\nabla^{h_{E,d}}\gamma\|_{L^2(h_{E,d})}^2+\int_\Sigma b_d|\gamma|_{g_0,h_{E,d}}^2\,d\vol_{g_0}.
\end{align*}
We claim that the scalar coercivity hypothesis implies
\begin{align*}
    Q_{d}(\gamma)\geq
    \lambda A_d\|\gamma\|_{L^2(h_{E,d})}^2.
\end{align*}
Indeed, for $\varepsilon>0$, set
\begin{align*}
    u_\varepsilon:=\bigl(|\gamma|_{g_0,h_{E,d}}^2+\varepsilon^2\bigr)^{\frac12}.
\end{align*}
Metric compatibility gives
\begin{align*}
    du_\varepsilon=\frac{\operatorname{Re}\langle\nabla^{h_{E,d}}\gamma,\gamma\rangle_{g_0,h_{E,d}}}{u_\varepsilon}, \qquad|du_\varepsilon|_{g_0}\leq|\nabla^{h_{E,d}}\gamma|_{g_0,h_{E,d}}.
\end{align*}
Applying the scalar hypothesis to the smooth real-valued function $u_\varepsilon$, we obtain
\begin{align*}
    \lambda A_d\int_\Sigma u_\varepsilon^2\,d\vol_{g_0}&\leq\int_\Sigma\bigl(|du_\varepsilon|_{g_0}^2+b_du_\varepsilon^2\bigr)\,d\vol_{g_0}\leq Q_d(\gamma)+\varepsilon^2\int_\Sigma b_d\,d\vol_{g_0}.
\end{align*}
Letting $\varepsilon\to0$ proves the claim.
On $E$-valued $(0,1)$-forms, the integrated
Bochner--Kodaira--Weitzenb\"ock identity reads
\begin{align*}
    2\|\bar\partial_{h_{E,d}}^*\gamma\|_{L^2(h_{E,d})}^2=\|\nabla^{h_{E,d}}\gamma\|_{L^2(h_{E,d})}^2+\int_\Sigma\bigl(\Lambda_{\omega_0}(iF_{h_{E,d}})+K_{g_0}\bigr)|\gamma|_{g_0,h_{E,d}}^2\,d\vol_{g_0}.
\end{align*}
Here $\bar\partial\gamma=0$ because $\Sigma$ has complex dimension one, and the term $K_{g_0}$ arises from the cotangent factor. Combining this identity with the curvature bound and applying the preceding claim to half of $Q_d(\gamma)$ gives
\begin{align}\label{Equation: weighted metric Bochner estimate}
    2\|\bar\partial_{h_{E,d}}^*\gamma\|_{L^2(h_{E,d})}^2
    &\geq Q_d(\gamma)
    -(C_M+C_{\mathrm{sc}}\delta A_d)
    \|\gamma\|_{L^2(h_{E,d})}^2\\
    &\geq\frac12Q_d(\gamma)
    +\left(\frac{\lambda}{2}A_d-C_M
    -C_{\mathrm{sc}}\delta A_d\right)
    \|\gamma\|_{L^2(h_{E,d})}^2.
    \nonumber
\end{align}
Choose $\delta_0>0$ such that $C_{\mathrm{sc}}\delta_0\leq\lambda/8$ and then, using $A_d\to+\infty$, choose $d_0$ such that $C_M\leq\lambda A_d/8$ for every $d\geq d_0$. For $d\geq d_0$ and $0<\delta\leq\delta_0$, it follows that
\begin{align*}
    \|\bar\partial_{h_{E,d}}^*\gamma\|_{L^2(h_{E,d})}^2\geq\frac14Q_d(\gamma)+\frac{\lambda}{8}A_d\|\gamma\|_{L^2(h_{E,d})}^2\geq\int_\Sigma\left(\frac14b_d+\frac{\lambda}{8}A_d\right)|\gamma|_{g_0,h_{E,d}}^2\,d\vol_{g_0}.
\end{align*}
Thus, setting $c_1:=\min\{1/4,\lambda/8\}$, we obtain
\begin{align}\label{Equation: weighted metric potential coercivity}
    \|\bar\partial_{h_{E,d}}^*\gamma\|_{L^2(h_{E,d})}^2
    \geq c_1\int_\Sigma(A_d+b_d)
    |\gamma|_{g_0,h_{E,d}}^2\,d\vol_{g_0}.
\end{align}
To express this estimate in terms of the original metric, we use
the transformation law for the adjoint. For every
$s\in\Gamma(E)$ and $\gamma\in\Omega^{0,1}(E)$, integration by
parts gives
\begin{align*}
    \langle\bar\partial s,\gamma\rangle_{L^2(h_{E,d})}=\langle\bar\partial s,e^{t_d}\gamma\rangle_{L^2(h_E)}=
    \langle s,\bar\partial_{h_E}^*(e^{t_d}\gamma)
    \rangle_{L^2(h_E)}=\langle s,e^{-t_d}\bar\partial_{h_E}^*(e^{t_d}\gamma)\rangle_{L^2(h_{E,d})}.
\end{align*}
Consequently,
\begin{align*}
    \bar\partial_{h_{E,d}}^*\gamma
    =e^{-t_d}\bar\partial_{h_E}^*(e^{t_d}\gamma),
\end{align*}
and, in particular,
\begin{align*}
    \bar\partial_{h_{E,d}}^*(e^{-t_d}\beta)
    =e^{-t_d}\bar\partial_{h_E}^*\beta
    \qquad\forall\,\beta\in\Omega^{0,1}(E).
\end{align*}
Since
\begin{align*}
    |e^{-t_d}\beta|_{g_0,h_{E,d}}^2
    =e^{-t_d}|\beta|_{g_0,h_E}^2,
\end{align*}
substituting $\gamma=e^{-t_d}\beta$ into
\eqref{Equation: weighted metric potential coercivity} yields
\begin{align*}
    \int_\Sigma e^{-t_d}
    |\bar\partial_{h_E}^*\beta|_{h_E}^2\,d\vol_{g_0}
    \geq
    c_1\int_\Sigma(A_d+b_d)e^{-t_d}
    |\beta|_{g_0,h_E}^2\,d\vol_{g_0}.
\end{align*}
Finally, using $b_d\geq0$, taking square roots, and recalling that
$\Phi_d=-t_d/2$, we conclude that
\begin{align}\label{Equation: fixed twist exponential weighted coercivity}
    \|e^{\Phi_d}\beta\|_{L^2(h_E)}
    \leq\frac{c_1^{-\frac12}}{\sqrt{A_d}}
    \|e^{\Phi_d}\bar\partial_{h_E}^*\beta\|_{L^2(h_E)}
    \qquad\forall\,\beta\in\Omega^{0,1}(E).
\end{align}
This is the required weighted coercivity estimate.

\medskip
\noindent
\textbf{Step 3: weighted $\bar\partial$-solvability with $L^2$-estimate}.
%
We apply the previous step to $M:=\mathcal O(-qx)$.
Let $m$ be the holomorphic frame of $M$ on $U$ whose image under
the natural inclusion $M\to\mathcal O_\Sigma$ is $z^q$.
On $\Sigma\smallsetminus\{x\}$, let $m_{\mathrm{out}}$ be the
frame whose image is $1$. Thus
$m=z^qm_{\mathrm{out}}$ on $U\smallsetminus\{x\}$.
Choose a smooth Hermitian metric $h_M$ for which $m$ is unit
on $V$ and $m_{\mathrm{out}}$ is unit outside
$\{y\in U:|z(y)|<3r_0/2\}$.
Interpolating the logarithms of these squared norms using a
fixed smooth cutoff on the annulus
$r_0<|z|<3r_0/2$ gives metrics whose connection coefficients
and curvature are uniformly controlled, independently of $x$.

Set $\widehat L:=L\otimes M$ and $\widehat h:=h\otimes h_M$,
and denote the natural inclusion $\widehat L\to L$ by $\iota$.
Choose a holomorphic frame $e$ for $L$ on $U$ and set
$\widehat e:=e\otimes m$. Then
$\iota(\widehat e)=z^qe$ and
$|\widehat e|_{\widehat h}^2=|e|_h^2$ on $V$.
Thus, for $\widehat\eta=f\widehat e$ on $V$, we have
\begin{align*}
     |\iota(\widehat{\eta})|_h^2 |z|^{-2q}= |z|^{2q} |f|^2 |e|_h^2|z|^{-2q} = |f|^2 |\widehat{e}|_{\widehat{h}}^2 = |\widehat{\eta}|_{\widehat{h}}^2 \qquad\mbox{ on } V.
\end{align*}
On $\Sigma\smallsetminus V$, the construction of $h_M$ gives
uniform positive upper and lower bounds for the norm of the
inclusion $M\to\mathcal O_\Sigma$. Since $w$ also has uniform
positive upper and lower bounds there, the preceding identity
gives a pointwise norm comparison on
$\Sigma\smallsetminus\{x\}$ with uniform constants.
Integrating, we obtain a constant $C_5>0$, independent of
$x$ and $d$, such that
\begin{equation}\label{Equation: section norm comparison}
    C_5^{-1}\|\widehat{\eta}\|_{L^2(\widehat{h})}^2 \le \int_\Sigma |\iota(\widehat{\eta})|_h^2w^{-q}\,d\vol_{g_0} \le C_5\|\widehat{\eta}\|_{L^2(\widehat{h})}^2.
\end{equation}
Obviously, an identical norm equivalence holds for every $\widehat{L}$-valued $(0,1)$-form on $\Sigma$, i.e.
\begin{equation}\label{Equation: form norm comparison}
    C_6^{-1}\|\widehat{\beta}\|_{L^2(\widehat{h})}^2 \le \int_\Sigma |\iota(\widehat{\beta})|_{g_0,h}^2w^{-q}\,d\vol_{g_0} \le C_6\|\widehat{\beta}\|_{L^2(\widehat{h})}^2 \qquad\forall\,\widehat{\beta}\in\Omega^{0,1}(\widehat{L}),
\end{equation}
with $C_6 > 0$ depending only on the fixed data in the statement. Now, let $\alpha \in \Omega^{0,1}(L)$ satisfy \eqref{Equation: 7}. Since $\iota$ is an isomorphism on $\Sigma \smallsetminus\{x\}$, there is a unique $\widehat{\alpha}\in\Gamma(\Sigma\smallsetminus\{x\})$ such that $\iota(\widehat{\alpha}) = \alpha$ on $\Sigma\smallsetminus\{x\}$. 
The pointwise comparison underlying \eqref{Equation: form norm comparison} gives
\begin{align*}
    |\widehat\alpha|_{g_0,\widehat h}^2\leq C_6|\alpha|_{g_0,h}^2w^{-q}\qquad\mbox{ on }\,\Sigma\smallsetminus\{x\}.
\end{align*}
The right-hand side is integrable by \eqref{Equation: 7}. Thus $\widehat\alpha$ is square-integrable on $\Sigma\smallsetminus\{x\}$ and defines an element of $L^2\Omega^{0,1}(\widehat L;\widehat h)$. For each fixed $d$, $x$, and $\delta$, the function $e^{\delta A_d\rho_{x,d}}$ is smooth and bounded on the compact surface $\Sigma$. Multiplying the preceding pointwise inequality by this function and integrating therefore yields
\begin{align}\label{Equation: alpha hat norm estimate}
    \big\|e^{\frac{\delta}{2}A_d\rho_{x,d}}\widehat\alpha\big\|_{L^2(\widehat h)}^2=\int_\Sigma
    |\widehat\alpha|_{g_0,\widehat h}^2
    e^{\delta A_d\rho_{x,d}}\,d\vol_{g_0}\leq C_6\int_\Sigma
    |\alpha|_{g_0,h}^2w^{-q}
    e^{\delta A_d\rho_{x,d}}\,d\vol_{g_0}
    <+\infty.
\end{align}
Applying \eqref{Equation: fixed twist exponential weighted coercivity} to $\widehat{L} = L(-qx)$, for $d$ sufficiently large depending only on the fixed data in the statement we have
\begin{equation}\label{Equation: Lhat adjoint estimate}
    \|\widehat{\beta}\|_{L^2\big(e^{-\delta  A_d\rho_{x,d}}\widehat{h}\big)} \le \frac{C_7}{\sqrt{A_d}}\|\bar\partial_{\widehat{h}}^*\widehat{\beta}\|_{L^2\big(e^{-\delta A_d\rho_{x,d}}\widehat{h}\big)} \qquad\forall\,\widehat{\beta} \in \Omega^{0,1}(\widehat{L})
\end{equation}
where $C_7>0$ is a constant depending only on the fixed data in the statement. We use this to solve the equation $\bar\partial_{\widehat{L}}\widehat{\xi} = \widehat{\alpha}$.  Define a conjugate-linear functional $f$ on
\begin{align*}
    \mathcal D:=
    \bar\partial_{\widehat h}^*
    \bigl(\Omega^{0,1}(\widehat L)\bigr)
    \subset L^2\bigl(\widehat L;
          e^{-\delta A_d\rho_{x,d}}\widehat h\bigr)
\end{align*}
by
\begin{align*}
    f\bigl(\bar\partial_{\widehat h}^*\widehat\beta\bigr)
    :=\langle\widehat\alpha,\widehat\beta\rangle_{L^2(\widehat h)}
    \qquad\forall\,\widehat\beta\in\Omega^{0,1}(\widehat L).
\end{align*}
Here the Hermitian inner product is linear in its first argument.
In particular, the adjoint is initially applied only to smooth test
forms; it is not applied to an arbitrary $L^2$-form.
This is well-defined because $\bar\partial_{\widehat{h}}^*\widehat{\beta} = 0$ implies $\widehat{\beta} = 0$ by \eqref{Equation: Lhat adjoint estimate}. Moreover, by the Cauchy--Schwarz inequality and \eqref{Equation: Lhat adjoint estimate}, we have
\begin{align*}
    \big|f\big(\bar\partial_{\widehat{h}}^*\beta\big)\big| \le \|e^{\frac{\delta}{2} A_d\rho_{x,d}}\widehat{\alpha}\|_{L^2(\widehat{h})}\|e^{-\frac{\delta}{2} A_d\rho_{x,d}}\beta\|_{L^2(\widehat{h})} \le \frac{C_7}{\sqrt{A_d}}\|e^{\frac{\delta}{2} A_d\rho_{x,d}}\widehat{\alpha}\|_{L^2(\widehat{h})} \|\bar\partial_{\widehat{h}}^*\beta\|_{L^2\big(e^{-\delta A_d\rho_{x,d}}\widehat{h}\big)},
\end{align*}
i.e. $f$ is bounded. Applying Hahn--Banach theorem to complex-linear functional $\bar f$ and then conjugating, Riesz representation theorem gives $\widehat{u} \in L^2\big(\widehat{L};e^{-\delta A_d\rho_{x,d}}\widehat{h}\big)$ such that
\begin{align*}
    \big\langle \widehat{u}, \bar\partial_{\widehat{h}}^*\widehat{\beta}\big\rangle_{L^2\big(e^{-\delta A_d\rho_{x,d}}\widehat{h}\big)} = \langle \widehat{\alpha},\widehat{\beta} \rangle_{L^2(\widehat{h})}
\end{align*}

for all smooth $\beta \in \Omega^{0,1}(\widehat{L})$, satisfying the bound
\begin{equation}\label{Equation: xihat estimate final}
    \|\widehat{u}\|_{L^2\big(e^{-\delta A_d\rho_{x,d}}\widehat{h}\big)} \le \frac{C_7}{\sqrt{A_d}} \|e^{\frac{\delta}{2} A_d\rho_{x,d}}\widehat{\alpha}\|_{L^2(\widehat{h})}.
\end{equation}
Let $\widehat{\xi}:=e^{-\delta A_d\rho_{x,d}}\widehat{u}$. Note that $\bar\partial_{\widehat{L}}\widehat{\xi} = \widehat{\alpha}$ holds weakly. Define $\xi := \iota(\widehat{\xi})$. Since $\iota$ is holomorphic, it commutes with $\bar\partial$, so $\bar\partial_L\xi = \iota(\bar\partial_{\widehat{L}}\widehat{\xi}) = \iota(\widehat{\alpha}) = \alpha$ weakly on $\Sigma$. Because $\alpha$ is smooth, standard elliptic regularity ensures that $\xi$ is smooth, i.e. $\xi\in\Gamma(L)$. Using \eqref{Equation: section norm comparison}, \eqref{Equation: xihat estimate final}, and \eqref{Equation: alpha hat norm estimate}, we obtain the required weighted estimate:
\begin{align*}
    \int_\Sigma |\xi|_{g_0,h}^2\,w^{-q}e^{\delta A_d\rho_{x,d}}\,d\vol_{g_0}&\le C_5 \|e^{\frac{\delta}{2} A_d\rho_{x,d}}\widehat{\xi}\|_{L^2(\widehat{h})}^2\\
    &=C_5 \|e^{-\frac{\delta}{2} A_d\rho_{x,d}}\widehat{u}\|_{L^2(\widehat{h})}^2\\
    &\le \frac{C_5C_7^2}{A_d} \|e^{\frac{\delta}{2} A_d\rho_{x,d}}\widehat{\alpha}\|_{L^2(\widehat{h})}^2\\
    &\le \frac{C_8}{A_d}\int_\Sigma |\alpha|_{g_0,h}^2\,w^{-q}e^{\delta A_d\rho_{x,d}}\,d\vol_{g_0},
\end{align*}
where $C_8:=C_5C_6C_7^2>0$ depends only on the fixed data in the statement. By \eqref{Equation: regularized distance uniform approximation}, the last inequality is equivalent to
\begin{align*}
    \int_\Sigma
    |\xi|_{g_0,h}^2w^{-q}e^{\delta A_d\dist_{g_0}^2(x,\,\cdot\,)}\,d\vol_{g_0}
    \le
    \frac{K}{A_d}
    \int_\Sigma
    |\alpha|_{g_0,h}^2w^{-q}
    e^{\delta A_d\dist_{g_0}^2(x,\,\cdot\,)}\,d\vol_{g_0}.
\end{align*}
for some $K>0$ depending only on the fixed data in the statement.
To prove \eqref{Equation: weighted correction local curvature gain},
apply \eqref{Equation: weighted metric potential coercivity} on
$\widehat L=L(-qx)$ with $\widehat h_t=e^{t_d}\widehat h$.
For every smooth test form $\gamma\in\Omega^{0,1}(\widehat L)$,
\begin{align*}
    |\langle\widehat\alpha,\gamma\rangle_{L^2(\widehat h_t)}|
    &\leq
    \left(\int_\Sigma
       \frac{|\widehat\alpha|_{g_0,\widehat h_t}^2}{A_d+b_d}
       \,d\vol_{g_0}\right)^{\frac12}
    \left(\int_\Sigma(A_d+b_d)|\gamma|_{g_0,\widehat h_t}^2
       \,d\vol_{g_0}\right)^{\frac12}\\
    &\leq c_1^{-\frac12}
    \left(\int_\Sigma
       \frac{|\widehat\alpha|_{g_0,\widehat h_t}^2}{A_d+b_d}
       \,d\vol_{g_0}\right)^{\frac12}
       \|\bar\partial_{\widehat h_t}^*\gamma\|_{L^2(\widehat h_t)}.
\end{align*}
Applying the same Hahn--Banach and Riesz argument, now with this metric,
gives a weak solution $\widehat\xi$ whose squared
$L^2(\widehat h_t)$-norm is bounded by $c_1^{-1}$ times the source
integral. The pointwise divisor norm comparisons and
\eqref{Equation: regularized distance uniform approximation} give
\eqref{Equation: weighted correction local curvature gain} for
$\xi=\iota(\widehat\xi)$. We take this solution from now on, which obviously satisfies the weaker estimate because $(A_d+b_d)^{-1}\leq A_d^{-1}$ and whose smoothness follows as above.

\medskip
\noindent
\textbf{Step 4: vanishing of the $\ell$-jet of the solution}.
Assume further that $\alpha$ vanishes identically in a neighbourhood of $x$. Consequently, $\widehat{\alpha}$ also vanishes near $x$, implying $\bar\partial_{\widehat{L}}\widehat{\xi} = 0$ weakly in this neighbourhood. Since $\widehat{\xi} \in L^2$, Weyl's lemma implies that $\widehat{\xi}$ is smooth and holomorphic near $x$. Writing $\widehat{\xi} = f \widehat{e}$ locally with $f$ holomorphic, we have $\xi = \iota(\widehat{\xi}) = z^q f e$. This explicitly shows that $\xi$ vanishes at $x$ to order at least $q = \ell + 1$. Therefore, the $\ell$-jet of $\xi$ at $x$ vanishes. The statement follows.
\end{proof}
\subsection{Holomorphic peak sections from averaged curvature positivity}
We now combine the preceding weighted solvability result with an averaged positivity assumption on the curvature of $L$. The resulting spectral gap allows us to correct a locally defined approximate peak section without changing its prescribed jet.
\begin{proposition}\label{Proposition: holomorphic peak sections}
    Let $(L,h)\to\Sigma$ be a holomorphic Hermitian line bundle of degree $d$. Let $\nabla$ be the Chern connection of $(L,h)$. Let $m\in\mathbb N$, $\ell\in\n\smallsetminus\{0\}$, and $c,r>0$. Let $\rho:\Sigma\to\big[0,\frac{1}{c}\big)$ be a non-negative smooth function such that $\|d\rho\|_{L^{\infty}(\Sigma)}\le\frac{\sqrt{d}}{c}$ and
    \begin{align}\label{Equation: assumption on rho}
        \fint_{B_{r\,d^{-\frac{1}{2(m+1)}}}(x)}\rho\,d\vol_{g_0}\ge c \qquad\forall\,x\in\Sigma.
    \end{align}
    Assume that
    \begin{align*}
        \frac{i}{2\pi}F_{\nabla}\ge \,d\rho\,\omega_0 \qquad\mbox{ on }\,\Sigma,
    \end{align*}
    and that there are fixed constants $\{B_k\}_{k\in\n}$ such that
    \begin{align}\label{Equation: natural scale curvature bounds}
        \|\nabla^kF_\nabla\|_{L^\infty(\Sigma)}\leq B_k d^{1+k/2}\qquad\forall\,k\in\n.
    \end{align}
    We can find $d_0\in\mathbb N$, depending only on $m$, $\ell$, $c$,
    $r$, the geometry of $(\Sigma,g_0)$, and finitely many of the constants $B_k$ determined by $\ell$, such that for every $d\geq d_0$, every $x\in\Sigma$ satisfying $\rho(x)\geq c$, and every $J\in J_{x,\ell}(L)$ there exists $\sigma_d=\sigma_d(x,J)\in H^0(L)$ such that $j_x^{\ell}(\sigma_d)=J$ and
    \begin{align}\label{Equation: Gaussian peak estimate}
            \lvert\nabla^j\sigma_d(y)\rvert_{g_0,h}\le Kd^{\,\frac{\ell+1+j}{2}}\lvert J\rvert_{g_0,h}\,e^{-\frac{\,d^{\frac{1}{m+1}}\,\operatorname{dist}_{g_0}^2(y,x)}{K}} \qquad\forall\,y\in\Sigma,
    \end{align}
    for every $j\in\n$, where $K>0$ is a positive constant depending only on $j$, $m$, $\ell$, $c$, $r$, the geometry of $(\Sigma,g_0)$, and finitely many of the constants $B_k$ determined by $j$ and $\ell$.
\end{proposition}
\begin{proof}
    We divide the proof into three steps.

    \medskip
    \noindent
    \textbf{Step 1: a polynomial $L^2$-spectral gap for
    $\bar\partial\bar\partial^*$}.
    We claim that there exist constants $\tilde C_1>0$ and $d_1\in\n$
    such that, for every $d\ge d_1$,
    \begin{align}\label{Equation: polynomial spectral gap for delbar}
        \|\bar\partial^*\beta\|_{L^2(\Sigma)}^2
        \ge
        \frac{d^{\frac{1}{m+1}}}{\tilde C_1}
        \|\beta\|_{L^2(\Sigma)}^2
        \qquad
        \forall\,\beta\in\Omega^{0,1}(L).
    \end{align}
    Set
    \begin{align*}
        r_d:=r\,d^{-\frac{1}{2(m+1)}}.
    \end{align*}
    For $d$ sufficiently large, let
    \begin{align*}
        \Sigma=\bigcup_{i=1}^{N_d}B_i,
        \qquad
        B_i:=B_{r_d}(p_i),
    \end{align*}
    be a finite cover whose multiplicity is bounded by a constant
    $N_0\in\n$ depending only on the geometry of $(\Sigma,g_0)$.
    For every $i=1,\ldots,N_d$, define
    \begin{align*}
        E_i:=
        \left\{
        y\in B_i:
        \rho(y)\ge\frac{c}{2}r_d^{2m}
        \right\}.
    \end{align*}
    By \eqref{Equation: assumption on rho} and the bound
    $0\le\rho\le c^{-1}$, we have
    \begin{align*}
        c\vol_{g_0}(B_i)\le\int_{B_i}\rho\,d\vol_{g_0}\le\frac{1}{c}\vol_{g_0}(E_i)+\frac{c}{2}r_d^{2m}\vol_{g_0}(B_i).
    \end{align*}
    Since $r_d\le1$, it follows that
    \begin{align}\label{Equation: positive energy content in every ball of right scale}
        \vol_{g_0}(E_i)
        \ge
        \frac{c^2}{2}\vol_{g_0}(B_i).
    \end{align}
    Fix $\beta\in\Omega^{0,1}(L)$ and set
    \begin{align*}
        u:=|\beta|_{g_0,h},
        \quad
        (u)_{B_i}:=\fint_{B_i}u\,d\vol_{g_0}.
    \end{align*}
    By the Poincar\'e inequality and the Kato inequality,
    \begin{align}\label{Equation: Poincare}
        \int_{B_i}|u-(u)_{B_i}|^2\,d\vol_{g_0}
        \le
        C_1r_d^2
        \int_{B_i}|du|_{g_0}^2\,d\vol_{g_0}
        \le
        C_1r_d^2
        \int_{B_i}|\nabla\beta|_{g_0,h}^2\,d\vol_{g_0},
    \end{align}
    where $C_1>0$ depends only on the geometry of
    $(\Sigma,g_0)$.
    Using \eqref{Equation: positive energy content in every ball of right scale}, we obtain
    \begin{align*}
        |(u)_{B_i}|^2\vol_{g_0}(B_i)
        \le
        \frac{2}{c^2}
        |(u)_{B_i}|^2\vol_{g_0}(E_i)
        \
        \le
        \frac{4}{c^2}
        \left(
        \int_{E_i}u^2\,d\vol_{g_0}
        +
        \int_{E_i}|u-(u)_{B_i}|^2\,d\vol_{g_0}
        \right).
    \end{align*}
    Since
    \begin{align*}
        \rho\ge\frac{c}{2}r_d^{2m}
        \qquad\mbox{ on }E_i,
    \end{align*}
    we have
    \begin{align*}
        \int_{E_i}u^2\,d\vol_{g_0}
        \le
        \frac{2}{cr_d^{2m}}
        \int_{B_i}\rho u^2\,d\vol_{g_0}.
    \end{align*}
    It follows from the last two estimates and
   \eqref{Equation: Poincare} that
    \begin{align}\label{Equation: estimate of average of beta}
        |(u)_{B_i}|^2\vol_{g_0}(B_i)
        \le
        \frac{8}{c^3r_d^{2m}}
        \int_{B_i}\rho u^2\,d\vol_{g_0}
        +
        \frac{4C_1r_d^2}{c^2}
        \int_{B_i}|\nabla\beta|_{g_0,h}^2\,d\vol_{g_0}.
    \end{align}
    Moreover,
    \begin{align*}
        \int_{B_i}u^2\,d\vol_{g_0}
        =
        \int_{B_i}|u-(u)_{B_i}|^2\,d\vol_{g_0}
        +
        |(u)_{B_i}|^2\vol_{g_0}(B_i).
    \end{align*}
    Combining this identity with \eqref{Equation: Poincare} and
    \eqref{Equation: estimate of average of beta}, we find
    \begin{align*}
        r_d^{-2}
        \int_{B_i}|\beta|_{g_0,h}^2\,d\vol_{g_0}
        \le
        C_2
        \left(
        \int_{B_i}|\nabla\beta|_{g_0,h}^2\,d\vol_{g_0}
        +
        r_d^{-2(m+1)}
        \int_{B_i}\rho|\beta|_{g_0,h}^2\,d\vol_{g_0}
        \right),
    \end{align*}
    where $C_2>0$ depends only on $c$, $r$, and the geometry of $(\Sigma,g_0)$. Since
    \begin{align*}
        r_d^{-2}=r^{-2}d^{\frac{1}{m+1}} \quad\Leftrightarrow\quad r_d^{-2(m+1)}=r^{-2(m+1)}d,
    \end{align*}
    and $r>0$ is fixed, all powers of $r$ may be absorbed into the
    constants and we conclude that
    \begin{align}\label{Equation: local full gradient spectral estimate}
        d^{\frac{1}{m+1}}
        \int_{B_i}|\beta|_{g_0,h}^2\,d\vol_{g_0}
        \le
        C_2
        \left(
        \int_{B_i}|\nabla\beta|_{g_0,h}^2\,d\vol_{g_0}
        +
        d\int_{B_i}\rho|\beta|_{g_0,h}^2\,d\vol_{g_0}
        \right).
    \end{align}
    Summing \eqref{Equation: local full gradient spectral estimate} over
    $i=1,\ldots,N_d$ and using the bounded multiplicity of the cover, we obtain
    \begin{align}\label{Equation: global full gradient spectral estimate}
        d^{\frac{1}{m+1}}
        \|\beta\|_{L^2(\Sigma)}^2
        \le
        C_3
        \left(
        \|\nabla\beta\|_{L^2(\Sigma)}^2
        +
        d\int_\Sigma
        \rho|\beta|_{g_0,h}^2\,d\vol_{g_0}
        \right),
    \end{align}
    for some $C_3>0$ independent of $d$.
    Since $\Sigma$ is a complex curve, we have
    \begin{align*}
        \bar\partial\beta\in\Omega^{0,2}(L)=\{0\}.
    \end{align*}
    Thus, by the Bochner--Kodaira--Weitzenb\"ock formula, we get
    \begin{align}\label{Equation: real Weitzenbock beta}
        2\|\bar\partial^*\beta\|_{L^2(\Sigma)}^2
        &=2\big(\|\bar\partial\beta\|_{L^2(\Sigma)}^2+\|\bar\partial^*\beta\|_{L^2(\Sigma)}^2\big)\\
        \nonumber
        &=
        \|\nabla\beta\|_{L^2(\Sigma)}^2
        +
        \int_\Sigma
        \left(
        \Lambda_{\omega_0}(iF_\nabla)
        +
        K_{g_0}
        \right)
        |\beta|_{g_0,h}^2\,d\vol_{g_0}.
    \end{align}
    Since $K_{g_0}\ge-1$ and $\Lambda_{\omega_0}(iF_\nabla)\ge 2\pi d\rho$, it follows that
    \begin{align*}
        2\|\bar\partial^*\beta\|_{L^2(\Sigma)}^2
        \ge
        \|\nabla\beta\|_{L^2(\Sigma)}^2
        +
        2\pi d
        \int_\Sigma\rho|\beta|_{g_0,h}^2\,d\vol_{g_0}
        -
        \|\beta\|_{L^2(\Sigma)}^2.
    \end{align*}
    In particular,
    \begin{align}\label{Equation: full gradient controlled by delbar adjoint}
        \|\nabla\beta\|_{L^2(\Sigma)}^2
        +
        d\int_\Sigma\rho|\beta|_{g_0,h}^2\,d\vol_{g_0}
        \le
        2\|\bar\partial^*\beta\|_{L^2(\Sigma)}^2
        +
        \|\beta\|_{L^2(\Sigma)}^2.
    \end{align}
    Combining \eqref{Equation: global full gradient spectral estimate} and
    \eqref{Equation: full gradient controlled by delbar adjoint}, we obtain
    \begin{align*}
        d^{\frac{1}{m+1}}
        \|\beta\|_{L^2(\Sigma)}^2
        \le
        C_3
        \left(
        2\|\bar\partial^*\beta\|_{L^2(\Sigma)}^2
        +
        \|\beta\|_{L^2(\Sigma)}^2
        \right).
    \end{align*}
    Choose $d_1\in\n$ sufficiently large that
    \begin{align*}
        d^{\frac{1}{m+1}}\ge2C_3
        \qquad
        \forall\,d\ge d_1.
    \end{align*}
    The last term on the right-hand side can then be absorbed into the
    left-hand side, giving
    \begin{align*}
        \|\bar\partial^*\beta\|_{L^2(\Sigma)}^2
        \ge
        \frac{1}{4C_3}
        d^{\frac{1}{m+1}}
        \|\beta\|_{L^2(\Sigma)}^2.
    \end{align*}
    Thus \eqref{Equation: polynomial spectral gap for delbar} holds with
    $\tilde C_1:=4C_3$.
    
    \medskip
    \noindent
    \textbf{Step 2: approximate peak section with prescribed jet}. Choose holomorphic coordinates on a fixed neighbourhood of $x$, uniformly
    controlled with respect to $g_0$, with $z(x)=0$. A disk of radius
    $r_d=c_0d^{-\frac12}$ is contained in this neighbourhood for large $d$.
    Choose $c_0>0$ small enough that $\rho\geq c/2$ on a slightly larger
    disk; this is possible because $\rho(x)\geq c$ and
    $\|d\rho\|_{L^\infty}\leq c^{-1}\sqrt d$.
    The constants multiplying these radii are fixed independently of $d$.
    
    We use the following consequence of
    \eqref{Equation: natural scale curvature bounds}.
    In the coordinate $u=\sqrt d\,z$, local holomorphic frames may be chosen
    so that the frame metric, its inverse, and their derivatives of each
    fixed order are uniformly bounded on a fixed smaller disk.
    To see this, write the curvature equation in the rescaled coordinate,
    solve the scalar Poisson equation for a local metric potential on a larger
    fixed disk with zero boundary values, and use interior elliptic estimates.
    The difference from the original local potential is harmonic and can be
    removed by a holomorphic change of frame. A constant change normalizes
    the frame norm at the centre. The rescaled connection coefficients and
    the fixed-order changes between ordinary and covariant jets are therefore
    uniformly bounded. The corresponding coordinate estimates are uniform
    also for the source metric.
    
    It follows that the prescribed jet $J$ is realized by a polynomial
    section $P_Je$ on this disk whose rescaled derivatives of every fixed
    order are bounded by $C_{j,\ell}|J|_{g_0,h}$. More explicitly, the ordinary
    Taylor coefficients in $u$ are obtained from
    $(d^{-a/2}J_a)_{a=0}^{\ell}$ by a uniformly bounded triangular linear
    change of variables. Since $d\geq1$, their norms are bounded by
    $C_\ell|J|_{g_0,h}$.
    Choose a cutoff $\chi$ equal to one near $x$, supported in the chosen
    disk, with $|\nabla^j\chi|\leq C_jd^{\frac{j}{2}}$, and define
    \begin{align*}
        \sigma_{\mathrm{app}}:=\chi P_Je,\qquad
        \alpha:=\bar\partial\sigma_{\mathrm{app}}.
    \end{align*}
    The cutoff is chosen with transition annulus at distance comparable to
    $d^{-\frac12}$ from $x$. Thus $\alpha$ vanishes near $x$, its support has
    area at most $C/d$, and
    \begin{align}\label{Equation: approximate peak natural scale estimates}
        j_x^\ell(\sigma_{\mathrm{app}})&=J,\\
        |\nabla^j\sigma_{\mathrm{app}}|_{g_0,h}
            &\leq C_jd^{\frac{j}{2}}|J|_{g_0,h},\qquad
        |\alpha|_{g_0,h}\leq C\sqrt d\,|J|_{g_0,h}.
        \nonumber
    \end{align}
    %
    
    \medskip
    \noindent
    \textbf{Step 3: weighted correction and pointwise estimates}.
    Choose $w=w_x$ equal to $|z|^2$ on a fixed coordinate neighbourhood
    of $x$, uniformly bounded above on $\Sigma$ and uniformly bounded below
    away from that fixed neighbourhood. In particular, the shrinking disks
    used in Step 2 are contained in the coordinate neighbourhood; the
    positive lower bound away from it is not asserted outside a shrinking
    disk. Such $w_x$ and the divisor metrics may be chosen with constants
    uniform in $x$.
    Put $q:=\ell+1$, and set
    \begin{align*}
        A_d:=d^{\frac{1}{m+1}},
        \qquad
        b_d:=2\pi d\rho.
    \end{align*}
    We verify the hypotheses of
    Lemma~\ref{Lemma: del bar solution with weighted norm}.
    By \eqref{Equation: polynomial spectral gap for delbar},
    \begin{align*}
        \|\beta\|_{L^2(h)}
        \leq\frac{\sqrt{\tilde C_1}}{\sqrt{A_d}}
        \|\bar\partial_h^*\beta\|_{L^2(h)}
        \qquad\forall\,\beta\in\Omega^{0,1}(L).
    \end{align*}
    Moreover, the curvature hypothesis gives
    \begin{align*}
        \Lambda_{\omega_0}(iF_h)\geq2\pi d\rho=b_d,
    \end{align*}
    so the curvature condition in the lemma holds with $C_0=0$.
    To verify the scalar coercivity condition, apply the covering and
    Poincar\'e argument of Step 1 directly to an arbitrary
    $u\in C^\infty(\Sigma,\r)$, without the use of Kato's inequality.
    The same argument leading to
    \eqref{Equation: global full gradient spectral estimate} gives
    \begin{align*}
        A_d\int_\Sigma|u|^2\,d\vol_{g_0}
        &\leq C_3\left(
            \int_\Sigma|du|_{g_0}^2\,d\vol_{g_0}
            +d\int_\Sigma\rho|u|^2\,d\vol_{g_0}
        \right)\\
        &\leq C_3\int_\Sigma
        \bigl(|du|_{g_0}^2+b_d|u|^2\bigr)\,d\vol_{g_0}.
    \end{align*}
    Thus the scalar condition holds with $\lambda:=C_3^{-1}$.
    Finally, $\alpha$ is smooth and vanishes in a neighbourhood of $x$,
    so it satisfies the integrability condition
    \eqref{Equation: 7}. By the choice of the cutoff support in Step 2,
    we also have
    \begin{align*}
        b_d\geq\pi c d
        \qquad\mbox{on }\,\operatorname{spt}\alpha.
    \end{align*}
    After increasing $d_0$ if necessary, we may therefore apply
    Lemma~\ref{Lemma: del bar solution with weighted norm}.
    Fix $\delta:=\delta_0/2$, where $\delta_0>0$ is supplied by the lemma,
    and use its strengthened estimate
    \eqref{Equation: weighted correction local curvature gain}.
    Since $A_d\leq d$, on its support we have
    \begin{align*}
        (A_d+b_d)^{-1}\leq C d^{-1},\qquad
        w^{-q}\leq C d^q,\qquad
        e^{\delta A_d\dist_{g_0}^2(x,\,\cdot\,)}\leq C.
    \end{align*}
    Together with \eqref{Equation: approximate peak natural scale estimates}
    and the area bound for the support, this gives a solution $\xi$ with
    $\bar\partial\xi=\alpha$, $j_x^\ell(\xi)=0$, and
    \begin{align}\label{Equation: corrected peak weighted norm}
        \int_\Sigma |\xi|_h^2w^{-q}
           e^{\delta A_d\dist_{g_0}^2(x,\,\cdot\,)}\,d\vol_{g_0}
        \leq C d^{q-1}|J|_{g_0,h}^2.
    \end{align}
    Here the power is $d^{-1}\cdot d^q\cdot d\cdot d^{-1}=d^{q-1}$:
    respectively the local curvature gain, the singular weight, the squared
    size of $\alpha$, and the area of its support.
    The section $\sigma_d:=\sigma_{\mathrm{app}}-\xi$ is holomorphic and
    has the prescribed jet. Since $w$ is uniformly bounded above and
    $\int_\Sigma|\sigma_{\mathrm{app}}|_h^2
     e^{\delta A_d\dist_{g_0}^2(x,\,\cdot\,)}\,d\vol_{g_0}
     \leq C d^{-1}|J|_{g_0,h}^2$, we conclude that
    \begin{align}\label{Equation: holomorphic peak global weighted norm}
        \int_\Sigma|\sigma_d|_h^2
           e^{\delta A_d\dist_{g_0}^2(x,\,\cdot\,)}\,d\vol_{g_0}
        \leq C d^{q-1}|J|_{g_0,h}^2.
    \end{align}
    At any $y\in\Sigma$, the uniformly controlled holomorphic frames above
    and the ordinary Cauchy estimates give
    \begin{align*}
        |\nabla^j\sigma_d(y)|_{g_0,h}
        \leq C_jd^{(j+1)/2}
           \|\sigma_d\|_{L^2(B_{c_1d^{-\frac12}}(y))}.
    \end{align*}
    For $z\in B_{c_1d^{-\frac12}}(y)$, the triangle inequality implies
    \begin{align*}
        \dist_{g_0}^2(x,z)
        \geq\tfrac12\dist_{g_0}^2(x,y)-c_1^2d^{-1}.
    \end{align*}
    Since $A_d/d\leq1$, \eqref{Equation: holomorphic peak global weighted norm}
    therefore yields
    \begin{align*}
        |\nabla^j\sigma_d(y)|_{g_0,h}
        &\leq C_jd^{(j+1)/2}d^{(q-1)/2}|J|_{g_0,h}
              e^{-\delta A_d\dist_{g_0}^2(x,y)/4}\\
        &\leq Kd^{(\ell+1+j)/2}|J|_{g_0,h}
              e^{-d^{1/(m+1)}\dist_{g_0}^2(x,y)/K}.
    \end{align*}
    This proves \eqref{Equation: Gaussian peak estimate}.
\end{proof}
%
%
\section{Balanced holomorphic covers of the sphere by closed Riemann surfaces}
The goal of this section is to construct holomorphic branched covers of $\cp^1$ whose geometry remains quantitatively controlled as their degree tends to infinity. The construction proceeds in three stages. We first use a sharp constant-curvature version of the peak-section construction of Section~2 to obtain a uniformly base-point-free pencil of high degree. We then perturb this pencil so that its Wronskian is quantitatively non-degenerate, first on a definite region at the natural scale $d^{-\frac12}$ and then globally up to a polynomial loss. Finally, taking the ratio of the two sections produces the balanced holomorphic covers required later in the construction of balanced Bor{\accent23 u}vka covers.

\medskip
We equip $\cp^1$ with its Fubini--Study metric. Unless otherwise specified, the norms of $df$ and $\nabla^2f$ are computed using $g_0$ on $\Sigma$ and the Fubini--Study metric on $\cp^1$. 
\begin{definition} \label{Definition: m-balanced cover}
	Let $f:\Sigma\to\cp^1$ be a branched holomorphic cover of degree $d$. We say that $f$ is \textit{balanced} if there exist a universal non-negative integer $n_0\in\n$ and constants $c,C>0$ depending only on the geometry of $(\Sigma,g_0)$ such that the following facts hold.
	\begin{enumerate}[(i)]
		\item $\lvert df\rvert_{g_0}^2\le Cd$ and $\lvert \nabla^2f\rvert_{g_0}\le Cd$ on $\Sigma$. 
		\item For every $x\in\Sigma$,
		\begin{align*}
			\fint_{B_{d^{-\frac{1}{2}}}(x)}\lvert df\rvert_{g_0}^2\, d\vol_{g_0}\ge cd.
		\end{align*}
        \item We have 
            \begin{align*}
                |df|_{g_0}^2+d^{-1}|\nabla^2f|_{g_0}^2\ge c d^{-n_0} \qquad\mbox{ on }\,\Sigma.
            \end{align*}
	\end{enumerate}
\end{definition}
The three conditions above describe complementary aspects of the geometry of $f$ at the natural scale $d^{-\frac12}$. Condition~(i) prevents the concentration of its first and second derivatives, while condition~(ii) ensures that a definite amount of energy is present in every ball of that scale. Condition~(iii) rules out excessive degeneracy: even near a branch point, the first and second derivatives cannot vanish simultaneously faster than a fixed polynomial rate in $d$. The precise value of the exponent $n_0$ will play no role; what is important is that the lower bound is polynomial rather than exponentially small.
\begin{remark}\label{Remark: balanced pullback bounded geometry}
    Throughout the uniform estimates below, the constants $c,C,n_0$ in the definition of balancedness are fixed. No extra higher-derivative hypothesis is required in the definition of a balanced cover. Holomorphicity and condition~(i), applied in source disks of radius $c d^{-\frac12}$ whose images remain in fixed target coordinate disks, give the Cauchy estimates
    \begin{align*}
        \|\nabla^k f\|_{L^\infty}\leq C_kd^{k/2} \qquad\forall\,k\geq1.
    \end{align*}
    Consequently, for every fixed positive integer $a$, the pullback metric on $f^*\mathcal O_{\cp^1}(a)$ satisfies
    \begin{align*}
        \|\nabla^kF_{f^*h_a}\|_{L^\infty}\leq C_{k,a}d^{1+k/2}\leq B_{k,a}(ad)^{1+k/2}.
    \end{align*}
    These are exactly the natural-scale curvature bounds for a bundle of degree $ad$. Here $h_a$ is the fixed Fubini--Study metric on $\mathcal O_{\cp^1}(a)$. 
    %
    %
    Moreover, setting
    \begin{align*}
        b_d:=\Lambda_{\omega_0}(iF_{f^*h_a})=\kappa_a|df|_{g_0}^2
    \end{align*}
    for a fixed constant $\kappa_a>0$, conditions~(i) and~(ii) give
    \begin{align*}
        0\leq b_d\leq C_a d, \qquad \fint_{B_{d^{-\frac12}}(x)}b_d\,d\vol_{g_0}\geq c_a d \qquad\forall\,x\in\Sigma.
    \end{align*}
    Applying the covering and Poincar\'e argument in Step~1 of Proposition~\ref{Proposition: holomorphic peak sections} to real-valued functions on balls of radius $d^{-\frac12}$, we obtain
    \begin{align*}
        \int_\Sigma\bigl(|du|_{g_0}^2+b_d|u|^2\bigr)\,d\vol_{g_0}\geq\lambda_a d\int_\Sigma|u|^2\,d\vol_{g_0} \qquad\forall\,u\in C^\infty(\Sigma,\r),
    \end{align*}
    where $\lambda_a>0$ is independent of $d$ and of the balanced cover. This verifies the scalar coercivity hypothesis of Lemma~\ref{Lemma: del bar solution with weighted norm} for these pullback metrics.
\end{remark}
\subsection{Quantitative holomorphic pencils}
Recall that a subset $\Lambda\subset\Sigma$ is said to be  \textit{$r$-separated} if
\begin{align*}
    \dist_{g_0}(p,q)\geq r \qquad\forall\,p,q\in\Lambda,\quad p\neq q.
\end{align*}
The following elementary estimate controls the superposition of Gaussian peaks centred at a separated collection of points. Its first conclusion gives a uniform bound for the total contribution of all peaks, while the second shows that interactions between distinct peaks
in the same well-separated family are exponentially small.
\begin{lemma}\label{Lemma: bound on sum of exponentials}
    Fix $\eta\in(0,+\infty)$ and $\kappa>0$. For $d$ sufficiently large, choose an $\eta d^{-\frac{1}{2}}$-separated set $\Lambda_d\subset\Sigma$.
    Then, there exist constants $c,C>0$ depending only on the geometry of $(\Sigma,g_0)$ and $\kappa$ such that
    \begin{align}\label{Equation: first bound lemma}
        \sum_{p\in\Lambda_d}e^{-\kappa d\dist_{g_0}^2(p,\,\cdot\,)}\le C(1+\eta^{-2})\qquad\mbox{ on }\,\Sigma.
    \end{align}
    and
    \begin{align}\label{Equation: second bound lemma}
        \sum_{\substack{p\in\Lambda_d\\p\neq q}}e^{-\kappa d\dist_{g_0}^2(p,q)}\le C(1+\eta^{-2})e^{-c\eta^2}\qquad\forall\,q\in\Lambda_d.
    \end{align}
\end{lemma}
\begin{proof}
    Fix any $x\in\Sigma$ and define
    \begin{align*}
        A_0&:=\big\{p\in\Lambda_d \mbox{ : } \dist_{g_0}^2(p,x)< d^{-1}\big\}\\
        A_j&:=\big\{p\in\Lambda_d \mbox{ : } jd^{-1}\le\dist_{g_0}^2(p,x)< (j+1)d^{-1}\big\} \qquad\forall\,j\ge 1.
    \end{align*}
    Note that, as $\Lambda_d$ is $\eta d^{-\frac{1}{2}}$-separated, the balls $\Big\{B_{\frac{\eta}{4}d^{-\frac{1}{2}}}(p)\Big\}_{p\in\Lambda_d}$ are disjoint. Moreover, for $d$ sufficiently large, there are constants $c_1,C_1>0$ depending only on the geometry of $(\Sigma,g_0)$ such that
    \begin{align*}
        \vol_{g_0}\Big(B_{\frac{\eta}{4}d^{-\frac{1}{2}}}(y)\Big)\ge c_1\eta^2d^{-1} \quad\mbox{ and }\quad \vol_{g_0}(B_r(y))\le C_1r^2,
    \end{align*}
    for every $y\in\Sigma$ and $r\in(0,+\infty)$. Then, letting $C_2:=\frac{2C_1}{c_1}>0$, for every $j\ge 0$ we get
    \begin{align*}
        \#A_j\cdot c_1\eta^2d^{-1}&\le\sum_{p\in A_j}\vol_{g_0}\Big(B_{\frac{\eta}{4}d^{-\frac{1}{2}}}(p)\Big)\\
        &\le\vol_{g_0}\left(B_{(\sqrt{j+1}+\frac{\eta}{4})d^{-\frac{1}{2}}}(x)\right)\\
        &\le 2C_1(j+1+\eta^2)d^{-1},
    \end{align*}
    which implies that 
    \begin{align*}
        \#A_j\le C_2\bigg(1+\frac{j+1}{\eta^2}\bigg).
    \end{align*}
    Hence,
    \begin{align*}
        \sum_{p\in\Lambda_d}e^{-\kappa d\dist_{g_0}^2(p,x)}&=\sum_{j=0}^{+\infty}\sum_{p\in A_j}e^{-\kappa d\dist_{g_0}^2(p,x)}\le\sum_{j=0}^{+\infty}\#A_je^{-\kappa j}\\
        &\le C_2\sum_{j=0}^{+\infty}\bigg(1+\frac{j+1}{\eta^2}\bigg)e^{-\kappa j}\le C_3(1+\eta^{-2}),
    \end{align*}
    where
    \begin{align*}
        C_3:=C_2\sum_{j=0}^{+\infty}(j+1)e^{-\kappa j}<+\infty
    \end{align*}
    depends on the geometry of $(\Sigma,g_0)$ and $\kappa$. Thus, \eqref{Equation: first bound lemma} follows by arbitrariness of $x\in\Sigma$. To prove \eqref{Equation: second bound lemma}, fix $q\in\Lambda_d$. Since $\Lambda_d$ is $\eta d^{-\frac{1}{2}}$-separated, we have
    \begin{align*}
        d\dist_{g_0}^2(p,q)\ge\eta^2 \qquad\forall\,p\in\Lambda_d\smallsetminus\{q\}.
    \end{align*}
    Consequently,
    \begin{align*}
        e^{-\kappa d\dist_{g_0}^2(p,q)}&=e^{-\frac{\kappa}{2}d\dist_{g_0}^2(p,q)}e^{-\frac{\kappa}{2}d\dist_{g_0}^2(p,q)}\\
        &\le e^{-\frac{\kappa}{2}\eta^2}e^{-\frac{\kappa}{2}d\dist_{g_0}^2(p,q)}
    \end{align*}
    for every $p\in\Lambda_d\smallsetminus\{q\}$. Applying \eqref{Equation: first bound lemma} with $\kappa$ replaced by $\frac{\kappa}{2}$, we obtain
    \begin{align*}
        \sum_{\substack{p\in\Lambda_d\\p\neq q}}e^{-\kappa d\dist_{g_0}^2(p,q)}\le e^{-\frac{\kappa}{2}\eta^2}\sum_{p\in\Lambda_d}e^{-\frac{\kappa}{2}d\dist_{g_0}^2(p,q)}\le C_3\bigl(1+\eta^{-2}\bigr)e^{-\frac{\kappa}{2}\eta^2}.
    \end{align*}
    Setting \(c\coloneqq\frac{\kappa}{2}>0\), the statement follows.
\end{proof}
We next use peak sections to construct quantitative holomorphic pencils. Recall that, if $L\rightarrow\Sigma$ is a holomorphic line bundle on $\Sigma$, a pair $s_0,s_1\in H^0(L)$ with no common zero determines a holomorphic map
\begin{align*}
    [s_0:s_1]:\Sigma\rightarrow\cp^1.
\end{align*}
The differential of this map is encoded by the holomorphic Wronskian
\begin{align*}
    W(s_0,s_1):=s_0\nabla s_1-s_1\nabla s_0\in H^0\big(\omega_\Sigma\otimes L^{\otimes2}\big),
\end{align*}
where $\omega_\Sigma:=(T^*\Sigma)^{1,0}$ denotes the canonical bundle of $\Sigma$. In particular, the zero set of $W(s_0,s_1)$ is precisely the branch locus of $[s_0:s_1]$. We use $h$ also to denote the Hermitian metrics
induced on tensor powers of $L$.
\begin{lemma}\label{Lemma: quantitative pencil lemma}
Let $(L,h)\to\Sigma$ be a holomorphic Hermitian line bundle of degree $d$, and let $\nabla$ be its Chern connection. Assume that
\begin{align}
    \frac{i}{2\pi}F_{\nabla}=\lambda_d\omega_0,\qquad\lambda_d:=\frac{d}{\vol_{g_0}(\Sigma)}
\end{align}
Then there exist constants $a,A,\rho>0$ and $d_0\ge1$, depending only on $(\Sigma,g_0)$, such that for every $d\ge d_0$ we can construct holomorphic sections $s_{0},s_{1}\in H^0(L)$ satisfying the following conditions.
\begin{enumerate}[(i)]
    \item  $a\le|s_{0}|_{h}^2+|s_{1}|_{h}^2\le A$ on $\Sigma$.
    \smallskip
    \item  For $i=0,1$ and $j=1,2,3$ we have
    \begin{align*}
         |\nabla^j s_{i}|_{g_0,h}\le Ad^{\frac{j}{2}}.
    \end{align*}
    \smallskip
    \item Letting
        \begin{align*}
            W:=W(s_0,s_1)=s_{0}\nabla s_{1}-s_{1}\nabla s_{0},
        \end{align*}
        for every $x\in\Sigma$ there exists a point $p_x\in B_{\frac{1}{2} d^{-\frac{1}{2}}}(x)$ such that $B_{\rho d^{-\frac{1}{2}}}(p_x)\subset B_{d^{-\frac{1}{2}}}(x)$ and
        \begin{align*}
            |W|_{g_0,h}\ge a\sqrt d \qquad\mbox{ on }\, B_{\rho d^{-\frac{1}{2}}}(p_x).
        \end{align*}
\end{enumerate}
\end{lemma}
\begin{proof}
We divide the proof into two steps. All connections on tensor products are induced by the Chern connection of $(L,h)$ and the Levi-Civita connection of $g_0$. In particular, the norm of $W\in H^0(\omega_\Sigma\otimes L^{\otimes2})$ is induced by $g_0$ and $h^{\otimes2}$.

\medskip
\noindent
\textbf{Step 1: construction of a uniformly base-point-free pair}.
We first establish the sharp peak-section estimates available for the constant-curvature metrics in this lemma. Fix $k\geq0$ and a holomorphic $k$-jet $J=(J_a)_{a=0}^k$ at $p$, expressed through the $(a,0)$-parts of the covariant derivatives, and put
\begin{align*}
    |J|_d^2:=\sum_{a=0}^k d^{-a}|J_a|_{g_0,h}^2.
\end{align*}
For all sufficiently large $d$, we claim that there are complex-linear choices $P_{d,p,k}J\in H^0(L)$ with jet $J$ and
\begin{align}\label{Equation: sharp constant curvature jet peaks}
    |\nabla^j(P_{d,p,k}J)|_{g_0,h}\leq C_{j,k}d^{\frac{j}{2}
    }|J|_de^{-c_kd\dist_{g_0}^2(p,\,\cdot\,)} \qquad\forall\,j\geq0.
\end{align}
The constants in the above estimate are uniform in $p$, $d$, and the constant-curvature bundle $(L,h)$.
Choose a uniformly controlled holomorphic coordinate $z$ centred at $p$, with $|dz|_{g_0}(p)=1$, on a disk of fixed radius. On this disk, choose a real-valued potential $\phi_p$ satisfying
\begin{align*}
    i\partial\bar\partial\phi_p=\frac{2\pi}{\vol_{g_0}(\Sigma)}\omega_0.
\end{align*}
Subtracting the real part of a holomorphic polynomial of degree at most two, we may arrange that
\begin{align*}
    \phi_p(z)=\nu_p|z|^2+O(|z|^3), \qquad 0<c\leq\nu_p\leq C.
\end{align*}
Here $\nu_p:=\partial_z\partial_{\bar z}\phi_p(0)$; its positivity and uniform upper and lower bounds follow from the curvature equation and the uniform control of the coordinates. The potentials and their fixed-order derivatives may be chosen uniformly bounded. After decreasing the coordinate radius, independently of $p$, we therefore have
\begin{align*}
    c|z|^2\leq\phi_p(z)\leq C|z|^2.
\end{align*}
The curvature equation gives a holomorphic frame $e_{d,p}$ with $|e_{d,p}|_h^2=e^{-d\phi_p}$. We fix its phase so that $e_{d,p}(p)=e_p$, where $e_p\in L_p$ is a chosen unit vector.
There is a unique polynomial
\begin{align*}
    Q_J(z)=\sum_{a=0}^k c_a(\sqrt d\,z)^a
\end{align*}
such that $Q_Je_{d,p}$ has jet $J$ at $p$. In the coordinate $\zeta=\sqrt d\,z$, the connection coefficients and their fixed-order derivatives are uniformly bounded near $\zeta=0$. At order $a$, the prescribed rescaled jet equals $a!c_a$ plus a linear combination of $c_0,\ldots,c_{a-1}$ with uniformly bounded coefficients, so solving successively for $c_0,\ldots,c_k$ gives
\begin{align*}
    \sum_{a=0}^k|c_a|\leq C_k|J|_d.
\end{align*}
Let $\chi_p$ be a uniformly controlled cut-off supported in the coordinate disk and equal to one on a smaller fixed disk. Extend
\begin{align*}
    \widetilde s_J:=\chi_pQ_Je_{d,p}
\end{align*}
by zero to $\Sigma$. Using $\nabla e_{d,p}=-d\,\partial\phi_p\otimes e_{d,p}$, we obtain
\begin{align*}
    |\nabla^j\widetilde s_J(y)|_{g_0,h}\leq C_{j,k}d^{\frac{j}{2}}|J|_d\bigl(1+\sqrt d\,\dist_{g_0}(p,y)\bigr)^{N_{j,k}} e^{-cd\dist_{g_0}^2(p,y)}.
\end{align*}
Here $N_{j,k}\geq0$ are integers depending only on $j$ and $k$. The polynomial factor is absorbed by decreasing the Gaussian exponent.
The error 
\begin{align*}
    \alpha_J:=\bar\partial\widetilde s_J=(\bar\partial\chi_p)Q_Je_{d,p}
\end{align*}
is supported in a fixed annulus separated from $p$. Consequently, after absorbing polynomial factors in $d$,
\begin{align*}
    \|\nabla^r\alpha_J\|_{L^\infty(g_0,h)}\leq C_{r,k}e^{-c_*d}|J|_d \qquad\forall\,r\geq0,
\end{align*}
where $c_*>0$ is independent of $r$.
Equip $\mathcal O(-(k+1)p)$ with the uniformly controlled
divisor metric $h_{M_p}$ used in Lemma~\ref{Lemma: del bar solution with weighted norm}. Set $\widehat L:=L(-(k+1)p)$ and $\widehat h:=h\otimes h_{M_p}$, and denote the natural inclusion into $L$ by $\iota$. Since $\alpha_J$ vanishes near $p$, its lift $\widehat\alpha_J$ is smooth and satisfies
\begin{align*}
    \|\widehat\alpha_J\|_{L^2(\widehat h)}\leq C_ke^{-c_*d}|J|_d.
\end{align*}
The curvature of the divisor metric is uniformly bounded, so
\begin{align*}
    \Lambda_{\omega_0}(iF_{\widehat h})+K_{g_0}\geq\frac{2\pi d}{\vol_{g_0}(\Sigma)}-C_k.
\end{align*}
For sufficiently large $d$, the Bochner--Kodaira--Weitzenb\"ock identity and the preceding curvature bound give
\begin{align*}
    \|\beta\|_{L^2(\widehat h)}\leq Cd^{-\frac 12}\|\bar\partial_{\widehat h}^*\beta\|_{L^2(\widehat h)}\qquad\forall\,\beta\in\Omega^{0,1}(\widehat L).
\end{align*}
The Hahn--Banach and Riesz argument used in Lemma~\ref{Lemma: del bar solution with weighted norm} then gives a solution of $\bar\partial\widehat\xi_J=\widehat\alpha_J$. Choosing the solution orthogonal to $H^0(\widehat L)$ gives the least-norm solution, which satisfies
\begin{align*}
    \|\widehat\xi_J\|_{L^2(\widehat h)}\leq Cd^{-\frac12}\|\widehat\alpha_J\|_{L^2(\widehat h)}.
\end{align*}
Elliptic regularity implies that $\widehat\xi_J$ is smooth. The section $\xi_J:=\iota(\widehat\xi_J)$ satisfies $\bar\partial\xi_J=\alpha_J$ and
\begin{align*}
    \|\xi_J\|_{L^2(h)}\leq C_kd^{-\frac12}e^{-c_*d}|J|_d.
\end{align*}
Moreover, $\widehat\xi_J$ is holomorphic near $p$, so $\xi_J$ vanishes there to order at least $k+1$.
Interior elliptic estimates on disks of radius comparable to $d^{-\frac12}$ give
\begin{align*}
    d^{-\frac{j}{2}}\|\nabla^j\xi_J\|_{L^\infty}\leq C_j\left(d^{\frac 12}\|\xi_J\|_{L^2}+\sum_{r=0}^{j+1}d^{-\frac{r+1}{2}}\|\nabla^r\alpha_J\|_{L^\infty}
    \right)\leq C_{j,k}e^{-c_*d}|J|_d.
\end{align*}
The constants are uniform because the metric and connection have uniformly bounded coefficients in rescaled local frames. Since $\dist_{g_0}(p,y)\leq\operatorname{diam}_{g_0}(\Sigma)$, the last bound is dominated by the right-hand side of \eqref{Equation: sharp constant curvature jet peaks} after decreasing $c_k$. Hence
\begin{align*}
    P_{d,p,k}J:=\widetilde s_J-\xi_J
\end{align*}
has the prescribed jet and satisfies \eqref{Equation: sharp constant curvature jet peaks}. All operations, including the least-norm solution, are complex linear in $J$.
In particular, let $\sigma_p^0$ be the section with $\sigma_p^0(p)=e_p$ and $\nabla\sigma_p^0(p)=0$ obtained by taking $k=1$. There are constants $c_0,c_1>0$ such that
\begin{align}\label{Equation: 14(2)}
    |\nabla^j\sigma_p^0(y)|_{g_0,h}
    &\leq C_jd^{\frac{j}{2}}
       e^{-c_1d\dist_{g_0}^2(p,y)},\\
    |\sigma_p^0(y)|_h
    &\geq c_0
    \qquad\mbox{on }B_{2c_0d^{-\frac12}}(p).
    \nonumber
\end{align}
Indeed, the first derivative bound gives $|\sigma_p^0(y)|_h\geq1-C_1\sqrt d\,\dist_{g_0}(p,y)$, which proves the second inequality after decreasing $c_0$.

\noindent
We now apply the local perturbation and globalization statements in \cite[Propositions~4.2 and~4.1]{Auroux2001Estimated} to obtain a uniformly base-point-free pair of sections. Work with the rescaled metric $g_d:=dg_0$. For a holomorphic pair $s=(s_0,s_1)$ with a fixed bound on its rescaled $C^3$ norm, the quotient
\begin{align*}
    f_p:=s/\sigma_p^0
\end{align*}
is a uniformly bounded holomorphic map to $\c^2$ on the disk where $|\sigma_p^0|\geq c_0$. Fix $c>0$ so small that $c\|\sigma_p^0\|_{C^3(g_d,h)}\leq1$ for every $p$ and all sufficiently large $d$. After rescaling the coordinate and $f_p$ by fixed factors, we apply \cite[Proposition~4.2]{Auroux2001Estimated} to a map from a disk in $\c$ to $\c^2$. Returning to the original normalization, we obtain for every sufficiently small $\varepsilon>0$ a constant vector $v_p\in\c^2$ satisfying
\begin{align*}
    |v_p|&\leq c\varepsilon,\\
    |f_p-v_p|&\geq c'\varepsilon(\log\varepsilon^{-1})^{-\nu}
\end{align*}
on a smaller disk of fixed $g_d$-radius. Here quantitative transversality to zero implies avoidance, since the differential cannot surject onto $\c^2$.
With our choice of $c$, the global holomorphic perturbation $u_p:=-\sigma_p^0v_p$ satisfies
\begin{align*}
    \|u_p\|_{C^3(g_d,h)}&\leq\varepsilon,\\
    |\nabla^ju_p(y)|_{g_d,h}&\leq C_j\varepsilon e^{-c_1d\dist_{g_0}^2(p,y)},\\
    |s(y)+u_p(y)|&\geq c''\varepsilon(\log\varepsilon^{-1})^{-\nu} \qquad\mbox{on }B_{r_*d^{-\frac12}}(p),
\end{align*}
where $r_*>0$ is fixed and $\nu>0$ is the fixed exponent supplied by \cite[Proposition~4.2]{Auroux2001Estimated}. For asymptotically holomorphic input pairs,
$\bar\partial f_p=(\bar\partial s)/\sigma_p^0$, and hence $\|\bar\partial f_p\|_{C^1(g_d)}=O(d^{-\frac12})$. For each fixed $\varepsilon>0$, the smallness assumptions of \cite[Proposition~4.2]{Auroux2001Estimated} therefore hold for sufficiently large $d$. This verifies the local perturbation hypothesis of \cite[Proposition~4.1]{Auroux2001Estimated}.
The property $|s(y)|>\eta$ is local and $C^3$-open, with its change controlled by the $C^0$ norm of the perturbation. Thus \cite[Proposition~4.1]{Auroux2001Estimated} applies with $r=2$. Starting from the zero pair, its globalization construction adds finite sums of the global perturbations $u_p$. Every intermediate pair is therefore holomorphic. We obtain $\tau_0,\tau_1\in H^0(L)$ with
\begin{align*}
    b\leq|\tau_0|_h^2+|\tau_1|_h^2\leq B \qquad\mbox{on }\Sigma
\end{align*}
and
\begin{align}\label{Equation: uniform derivative bounds tau}
    |\nabla^j\tau_i|_{g_0,h}\leq B_jd^{\frac{j}{2}}, \qquad i=0,1,\quad j=0,1,2,3.
\end{align}
All constants are uniform in $d$ and $(L,h)$: the local estimates above depend only on the background geometry, while $iF_\nabla$ is a fixed positive multiple of $d\omega_0$ and its positive-order covariant derivatives vanish.

\medskip
\noindent
\textbf{Step 2: perturbing the base-point-free pair to force the Wronskian lower bound}.
Let $\tau:=(\tau_0,\tau_1)$ be the pair obtained in Step~1. We will perturb this pair so that its Wronskian is uniformly bounded below at the points of a net of mesh $\frac14d^{-\frac12}$. The derivative bounds will then extend these inequalities to small balls around the net points, yielding assertion~(iii).

Fix $\delta:=1/4$ and choose a maximal $\delta d^{-\frac12}$-separated set $\Lambda_d$. By maximality,
\begin{align*}
    \dist_{g_0}(x,\Lambda_d)\leq\delta d^{-\frac12}\qquad\forall\,x\in\Sigma.
\end{align*}
For each $p\in\Lambda_d$, apply \eqref{Equation: sharp constant curvature jet peaks} with $k=1$ to obtain $\sigma_p^1\in H^0(L)$ satisfying
\begin{align*}
    \sigma_p^1(p)=0,\qquad\nabla\sigma_p^1(p)=\sqrt d\,dz\otimes e_p.
\end{align*}
Here $|dz|_{g_0}(p)=|e_p|_h=1$, as in Step~1. The prescribed jet has rescaled norm one, so, decreasing $c_1$ if necessary,
\begin{align*}
    |\nabla^j\sigma_p^1(y)|_{g_0,h}\leq C_jd^{\frac{j}{2}}e^{-c_1d\dist_{g_0}^2(p,y)} \qquad j=0,1,2,3.
\end{align*}
Applying Lemma~\ref{Lemma: bound on sum of exponentials} with $\eta=\delta$ and $\kappa=c_1$ gives
\begin{align}\label{Equation: 15(2)}
    \sum_{p\in\Lambda_d}e^{-c_1d\dist_{g_0}^2(p,y)}\leq K_1 \qquad\forall\,y\in\Sigma,
\end{align}
where $K_1$ is independent of $d$.
To control the interaction between peaks added simultaneously, let $R>1$ be a separation parameter, to be chosen below. Uniform packing implies that each point of $\Lambda_d$ has at most $C(1+R)^2$ other net points at distance less than $Rd^{-\frac12}$. Assigning points successively to classes, while keeping points at distance less than $Rd^{-\frac12}$ in different classes, gives a decomposition
\begin{align*}
    \Lambda_d=\bigcup_{\mu=1}^{M}\Lambda_d^\mu, \qquad M=M(R)\leq C(1+R)^2,
\end{align*}
such that each $\Lambda_d^\mu$ is $Rd^{-\frac12}$-separated. We fix $M=M(R)$ independently of $d$, allowing some classes $\Lambda_d^\mu$ to be empty. Applying Lemma~\ref{Lemma: bound on sum of exponentials} to each class, with $\eta=R$ and $\kappa=c_1$, yields
\begin{align}\label{Equation: same class Gaussian tail}
    \sum_{\substack{q\in\Lambda_d^\mu\\q\neq p}}e^{-c_1d\dist_{g_0}^2(p,q)}&\leq C(1+R^{-2})e^{-cR^2}\leq K_2e^{-cR^2} \qquad\forall\,p\in\Lambda_d^\mu.
\end{align}
The constants $K_2,c>0$ are independent of $R$, $d$, and $\mu$; for each fixed $R$, this estimate holds for sufficiently large $d$.
We next fix the constants controlling the perturbations.
For any subset $S\subset\Lambda_d$ and coefficients
$a_p\in\c^2$ with $|a_p|\leq1$, the peak estimates and
\eqref{Equation: 15(2)} give
\begin{align}\label{Equation: uniform first jet perturbation bound}
    \left\|
       \eta\sum_{p\in S}a_p\sigma_p^1
    \right\|_{C^3(g_d,h)}
    \leq C_{\mathrm{pert}}\eta.
\end{align}
Choose
\begin{align*}
    0<\eta_{\mathrm{tot}}
    \leq
    \min\left\{1,\frac{\sqrt b}{2C_{\mathrm{pert}}}\right\}.
\end{align*}
Any sequence of such perturbations whose amplitudes sum to at
most $\eta_{\mathrm{tot}}$ then gives intermediate pairs $s$ with
\begin{align*}
    \|s-\tau\|_{C^3(g_d,h)}
    &\leq C_{\mathrm{pert}}\eta_{\mathrm{tot}},\\
    |s|&\geq\sqrt b-C_{\mathrm{pert}}\eta_{\mathrm{tot}}
    \geq\frac{\sqrt b}{2}=:c_b.
\end{align*}
Their rescaled $C^3$ norms are uniformly bounded, independently of $R$ and $M$.
For any pairs $s=(s_0,s_1)$ and $v=(v_0,v_1)$, we have
\begin{align*}
    W(s+v)-W(s)=v_0\nabla s_1-v_1\nabla s_0+s_0\nabla v_1-s_1\nabla v_0+v_0\nabla v_1-v_1\nabla v_0.
\end{align*}
Thus, for the uniformly bounded pairs and perturbations under consideration,
\begin{align}\label{Equation: Wronskian perturbation bound}
    d^{-\frac12}|W(s+v)-W(s)|(y)\leq C_W\bigl(|v(y)|+d^{-\frac12}|\nabla v(y)|\bigr).
\end{align}
Together with \eqref{Equation: 15(2)}, this implies that a perturbation of amplitude $\eta$ changes the normalized Wronskian by at most $C_{\mathrm{later}}\eta$ everywhere. For $p\in\Lambda_d^\mu$, a perturbation of amplitude $\eta$ whose centres lie in $\Lambda_d^\mu\smallsetminus\{p\}$ changes the normalized Wronskian at $p$ by at most $C_{\mathrm{tail}}\eta e^{-cR^2}$, by \eqref{Equation: same class Gaussian tail}. The constants $C_{\mathrm{later}}$ and $C_{\mathrm{tail}}$ can therefore be fixed independently of $R$ and $M$.
Set $c_{\mathrm{self}}:=c_b/\sqrt2$. Choose $R>1$ so large that
\begin{align*}
    C_{\mathrm{tail}}e^{-cR^2}\leq\frac{c_{\mathrm{self}}}{2},
\end{align*}
and fix the corresponding decomposition and its number $M$. Next choose $\kappa\in(0,\frac12)$ so small that
\begin{align*}
    2C_{\mathrm{later}}\kappa\leq\frac{c_{\mathrm{self}}}{4}.
\end{align*}
Define
\begin{align*}
    \eta_\mu:=(1-\kappa)\eta_{\mathrm{tot}}\kappa^{\mu-1}, \qquad \mu=1,\ldots,M.
\end{align*}
Then
\begin{align*}
    \sum_{\mu=1}^M\eta_\mu&\leq\eta_{\mathrm{tot}},\\
    \sum_{\nu>\mu}\eta_\nu
    &\leq\frac{\kappa}{1-\kappa}\eta_\mu
    \leq2\kappa\eta_\mu.
\end{align*}
All these choices are independent of $d$. Increase its lower threshold, if necessary, so that the preceding estimates hold for the chosen $R$.
Set $s_i^{(0)}:=\tau_i$ and proceed recursively over the classes. Suppose that $s^{(\mu-1)}$ has been constructed. For each $p\in\Lambda_d^\mu$, use $e_p$ and $dz\otimes e_p^{\otimes2}$ to identify the relevant fibres isometrically with $\c$, and put
\begin{align*}
    w_p=(w_p^0,w_p^1)&:=\bigl(s_0^{(\mu-1)}(p),s_1^{(\mu-1)}(p)\bigr),\\
    Z_p&:=d^{-\frac12}W(s^{(\mu-1)})(p).
\end{align*}
The preceding perturbation bounds give $|w_p|\geq c_b$. Define
\begin{align*}
    \zeta_p:=
    \begin{cases}
        Z_p/|Z_p|,& Z_p\neq0,\\
        1,& Z_p=0,
    \end{cases}
    \qquad
    (\gamma_p,\theta_p)
    :=
    \zeta_p
    \frac{(-\overline{w_p^1},\,\overline{w_p^0})}
         {|w_p^0|+|w_p^1|}.
\end{align*}
The denominator is nonzero, and
\begin{align*}
    |\gamma_p|+|\theta_p|&=1,\\
    w_p^0\theta_p-w_p^1\gamma_p
    &=\zeta_p
       \frac{|w_p|^2}{|w_p^0|+|w_p^1|}.
\end{align*}
Now set
\begin{align*}
    s_0^{(\mu)}
    &:=s_0^{(\mu-1)}
       +\eta_\mu\sum_{p\in\Lambda_d^\mu}\gamma_p\sigma_p^1,\\
    s_1^{(\mu)}
    &:=s_1^{(\mu-1)}
       +\eta_\mu\sum_{p\in\Lambda_d^\mu}\theta_p\sigma_p^1,
\end{align*}
and write $W^{(\mu)}:=W(s^{(\mu)})$. These sections are holomorphic, and the bounds established above apply at every stage of the recursion.
Fix $p\in\Lambda_d^\mu$ and first consider the pair obtained by adding only the peak centred at $p$:
\begin{align*}
    s^{[p]}:=s^{(\mu-1)}
       +\eta_\mu(\gamma_p,\theta_p)\sigma_p^1.
\end{align*}
Since $\sigma_p^1(p)=0$ and $\nabla\sigma_p^1(p)=\sqrt d\,dz\otimes e_p$, we have the exact identity
\begin{align*}
    d^{-\frac12}W(s^{[p]})(p)=Z_p+\eta_\mu\bigl(w_p^0\theta_p-w_p^1\gamma_p\bigr).
\end{align*}
Our choice of $\zeta_p$ therefore gives
\begin{align*}
    d^{-\frac12}|W(s^{[p]})(p)|
    &=|Z_p|+\eta_\mu
       \frac{|w_p|^2}{|w_p^0|+|w_p^1|}\\
    &\geq\frac{|w_p|}{\sqrt2}\eta_\mu\\
    &\geq c_{\mathrm{self}}\eta_\mu.
\end{align*}
The remaining perturbation is
\begin{align*}
    s^{(\mu)}-s^{[p]}
    =
    \eta_\mu
    \sum_{\substack{q\in\Lambda_d^\mu\\q\neq p}}
       (\gamma_q,\theta_q)\sigma_q^1.
\end{align*}
By \eqref{Equation: same class Gaussian tail} and \eqref{Equation: Wronskian perturbation bound},
\begin{align*}
    d^{-\frac12}
    |W^{(\mu)}(p)-W(s^{[p]})(p)|
    \leq C_{\mathrm{tail}}\eta_\mu e^{-cR^2}.
\end{align*}
The choice of $R$ yields
\begin{align*}
    |W^{(\mu)}(p)|
    \geq\frac{c_{\mathrm{self}}}{2}\eta_\mu\sqrt d
    \qquad\forall\,p\in\Lambda_d^\mu.
\end{align*}
It remains to control the changes caused by subsequent classes. For a point $p\in\Lambda_d^\mu$,
\begin{align*}
    d^{-\frac12}|W^{(M)}(p)-W^{(\mu)}(p)|
    &\leq C_{\mathrm{later}}\sum_{\nu>\mu}\eta_\nu\\
    &\leq2C_{\mathrm{later}}\kappa\eta_\mu\\
    &\leq\frac{c_{\mathrm{self}}}{4}\eta_\mu.
\end{align*}
Consequently, for the final sections
\begin{align*}
    s_i:=s_i^{(M)},\qquad i=0,1,
\end{align*}
their Wronskian $W:=W(s_0,s_1)$ satisfies
\begin{align*}
    |W(p)|
    \geq\frac{c_{\mathrm{self}}}{4}\eta_\mu\sqrt d
    \geq a_W\sqrt d
    \qquad\forall\,p\in\Lambda_d^\mu,   
\end{align*}
with
\begin{align*}
    a_W:=\frac{c_{\mathrm{self}}}{4}\eta_M>0.
\end{align*}
Thus $|W(p)|\geq a_W\sqrt d$ at every point of $\Lambda_d$.
The total perturbation bound gives
\begin{align*}
    |s_0|_h^2+|s_1|_h^2\geq\frac b4\qquad\mbox{on }\Sigma.
\end{align*}
It also gives, after increasing a constant $A$ depending only on the background geometry,
\begin{align}\label{Equation: 21}
    |s_0|_h^2+|s_1|_h^2\leq A\qquad\mbox{on }\Sigma
\end{align}
and
\begin{align}\label{Equation: 22}
    |\nabla^js_i|_{g_0,h}\leq Ad^{\frac{j}{2}},\qquad i=0,1,\quad j=1,2,3.
\end{align}
In particular,
\begin{align*}
    |\nabla W|\leq 2|\nabla s_0||\nabla s_1|+|s_0||\nabla^2s_1|+|s_1||\nabla^2s_0|\leq C_{\mathrm{der}}d
\end{align*}
for a uniform constant $C_{\mathrm{der}}>0$.
Choose
\begin{align*}
    0<\rho\leq
    \min\left\{\frac14,\frac{a_W}{2C_{\mathrm{der}}}\right\}.
\end{align*}
For $p\in\Lambda_d$ and $y\in B_{\rho d^{-\frac12}}(p)$, integration along a minimizing geodesic gives
\begin{align*}
    |W(y)|\geq |W(p)|-C_{\mathrm{der}}d\,\dist_{g_0}(p,y)\geq(a_W-C_{\mathrm{der}}\rho)\sqrt d\geq\frac{a_W}{2}\sqrt d.
\end{align*}
Finally, for each $x\in\Sigma$, choose $p_x\in\Lambda_d$ with $\dist_{g_0}(x,p_x)\leq\delta d^{-\frac12}$. Since $\delta=1/4$ and $\rho\leq1/4$,
\begin{align*}
    p_x&\in B_{\frac12d^{-\frac12}}(x),\\
    B_{\rho d^{-\frac12}}(p_x)&\subset B_{d^{-\frac12}}(x).
\end{align*}
Taking
\begin{align*}
    a:=\min\left\{\frac b4,\frac{a_W}{2}\right\}
\end{align*}
proves all three assertions. All parameters were chosen using only the background geometry. Increasing $d_0$ to include the preceding thresholds completes the proof.
\end{proof}
\subsection{Polynomial transversality and balanced covers}
The preceding construction guarantees that the Wronskian is large on a definite subset of every ball of radius $d^{-\frac12}$, but it does not prevent it from vanishing elsewhere. To obtain the final non-degeneracy condition in the definition of a balanced cover, we perturb the pencil once more so that the Wronskian and its first derivative cannot become simultaneously too small. The next lemma establishes this global polynomial transversality.
\begin{lemma}\label{Lemma: polynomially transverse quantitative pencil lemma}
Let $(L,h)\to\Sigma$ be a holomorphic Hermitian line bundle of degree $d$, and let $\nabla$ be its Chern connection. Assume that
\begin{align}
    \frac{i}{2\pi}F_{\nabla}=\lambda_d\omega_0,\qquad\lambda_d:=\frac{d}{\vol_{g_0}(\Sigma)}.
\end{align}
Then there exist constants $a,A,\rho>0$ and $d_0\ge1$, depending only on $(\Sigma,g_0)$, such that for every $d\ge d_0$ there exist holomorphic sections $s_{0},s_{1}\in H^0(L)$ satisfying the following conditions.
\begin{enumerate}[(i)]
    \item  $a\le|s_{0}|_{h}^2+|s_{1}|_{h}^2\le A$ on $\Sigma$.
    \smallskip
    \item For $i=0,1$ and $j=1,2,3$, we have
    \begin{align*}
        |\nabla^j s_i|_{g_0,h}\leq A d^{\frac j2}.
    \end{align*}
    \smallskip
    \item Letting
        \begin{align*}
            W:=W(s_0,s_1)=s_{0}\nabla s_{1}-s_{1}\nabla s_{0},
        \end{align*}
        for every $x\in\Sigma$ there exists a point $p_x\in B_{\frac{1}{2} d^{-\frac{1}{2}}}(x)$ such that $B_{\rho d^{-\frac{1}{2}}}(p_x)\subset B_{d^{-\frac{1}{2}}}(x)$ and
        \begin{align*}
            |W|_{g_0,h}\ge a\sqrt d \qquad\mbox{ on }\, B_{\rho d^{-\frac{1}{2}}}(p_x).
        \end{align*}
    \item The Wronskian $W$ is polynomially transverse to the zero section, i.e.
        \begin{align} \label{Equation: polynomial W transversality}
            |W|_{g_0,h}^2+d^{-1}|\nabla W|_{g_0,h}^2\ge a\,d^{-61} \qquad\mbox{ on }\,\Sigma.
        \end{align}
\end{enumerate}
\end{lemma}
\begin{proof}
    Start with a pair $s_0,s_1\in H^0(L)$ given by Lemma \ref{Lemma: quantitative pencil lemma}. Let $a,A,\rho>0$ be such that (i), (ii), and (iii) in \ref{Lemma: quantitative pencil lemma} hold for the couple $s_0,s_1$. 

    \medskip
    \noindent
    \textbf{Step 1: a one-sided probabilistic perturbation}.  We prove the following fact. Suppose that $r,q\in H^0(L)$ satisfy
    \begin{align}\label{Equation: initial derivatives bound}
        |\nabla^jr|_{g_0,h}+|\nabla^jq|_{g_0,h}
        \leq C_1d^{\frac j2}
        \qquad\forall\,j=0,1,2,3.
    \end{align}
    and %
    \begin{align}\label{Equation: q lower}
        |q|_h\ge c_0>0 \qquad\mbox{ on an arbitrary subset }\,\Omega\subset\Sigma.
    \end{align}
    Fix $K\ge3$ and set $\lambda=d^{-K}$. If $M>2K+4$, then there exists $u\in H^0(L)$ with $\|u\|_{L^2(\Sigma)}\le 2$ such that, for
    \begin{align*}
        W_u:=q\nabla(r+\lambda u)-(r+\lambda u)\nabla q,
    \end{align*}
    one has
    \begin{align}\label{Equation: local random transversality}
        d^{-1}|W_u|_{g_0,h}^2+d^{-2}|\nabla W_u|_{g_0,h}^2\ge d^{-2M} \qquad\mbox{ on }\,\Omega.
    \end{align}
    To prove the assertion, set $N_d:=\dim_\mathbb C H^0(L)$. By the Riemann--Roch theorem, there exists $C_2>0$ depending only on the topology of $\Sigma$ such that
    \begin{align}\label{Equation: dimension H0}
        N_d\le C_2d.
    \end{align}
    Choose an $L^2$-orthonormal basis $\{e_\alpha\}_{\alpha=1,\dots,N_d}$ of $H^0(L)$ and let
    \begin{align*}
        U:=\frac{1}{\sqrt{N_d}}\sum_{\alpha=1}^{N_d}g_\alpha e_\alpha,
    \end{align*}
    where the $g_\alpha$ are independent standard complex Gaussian variables. Then
    \begin{align*}
        \mathbb E\Big[\|U\|_{L^2(\Sigma)}^2\Big]=1
    \end{align*}
    and consequently, by Markov's inequality,
    \begin{align}\label{Equation: Gaussian L2 event}
        \mathbb P\big(\|U\|_{L^2(\Sigma)}>2\big)=\mathbb P\big(\|U\|_{L^2(\Sigma)}^2>4\big)\le\frac{\mathbb E\Big[\|U\|_{L^2(\Sigma)}^2\Big]}{4}=\frac{1}{4}.
    \end{align}
    For $x\in\Omega$, define the $\c$-linear map $\mathcal A_x:H^0(L)\rightarrow\c^2$ given by
    \begin{align*}
        \mathcal A_xu:=\left(d^{-\frac{1}{2}}\big(q\nabla u-u\nabla q\big)(x),d^{-1}\nabla\big(q\nabla u-u\nabla q\big)(x)\right),
    \end{align*}
    after making unitary identifications of the two target fibres with
    $\mathbb C$. We claim that $\mathcal A_x$ admits a right inverse $R_x$ satisfying
    \begin{align}\label{Equation: right inverse jet map}
        \|R_x\|_{\c^2\to L^2}\le C_3d.
    \end{align}
    Let $(z_0,z_1)\in\c^2$. Apply the sharp constant-curvature jet construction
    \eqref{Equation: sharp constant curvature jet peaks}, to prescribe
    \begin{align*}
        u(x)=0,\qquad
        \nabla^{1,0}u(x)=\sqrt d\,q(x)^{-1}z_0,\qquad
        (\nabla^2u)^{2,0}(x)=d\,q(x)^{-1}z_1.
    \end{align*}
    As throughout this argument, the target fibres are identified unitarily
    with $\c$. The natural rescaled norm of this jet is at most
    $C(|z_0|+|z_1|)$, because $|q(x)|\geq c_0$.
    The construction is complex linear in the prescribed jet and gives
    \begin{align}\label{Equation: two jet peak bound}
        |\nabla^j u(y)|_{g_0,h}
        \leq C_j d^{\frac{j}{2}}(|z_0|+|z_1|)
                  e^{-c d\dist_{g_0}^2(x,y)}.
    \end{align}
    At $x$, the terms containing $u(x)$ vanish and the mixed first-derivative
    terms cancel in the holomorphic second derivative, since the source has
    complex dimension one. Hence $\mathcal A_xu=(z_0,z_1)$.
    Furthermore,
    \begin{align*}
        \|u\|_{L^2(\Sigma)}
        \leq C d^{-\frac12}(|z_0|+|z_1|)
        \leq C' d\,\|(z_0,z_1)\|_{\c^2}.
    \end{align*}
    This proves \eqref{Equation: right inverse jet map}.
    It follows that
    \begin{align*}
        \mathcal A_x\mathcal A_x^*\ge c_1d^{-2}\Id_{\c^2}.
    \end{align*}
    for some $c_1:=C_3^{-2}>0$. Therefore, the complex random vector $\lambda\mathcal A_xU$ has covariance bounded below by
    \begin{align*}
        \operatorname{Cov}(\lambda\mathcal A_xU)=\frac{\lambda^2}{N_d}\mathcal A_x\mathcal A_x^*\ge c_2\lambda^2d^{-3}\Id_{\c^2},
    \end{align*}
    with $c_2:=\frac{c_1}{C_2}>0$. The density of the random vector $\lambda\mathcal A_xU$ is consequently bounded above by $C_6\lambda^{-4}d^6$ for $C_6:=\frac{1}{\pi^2c_2^2}>0$. Thus, for every deterministic centre $v\in\c^2$ and every $\delta>0$,
    \begin{align}\label{Equation: fixed point small ball}
        \mathbb P\big(|v+\lambda\mathcal A_xU|\le 2\delta\big)\le C_7\delta^4\lambda^{-4}d^6.
    \end{align}
    where $C_7>0$ is a uniform constant. Set $\delta:=d^{-M}$. Recall that, by standard $L^2$--$C^j$ Cauchy estimates, there are constants $K_j$, independent of $d$, such that 
    \begin{align}\label{Equation: natural scale Cauchy}
        \|\nabla^j\tau\|_{L^\infty(\Sigma)}\le K_jd^{\frac{j+1}{2}}\|\tau\|_{L^2(\Sigma)} \qquad\forall\,\tau\in H^0(L),\quad 0\le j\le 3.
    \end{align}
    On the event $\|U\|_{L^2(\Sigma)}\le 2$, the estimates \eqref{Equation: natural scale Cauchy} and \eqref{Equation: initial derivatives bound} give
    \begin{align*}
        |\nabla W_U|_{g_0,h}\le C_8d,\qquad |\nabla^2W_U|_{g_0,h} \le C_8d^{\frac{3}{2}}.
    \end{align*}
    for $C_8>0$. Consequently, setting
    \begin{align*}
        \mathcal T_U:=\bigl(d^{-\frac12}W_U,d^{-1}\nabla W_U\bigr),
    \end{align*}
    the function $|\mathcal T_U|_{g_0,h}$ is
    $C_8\sqrt d$-Lipschitz. Let $r_d:=\frac{\delta}{4C_8\sqrt d}$ and choose a maximal $r_d$-separated subset $\Lambda\subset\Omega$. By construction, $\Omega$ is covered by balls of radius $r_d$ centred at points in $\Lambda$ and 
    \begin{align}\label{Equation: fine random net size}
        \#\Lambda\le C_9r_d^{-2}\le C_{10}d\,\delta^{-2},
    \end{align}
    with $C_{10}:=16C_8^2C_9>0$. At each point $p\in\Lambda$, the value $\mathcal T_U(p)$ is a deterministic vector plus $\lambda\mathcal A_pU$. Hence \eqref{Equation: fixed point small ball} and \eqref{Equation: fine random net size} give
    \begin{align*}
        \mathbb P\Big(|\mathcal T_U(p)|_{g_0,h}\le2\delta\text{ for some }p\in\Lambda\Big)&\le C_{11}d\delta^{-2}\delta^4\lambda^{-4}d^6=C_{11}d^{7-2M+4K}.
    \end{align*}
    By our choice of $M$, this tends to zero. For sufficiently large
    $d$, it is smaller than $\frac{1}{4}$. Combining this with \eqref{Equation: Gaussian L2 event}, there exists a realization $u$ such that $\|u\|_{L^2(\Sigma)}\le 2$
    and
    \begin{align*}
        |\mathcal T_u(p)|_{g_0,h}>2\delta \qquad\forall\,p\in\Lambda.
    \end{align*}
    The Lipschitz bound on $\mathcal T_u$ and the definition of $r_d$ then imply
    \begin{align*}
        |\mathcal T_u(x)|_{g_0,h}\ge\delta \qquad\forall\, x\in\Omega,
    \end{align*}
    which is precisely \eqref{Equation: local random transversality}.

    \medskip
    \noindent
    \textbf{Step 2: the first perturbation}.
    Let $b_*:=\sqrt{a}>0$ and define
    \begin{align*}
        \Omega_0:=\left\{x\in\Sigma \mbox{ : } |s_0(x)|_h\ge\frac{b_*}{2}\right\}.
    \end{align*}
    Apply Step 1 with
    \begin{align*}
        \Omega=\Omega_0,\qquad
        q=s_0,\qquad
        r=s_1,\qquad
        K=3,\qquad
        M=11.
    \end{align*}
    There exists $u_1\in H^0(L)$ with $\|u_1\|_{L^2(\Sigma)}\le 2$ such that, setting
    \begin{align*}
        s_0^{(1)}:=s_0,\qquad
        s_1^{(1)}:=s_1+d^{-3}u_1,
    \end{align*}
    and
    \begin{align*}
        W^{(1)}:=s_0^{(1)}\nabla s_1^{(1)}-s_1^{(1)}\nabla s_0^{(1)},
    \end{align*}
    we have
    \begin{align}\label{Equation: first stage transverse}
        d^{-1}|W^{(1)}|_{g_0,h}^2+d^{-2}|\nabla W^{(1)}|_{g_0,h}^2\ge d^{-22} \qquad\mbox{ on }\,\Omega_0.
    \end{align}
    Now let
    \begin{align*}
        \Omega_1:=\Sigma\smallsetminus\Omega_0.
    \end{align*}
    Note that, by the properties of the pair $(s_0,s_1)$, we must have    
    \begin{align*}
        |s_1(x)|_h\ge\frac{\sqrt3}{2}b_* \qquad\forall\,x\in\Omega_1.
    \end{align*}
    By \eqref{Equation: natural scale Cauchy},
    \begin{align*}
        \|d^{-3}u_1\|_{L^\infty(\Sigma)}\le C_{12}d^{-\frac{5}{2}}.
    \end{align*}
    Therefore, for all sufficiently large $d$,
    \begin{align}\label{Equation: second stage coefficient lower}
        \big|s_1^{(1)}\big|_h\ge\frac{b_{*}}{2}\qquad\mbox{ on }\,\Omega_1.
    \end{align}

    \medskip
    \noindent
    \textbf{Step 3: the second perturbation}.
    Apply Step 1 with
    \begin{align*}
        \Omega=\Omega_1,\qquad
        q=s_1^{(1)},\qquad
        r=s_0^{(1)},\qquad
        K=13,\qquad
        M=31.
    \end{align*}
    There exists $u_0\in H^0(L)$ with $\|u_0\|_{L^2(\Sigma)}\le 2$ such that, setting
    \begin{align*}
        \tilde s_0:=s_0^{(1)}+d^{-13}u_0,\qquad \tilde s_1:=s_1^{(1)},
    \end{align*}
    the final Wronskian
    \begin{align*}
        W:=\tilde s_0\nabla \tilde s_1-\tilde s_1\nabla \tilde s_0
    \end{align*}
    satisfies
    \begin{align}\label{Equation: second stage transverse}
        d^{-1}|W|_{g_0,h}^2+d^{-2}|\nabla W|_{g_0,h}^2\ge d^{-62} \qquad\mbox{ on }\,\Omega_1.
    \end{align}
    The second perturbation changes the normalized jet by at most
    \begin{align}\label{Equation: second perturbation normalized change}
        \left|\big(d^{-\frac{1}{2}}(W-W^{(1)}),d^{-1}(\nabla W-\nabla W^{(1)})\big)\right|_{g_0,h}\le C_{12}d^{-13}\sqrt d=C_{13}d^{-\frac{25}{2}}.
    \end{align}
    Since $d^{-\frac{25}{2}}=o(d^{-11})$, the lower bound \eqref{Equation: first stage transverse} remains valid, with one half of its original constant, on $\Omega_0$. Since $d^{-11}\gg d^{-31}$, combining \eqref{Equation: first stage transverse} and \eqref{Equation: second stage transverse} gives
    \begin{align}\label{Equation: normalized global transversality}
        d^{-1}|W|_{g_0,h}^2+d^{-2}|\nabla W|_{g_0,h}^2\ge c d^{-62} \qquad\mbox{ on }\,\Sigma.
    \end{align}
    Multiplying by $d$ proves \eqref{Equation: polynomial W transversality}.

    \medskip
    \noindent
    \textbf{Step 4: preservation of the previous conclusions}.
    By \eqref{Equation: natural scale Cauchy},
    \begin{align*}
        \|d^{-3}u_1+d^{-13}u_0\|_{L^\infty(\Sigma)}\le C_{13}d^{-\frac{5}{2}}.
    \end{align*}
    From this, we conclude that $\tilde s_0,\tilde s_1$ still satisfy the bounds (i) and (ii) in Lemma \ref{Lemma: quantitative pencil lemma}. Moreover,
    \begin{align*}
        d^{-\frac{1}{2}}\|W-W_{s}\|_{L^\infty(\Sigma)}\le C_{14}(d^{-3}+d^{-13})\sqrt{d}=o(1),
    \end{align*}
    where
    \begin{align*}
        W_s:=s_0\nabla s_1-s_1\nabla s_0.
    \end{align*}
    Consequently, the lower bound for $W_s$ on small balls given by Lemma \ref{Lemma: quantitative pencil lemma}(iii) remains valid, after decreasing its constant. This completes the proof.
\end{proof}
We can now pass from the quantitative pencil to the desired holomorphic
cover. The uniform lower bound for $\lvert s_0\rvert_h^2+\lvert s_1\rvert_h^2$ ensures that $f=[s_0:s_1]$ is globally defined, while the two Wronskian estimates obtained above yield respectively the averaged energy bound and the polynomial non-degeneracy condition.
\begin{proposition}\label{Proposition: existence of balanced covers} 
There exists $d_0\in\n$ such that for every $d\ge d_0$ there exists a balanced branched holomorphic cover $f:\Sigma\to\cp^1$ of degree $d$.
\end{proposition}	
\begin{proof}
Choose a holomorphic line bundle $L_1\to\Sigma$ of degree one and set
\begin{align*}
    L:=L_1^{\otimes d}.
\end{align*}
Then $\deg(L)=d$. We equip $L$ with a Hermitian metric $h$ whose Chern
connection satisfies
\begin{align*}
    \frac{i}{2\pi}F_\nabla
    =
    \frac{d}{\vol_{g_0}(\Sigma)}\,\omega_0.
\end{align*}
Such a metric exists by the standard constant-curvature representative of the first Chern class. By Lemma \ref{Lemma: polynomially transverse quantitative pencil lemma}, for every sufficiently large $d$ there exist $a,A,\rho>0$ and holomorphic sections $s_{0},s_{1}\in H^0(L)$ such that
\begin{align}\label{Equation: 25}
    a\le |s_{0}|_{h}^2+|s_{1}|_{h}^2\le A,
\end{align}
\begin{align}\label{Equation: 23}
    |\nabla^j s_i|_{g_0,h}\le Ad^{\frac{j}{2}},\qquad i=0,1,\,j=1,2,3,
\end{align}
and $W:=s_{0}\nabla s_{1}-s_{1}\nabla s_{0}$ satisfies (iii) and (iv) in Lemma \ref{Lemma: polynomially transverse quantitative pencil lemma}. By \eqref{Equation: 25}, $s_0$ and $s_1$ have no common zeros. Therefore the map $f:\Sigma\to\cp^1$ given by
\begin{align*}
    f:=[s_0:s_1]
\end{align*}
is a well-defined holomorphic map. Moreover, $f^*\mathcal{O}_{\mathbb{CP}^1}(1)\cong L$, which implies that $\deg f=\deg L=d$. We now use the standard Fubini--Study formula. On the open set where $s_{0}\neq0$, write $w=\frac{s_{1}}{s_{0}}$. Then
\begin{align}\label{Equation: 26}
    dw=\frac{s_{0}\nabla s_{1}-s_{1}\nabla s_{0}}{s_{0}^2}=\frac{W}{s_{0}^2}.
\end{align}
Since $\lvert df\rvert_{g_0}^2$ is proportional to
\begin{align*}
    \frac{|dw|_{g_0}^2}{(1+|w|^2)^2},
\end{align*}
we obtain
\begin{align*}
    |df|_{g_0}^2\asymp\frac{|W|_{g_0,h}^2}{\left(|s_{0}|_{h}^2+|s_{1}|_{h}^2\right)^2}.
\end{align*}
The same formula holds on the open set where $s_1\neq0$,
hence globally on $\Sigma$. By \eqref{Equation: 25} and
\eqref{Equation: 23}, there is a uniform constant $C>0$ such that
\begin{align*}
    |W|_{g_0,h}\leq C\sqrt d,
    \qquad
    |df|_{g_0}^2\leq Cd.
\end{align*}
At each point, choose $i\in\{0,1\}$ with
$|s_i|_h\geq\sqrt{a/2}$ and use the corresponding target
coordinate $s_{1-i}/s_i$. This coordinate is uniformly bounded
at that point. Differentiating the quotient twice and using
\eqref{Equation: 25} and \eqref{Equation: 23} gives
$|\nabla^2f|_{g_0}\leq Cd$.

\smallskip
\noindent
We now prove the averaged lower bound. Fix $x\in\Sigma$.
Property~(iii) of
Lemma~\ref{Lemma: polynomially transverse quantitative pencil lemma}
and the Fubini--Study comparison give a uniform constant $c>0$
such that
\begin{align*}
    |df|_{g_0}^2\geq cd
    \qquad\mbox{on }B_{\rho d^{-\frac12}}(p_x).
\end{align*}
Since $g_0$ is smooth and $\Sigma$ is compact, there are
constants $c_1,C_1>0$ such that
\begin{align*}
    c_1r^2\leq\vol_{g_0}(B_r(y))\leq C_1r^2
\end{align*}
for every $y\in\Sigma$ and every sufficiently small $r>0$.
Using
$B_{\rho d^{-\frac12}}(p_x)\subset B_{d^{-\frac12}}(x)$,
we obtain
\begin{align*}
    \fint_{B_{d^{-\frac12}}(x)}
    |df|_{g_0}^2\,d\vol_{g_0}\geq cd\,\frac{\vol_{g_0}(B_{\rho d^{-\frac12}}(p_x))}
         {\vol_{g_0}(B_{d^{-\frac12}}(x))}\geq\frac{cc_1\rho^2}{C_1}d.
\end{align*}
Differentiating again in \eqref{Equation: 26}, we get
\begin{align}\label{Equation: derivative W df comparison}
    |\nabla W|_{g_0,h}\le C_2|\nabla^2f|_{g_0}+C_2\sqrt{d}|df|_{g_0}.
\end{align}
It follows that
\begin{align*}
    |W|_{g_0,h}^2+d^{-1}|\nabla W|_{g_0,h}^2\le C_3\left(|df|_{g_0}^2+d^{-1}|\nabla^2f|_{g_0}^2\right).
\end{align*}
Combining this with \eqref{Equation: polynomial W transversality} proves property (iii) in Definition \ref{Definition: m-balanced cover}. The statement follows.
\end{proof}	
%
%
\section{Regular families of branched superminimal immersions}
The purpose of this section is to show that the balanced Bor{\accent23 u}vka covers constructed from the balanced holomorphic covers constructed in Section~3 are regular points of the space of holomorphic horizontal maps. More precisely, let
\begin{align*}
    \tilde\varphi_n:\cp^1\rightarrow\mathscr Z_n
\end{align*}
be the twistor lift of the Bor{\accent23 u}vka sphere and let $f:\Sigma\to\cp^1$ be a balanced holomorphic cover of large degree $\ell$. For
\begin{align*}
    \psi:=\tilde\varphi_n\circ f,
\end{align*}
we shall prove the surjectivity of the linearized horizontality operator restricted to holomorphic sections. This is the principal regularity input in the proof of Theorem~\ref{Theorem: main statement 1}.

We begin by recalling the twistor geometry of $\s^{2n}$ and the intrinsic linearization of the horizontality equation. We then identify the relevant pullback bundles and construct a polynomially bounded holomorphic right inverse for $L_\psi$ using a homogeneous filtration and exact jet conditions at the branch points. Surjectivity follows from this construction, and a duality argument gives the projected spectral gap. Finally, the holomorphic implicit function theorem and Riemann--Roch yield the required regular families and their dimensions.
\subsection{Twistor spaces and branched superminimal immersions}
Recall that we denote by $(\,\cdot\,,\,\cdot\,)$ the complex-bilinear extension to $\c^{2n+1}$ of the Euclidean inner product on $\r^{2n+1}$, and by
\begin{align*}
    \langle v,w\rangle:=(v,\overline w)
\end{align*}
the associated Hermitian inner product. A complex subspace
$P\subset\c^{2n+1}$ is called \emph{totally isotropic} if
\begin{align*}
    (v,w)=0
    \qquad
    \forall\,v,w\in P.
\end{align*}
Let $P\subset\c^{2n+1}$ be a totally isotropic complex $n$-plane.
There is a unique real unit vector $x_P\in\s^{2n}$ orthogonal to
$P\oplus\overline P$ such that
\begin{align*}
    \bigl(x_P,\sqrt2\operatorname{Re}(e_1),
    \sqrt2\operatorname{Im}(e_1),\ldots,
    \sqrt2\operatorname{Re}(e_n),\sqrt2\operatorname{Im}e_n\bigr)
\end{align*}
is a positively oriented orthonormal basis of $\r^{2n+1}$, for any Hermitian orthonormal basis $\{e_1,\ldots,e_n\}$ of $P$. This condition is independent of that basis.
\begin{definition}[Twistor bundle]\label{Definition: twistor bundle}
    The \textit{twistor bundle} $\pi:\mathscr{Z}_n\to\s^{2n}$ of the sphere $\s^{2n}$ is given by
    \begin{align*}
        &\mathscr{Z}_n:=\big\{P\subset\c^{2n+1} \mbox{ is a totally isotropic } n\mbox{-subspace of } \c^{2n+1}\big\},\\
        &\pi(P):=x_P\in\s^{2n} \qquad\forall\,P\in\mathscr{Z}_n.
    \end{align*}
    The total space $\mathscr{Z}_n$ of the twistor bundle is called the \textit{twistor space} of $\s^{2n}$.
\end{definition}
\begin{remark}
    Recall that $\mathscr{Z}_n$ is a closed complex algebraic variety of dimension $\frac{n(n+1)}{2}$. In particular, $\mathscr{Z}_n$ is a closed complex submanifold of $\cp^{2^n-1}$ for every $n\ge 2$. Moreover, $\mathscr{Z}_2$ is biholomorphic to $\cp^3$. 
    More precisely, $\mathscr{Z}_n$ is a closed complex flag variety for every $n\ge 2$. Indeed, we have 
    \begin{align*}
        \mathscr{Z}_n=\frac{\operatorname{SO}(2n+1,\c)}{\operatorname{Stab}(V_0)} \qquad\forall\,n\ge 2,
    \end{align*}
    where $V_0\subset\c^{2n+1}$ is any reference maximal isotropic subspace of $\c^{2n+1}$. 
\end{remark}
\begin{definition}\label{Definition: universal subbundle}
    The holomorphic rank $n$ vector subbundle $\mathscr{S}_n\to\mathscr{Z}_n$ of the trivial bundle $\mathscr{Z}_n\times\c^{2n+1}$ over $\mathscr{Z}_n$ given by
    \begin{align*}
        \mathscr{S}_n:=\big\{(P,v) \mbox{ : } P\in\mathscr{Z}_n,\,v\in P\big\}
    \end{align*}
    is called the \textit{universal subbundle} over $\mathscr{Z}_n$. We let $\mathscr{S}_n^{\vee}\to\mathscr{Z}_n$ be the dual bundle of $\mathscr{S}_n$, i.e.
    \begin{align*}
        \mathscr{S}_n^{\vee}:=\big\{(P,\phi) \mbox{ : } P\in\mathscr{Z}_n,\,\phi\in P^*:=\operatorname{Hom}_{\c}(P,\c)\big\}
    \end{align*}
   Lastly, let $\mathscr S_n^\perp\subset\mathscr Z_n\times\mathbb C^{2n+1}$ be the holomorphic subbundle of the trivial bundle $\mathscr{Z}_n\times\c^{2n+1}$ over $\mathscr{Z}_n$ given by
    \begin{align*}
        \mathscr{S}_n^{\perp}:=\big\{(P,v) \mbox{ : } P\in\mathscr{Z}_n,\,v\in P^{\perp}\big\},
    \end{align*}
    where $P^{\perp}$ denotes the complex bilinear orthogonal complement of $P$. We define the holomorphic quotient line bundle
    \begin{align*}
        \mathscr L_n:=\mathscr S_n^\perp/\mathscr S_n.
    \end{align*}
    For every $P\in\mathscr Z_n$, the class of $x_P$ in $P^\perp/P$ is non-zero and determines a smooth complex-linear identification
    \begin{align*}
        \operatorname{span}_{\mathbb C}\{x_P\}\cong(\mathscr L_n)_P.
    \end{align*}
\end{definition}
The universal subbundle gives a convenient description of the holomorphic tangent bundle of $\mathscr Z_n$. The resulting exact sequence separates two geometrically different types of infinitesimal motion: the vertical directions deform the maximal isotropic plane while keeping its projection to $\s^{2n}$ fixed, whereas the horizontal directions describe motion of the projected point in the sphere.
\begin{lemma}\label{Lemma: short exact sequence}
    There exists a short exact sequence of holomorphic vector bundles
    \begin{align*}
        0\rightarrow\mathscr{S}_n^{\vee}\otimes\mathscr{L}_n\rightarrow T^{1,0}\mathscr{Z}_n\overset{q}{\rightarrow}\wedge^2\mathscr{S}_n^{\vee}\rightarrow 0.
    \end{align*}
\end{lemma}
\begin{proof}
    Fix $P\in\mathscr Z_n$. The tangent space to the Grassmannian at
    $P$ is naturally identified with
    \begin{align*}
        \operatorname{Hom}_{\mathbb C}
        \big(
            P,
            \mathbb C^{2n+1}/P
        \big).
    \end{align*}
    Under this identification, $T_P^{1,0}\mathscr Z_n$ consists of
    those homomorphisms $A$ for which
    \begin{align*}
        (\tilde A v,w)+(v,\tilde A w)=0 \qquad\forall\,v,w\in P,
    \end{align*}
    where $\tilde A:P\to\mathbb C^{2n+1}$ is any lift of $A$. The preceding condition is independent of the chosen lift because
    $P$ is totally isotropic.
    Define
    \begin{align*}
        q_P:
        T_P^{1,0}\mathscr Z_n
        \rightarrow
        \wedge^2P^\vee,
        \qquad
        q_P(A)(v,w)
        :=
        (\tilde A v,w).
    \end{align*}
    The defining relation for $T_P^{1,0}\mathscr Z_n$ shows that
    $q_P(A)$ is skew-symmetric. Moreover,
    \begin{align*}
        q_P(A)=0
    \end{align*}
    if and only if the image of $A$ is contained in
    $P^\perp/P$. Therefore,
    \begin{align*}
        \ker q_P
        \cong
        \operatorname{Hom}_{\mathbb C}
        \big(
            P,
            P^\perp/P
        \big)
        \cong
        P^\vee\otimes(\mathscr L_n)_P.
    \end{align*}
    Since
    \begin{align*}
        \dim_{\mathbb C}T_P^{1,0}\mathscr Z_n
        =
        \frac{n(n+1)}2,
    \end{align*}
    the map $q_P$ is surjective. These fibrewise maps depend
    holomorphically on $P$, and hence give the asserted holomorphic
    exact sequence.
\end{proof}
\begin{remark}[Horizontal and vertical distributions on the twistor bundle]
    We define the holomorphic horizontal subbundle by
    \begin{align*}
        H\mathscr Z_n
        :=
        \ker q
        \cong
        \mathscr S_n^\vee\otimes\mathscr L_n,
    \end{align*}
    and the holomorphic vertical bundle by
    \begin{align*}
        V\mathscr Z_n
        :=
        \wedge^2\mathscr S_n^\vee.
    \end{align*}
    Thus,
    \begin{align*}
        0
        \rightarrow
        H\mathscr Z_n
        \rightarrow
        T^{1,0}\mathscr Z_n
        \overset{q}{\rightarrow}
        V\mathscr Z_n
        \rightarrow
        0
    \end{align*}
    is a short exact sequence of holomorphic vector bundles.
    On the other hand, the differential of the twistor projection
    determines the smooth complex subbundle
    \begin{align*}
        V_{\mathrm{geom}}\mathscr Z_n:=\ker(d\pi\otimes_{\r}\c)\cap T^{1,0}\mathscr Z_n.
    \end{align*}
    The restriction
    \begin{align*}
        q|_{V_{\mathrm{geom}}\mathscr Z_n}:
        V_{\mathrm{geom}}\mathscr Z_n
        \rightarrow
        V\mathscr Z_n
    \end{align*}
    is a smooth complex-linear isomorphism. We use it to identify the
    two bundles smoothly and to transport the holomorphic structure of
    $V\mathscr Z_n$ to the geometric vertical bundle.
    With this convention, the bundle map
    \begin{align*}
        \Pi_{\operatorname V}:
        T^{1,0}\mathscr Z_n
        \rightarrow
        V\mathscr Z_n
    \end{align*}
    is the holomorphic quotient map $q$; under the smooth orthogonal
    decomposition
    \begin{align*}
        T^{1,0}\mathscr Z_n
        =
        H\mathscr Z_n
        \oplus
        V_{\mathrm{geom}}\mathscr Z_n,
    \end{align*}
    it agrees with the ordinary vertical projection. The displayed orthogonal decomposition is a smooth splitting and is not asserted to be holomorphic. Note that, when $n=2$, the exact sequence above is the standard holomorphic contact sequence on $\mathscr Z_2\cong\mathbb{CP}^3$.
\end{remark}
\begin{definition}\label{Definition: horizontality}
    A smooth map $\psi:\Sigma\to\mathscr Z_n$ is called
    \textit{horizontal} if
    \begin{align*}
        \tfrac12\bigl(d\psi(x)[v]-iJ_n d\psi(x)[v]\bigr)
        \in H_{\psi(x)}\mathscr Z_n
        \qquad\forall\,x\in\Sigma,\quad v\in T_x\Sigma,
    \end{align*}
    where $J_n$ denotes the complex structure of $\mathscr Z_n$.
\end{definition}
Equivalently, for a smooth map both type components must be horizontal:
\begin{align*}
    \Pi_{\operatorname V}(\partial\psi)=0,
    \qquad
    \Pi_{\operatorname V}(\bar\partial\psi)=0.
\end{align*}
The quotient map acts on the target factor in each expression.
For a holomorphic map the second equality is automatic, so its
horizontality is equivalent to
$\Pi_{\operatorname V}(\partial\psi)=0$.
This is the formulation used after imposing holomorphicity in the
deformation problem below.
\begin{definition}
    Let $\psi:\Sigma\to\mathscr{Z}_n$ be a smooth map. We say that $\psi$ is \textit{linearly full} if its image is not contained in any submanifold of the form
    \begin{align*}
        \mathscr{Z}_n^F:=\big\{W\oplus F\in\mathscr{Z}_n \mbox{ : } W\in\mathscr{Z}_r\big((F\oplus\overline{F})^{\perp}\big)\big\}\cong\mathscr{Z}_r,
    \end{align*}
    for some $(n-r)$-dimensional subspace $F$ of $\c^{2n+1}$ with $1\le r<n$. 
\end{definition}
Geometrically, $\mathscr Z_n^F$ is the twistor space associated with a proper even-dimensional totally geodesic subsphere of $\s^{2n}$. Thus, for a horizontal holomorphic curve $\psi$, linear fullness is equivalent to the condition that its projection $\pi\circ\psi$ is not contained in any proper totally geodesic subsphere. This explains the compatibility between the notion introduced above and the notion of linear fullness for superminimal immersions used in the introduction.
\begin{proposition}\label{Proposition: correspondence holo+hor with supermin}
    If $\psi:\Sigma\to\mathscr{Z}_n$ is a nonconstant horizontal holomorphic curve, then $\pi\circ\psi:\Sigma\to\s^{2n}$ is a branched superminimal immersion. Moreover, for every linearly full branched superminimal immersion $\varphi:\Sigma\to\s^{2n}$ there exists a unique linearly full horizontal holomorphic curve $\tilde\varphi:\Sigma\to\mathscr{Z}_n$ such that $\pi\circ\tilde\varphi$ is either $\varphi$ or $-\varphi$.
\end{proposition}
\begin{proof}
    See \cite{Fernandez1}.
\end{proof}
The possible antipodal sign in the preceding statement has no effect on the degree or on the local deformation theory. We shall therefore suppress it from the notation and refer to the corresponding horizontal holomorphic curve simply as the twistor lift of $\varphi$.
\begin{definition}[Twistor lift]
    Let $\varphi:\Sigma\to\s^{2n}$ be a linearly full branched superminimal immersion. The unique linearly full horizontal holomorphic curve $\tilde\varphi:\Sigma\to\mathscr{Z}_n$ given by Proposition \ref{Proposition: correspondence holo+hor with supermin} is called the \textit{twistor lift} of $\varphi$.
\end{definition}
\begin{definition}\label{Definition: degree}
    Let $\psi:\Sigma\to\mathscr{Z}_n$ be a holomorphic map. Let $\mathcal{O}_{\mathscr{Z}_n}(1)$ be the unique ample generator of the Picard group $\operatorname{Pic}(\mathscr{Z}_n)\cong\z$ of $\mathscr{Z}_n$. The \textit{degree} of $\psi$ is the non-negative integer $d\in\n$ given by
    \begin{align*}
        d:=\big\langle\psi_{*}[\Sigma],c_1\big(\mathcal{O}_{\mathscr{Z}_n}(1)\big)\big\rangle=\int_{\Sigma}\psi^*c_1\big(\mathcal{O}_{\mathscr{Z}_n}(1)\big).
    \end{align*}
    If $\varphi:\Sigma\to\s^{2n}$ is a linearly full branched superminimal immersion, the \textit{degree} of $\varphi$ is the degree of its unique twistor lift $\tilde\varphi$.
\end{definition}
\begin{remark}\label{Remark: degrees of pullback bundles}
	It can be computed explicitly that 
	\begin{align*}
		c_1(H\mathscr{Z}_n)&=2\,c_1\big(\mathcal{O}_{\mathscr{Z}_n}(1)\big)\\
		c_1(V\mathscr{Z}_n)&=2(n-1)\,c_1\big(\mathcal{O}_{\mathscr{Z}_n}(1)\big)\\
		c_1(T^{1,0}\mathscr{Z}_n)&=c_1(H\mathscr{Z}_n)+c_1(V\mathscr{Z}_n)=2n\,c_1\big(\mathcal{O}_{\mathscr{Z}_n}(1)\big)
	\end{align*}
	In particular, for every holomorphic map $\psi:\Sigma\to\mathscr{Z}_n$ of degree $d$, we have
	\begin{align*}
		\deg(\psi^*T^{1,0}\mathscr{Z}_n)=2nd \quad\mbox{ and }\quad\deg(\psi^*V\mathscr{Z}_n)=2(n-1)d.
	\end{align*}
\end{remark}
\begin{remark}\label{Remark: degree and area}
    We fix the metric normalization used on the twistor space.
    Equip $\operatorname{Gr}(n,\c^{2n+1})$ with its standard
    Hermitian symmetric metric, normalized so that a real tangent
    vector represented by $A:P\to P^{\perp_{\mathrm H}}$ has squared norm $\operatorname{tr}(A^\dagger A)$. Let $g_{\mathscr Z_n}$ be its restriction to $\mathscr Z_n$, and let $\omega_{\mathscr Z_n}$ be the corresponding K\"ahler form. The Grassmannian form represents $\pi$ times the first Chern class
    of the dual determinant of its tautological bundle. Restriction gives
    \begin{align*}
        [\omega_{\mathscr Z_n}]
        =\pi c_1(\det\mathscr S_n^\vee)
        =2\pi c_1\bigl(\mathcal O_{\mathscr Z_n}(1)\bigr).
    \end{align*}
    We use this invariant K\"ahler metric in the covariant derivatives
    and adjoints on the twistor space below. In particular, its Chern
    connection on $T^{1,0}\mathscr Z_n$ is torsion-free.
    To compare this metric with the unit round metric, let $u$ be a
    real horizontal tangent vector at $P$, represented by
    $A(e_i)=a_i x_P$ in a Hermitian orthonormal basis of $P$.
    Differentiate $(x_P,e_i)=0$ and $(x_P,x_P)=1$ along a curve
    representing $u$. This gives
    \begin{align*}
        d\pi_P(u)
        &=-\sum_{i=1}^n\bigl(a_i\overline e_i+\overline a_i e_i\bigr),\\
        |d\pi_P(u)|_{\s^{2n}}^2
        &=2\sum_{i=1}^n|a_i|^2
        =2|u|_{g_{\mathscr Z_n}}^2.
    \end{align*}
    Consequently, if $\varphi:\Sigma\to\s^{2n}$ is a linearly full
    branched superminimal immersion of degree $d$ and
    $\pi\circ\tilde\varphi=\pm\varphi$, then
    \begin{align*}
        \operatorname{Area}(\varphi)
        &=\operatorname{E}(\varphi)
        =2\operatorname{E}_{g_{\mathscr Z_n}}(\tilde\varphi)\\
        &=2\int_\Sigma\tilde\varphi^*\omega_{\mathscr Z_n}
        =4\pi d.
    \end{align*}
    Here the equality between the energy of the holomorphic lift and
    its K\"ahler area uses the convention
    $\operatorname{E}=\frac12\int|d\varphi|^2$.
    Thus the degree remains the area of the spherical map divided by
    $4\pi$.
\end{remark}
\subsection{Linearized horizontality along balanced \texorpdfstring{Bor{\accent23 u}vka}{Boruvka} covers}
Here and throughout we will use the following notation. We write
\begin{align*}
    \overline{\omega_\Sigma}:=(T^*\Sigma)^{0,1}.
\end{align*}
Fix $\alpha\in(0,1)$ and denote by $C_d^{1,\alpha}(\Sigma,\mathscr Z_n)$ the Banach manifold of $C^{1,\alpha}$ maps $\psi:\Sigma\to\mathscr Z_n$ of topological degree $d$. We also set
\begin{align*}
    \operatorname{H}_d(\Sigma,\mathscr Z_n)
    :=
    \big\{
        \psi:\Sigma\to\mathscr Z_n:
        \psi\text{ is holomorphic and }\deg(\psi)=d
    \big\}.
\end{align*}
Unless explicitly stated otherwise, we use the Hermitian metrics induced by the invariant K\"ahler metric $g_{\mathscr Z_n}$ fixed in
Remark~\ref{Remark: degree and area} and by $g_0$, together with the corresponding connections. The auxiliary metrics on the line-bundle summands used for the peak-section construction will be specified separately. At a zero of the constraint section, its differential is understood as its canonical vertical differential.
Let
\begin{itemize}
    \item $\HH_{d}(\Sigma,\mathscr{Z}_n):=\{\psi:\Sigma\to\mathscr{Z}_n \mbox{ holomorphic, horizontal, and of degree } d\}$;
    \item $\HH_{d}^f(\Sigma,\mathscr{Z}_n):=\{\psi:\Sigma\to\mathscr{Z}_n \mbox{ holomorphic, horizontal, linearly full, and of degree } d\}$.
\end{itemize}
\begin{remark}
    Note that, by standard elliptic regularity and by definition, we have the following chain of inclusions:
    \begin{align*}
        \HH_{d}^f(\Sigma,\mathscr{Z}_n)\subset\HH_{d}(\Sigma,\mathscr{Z}_n)\subset\H_{d}(\Sigma,\mathscr{Z}_n)\subset C^{\omega}(\Sigma,\mathscr{Z}_n).
    \end{align*}
\end{remark}
We let $\mathscr{E}_d^{0,\alpha}(\Sigma,\mathscr{Z}_n),\mathscr{F}_d^{0,\alpha}(\Sigma,\mathscr{Z}_n)$ be the smooth Banach bundles over the smooth Banach manifold $C_d^{1,\alpha}(\Sigma,\mathscr{Z}_n)$ given by 
\begin{align*}
    \mathscr{E}_d^{0,\alpha}(\Sigma,\mathscr{Z}_n)&:=\bigsqcup_{\psi\in C_d^{1,\alpha}(\Sigma,\mathscr{Z}_n)}C^{0,\alpha}(\omega_{\Sigma}\otimes\psi^*T^{1,0}\mathscr{Z}_n)\\
    \mathscr{F}_d^{0,\alpha}(\Sigma,\mathscr{Z}_n)&:=\bigsqcup_{\psi\in C_d^{1,\alpha}(\Sigma,\mathscr{Z}_n)}C^{0,\alpha}(\overline{\omega_{\Sigma}}\otimes\psi^*T^{1,0}\mathscr{Z}_n)
\end{align*}
Let $J_n$ be the complex structure on $\mathscr{Z}_n$. We identify real tangent vectors to $\mathscr Z_n$ with their $(1,0)$-parts by
$U\mapsto\frac12(U-iJ_nU)$. The expressions below are understood through this identification. Let the operators $\partial:C_d^{1,\alpha}(\Sigma,\mathscr{Z}_n)\to\mathscr{E}_d^{0,\alpha}(\Sigma,\mathscr{Z}_n)$ and $\bar\partial:C_d^{1,\alpha}(\Sigma,\mathscr{Z}_n)\to\mathscr{F}_d^{0,\alpha}(\Sigma,\mathscr{Z}_n)$ be the standard ``del`` and ``del bar'' operators on $C_d^{1,\alpha}(\Sigma,\mathscr{Z}_n)$, given by 
\begin{align*}
    \partial\psi&:=\frac{1}{2}\big(d\psi-J_n\circ d\psi\circ j\big),\\
    \bar\partial\psi&:=\frac{1}{2}\big(d\psi+J_n\circ d\psi\circ j\big),
\end{align*}
for every $\psi\in C_d^{1,\alpha}(\Sigma,\mathscr{Z}_n)$. Let
\begin{align*}
    q:T^{1,0}\mathscr Z_n\rightarrow V\mathscr Z_n
\end{align*}
be the holomorphic quotient map in Lemma~\ref{Lemma: short exact sequence}. For each $\psi\in C_d^{1,\alpha}(\Sigma,\mathscr Z_n)$, we write
\begin{align*}
    \Pi_{\operatorname V}:=\operatorname{id}_{\omega_\Sigma}\otimes\,\psi^*q:\omega_\Sigma\otimes\psi^*T^{1,0}\mathscr Z_n\rightarrow\omega_\Sigma\otimes\psi^*V\mathscr Z_n.
\end{align*}
Under the smooth identification of $V\mathscr Z_n$ with $V_{\mathrm{geom}}\mathscr Z_n$ established above, $q$ corresponds to the orthogonal vertical projection. Consider the smooth section 
\begin{align*}
    F:C_d^{1,\alpha}(\Sigma,\mathscr{Z}_n)&\to\mathscr{F}_d^{0,\alpha}(\Sigma,\mathscr{Z}_n)\oplus\mathscr{G}_d^{0,\alpha}(\Sigma,\mathscr{Z}_n)\\
    \psi&\mapsto(\bar\partial\psi,\Pi_{\operatorname{V}}(\partial\psi))
\end{align*}
where $\mathscr G_d^{0,\alpha}(\Sigma,\mathscr Z_n)$ is the smooth Banach bundle over $C_d^{1,\alpha}(\Sigma,\mathscr Z_n)$ given by
\begin{align*}
    \mathscr{G}_d^{0,\alpha}(\Sigma,\mathscr{Z}_n)&:=\bigsqcup_{\psi\in C_d^{1,\alpha}(\Sigma,\mathscr{Z}_n)}C^{0,\alpha}(\omega_{\Sigma}\otimes\psi^*V\mathscr{Z}_n).
\end{align*}
The two components of $F$ encode the two defining conditions of a
holomorphic horizontal map: the equation
\begin{align*}
    \bar\partial\psi=0
\end{align*}
imposes holomorphicity, while
\begin{align*}
    \Pi_{\operatorname V}(\partial\psi)=0
\end{align*}
imposes horizontality. Consequently,
\begin{align*}
    F^{-1}(0)=\HH_d(\Sigma,\mathscr Z_n).
\end{align*}
At a holomorphic horizontal map $\psi$, the section $F$ vanishes.
Its vertical differential is therefore intrinsic and defines
\begin{align*}
    D_\psi:C^{1,\alpha}(\psi^*T^{1,0}\mathscr Z_n)
    \rightarrow
    C^{0,\alpha}(\overline{\omega_\Sigma}\otimes\psi^*T^{1,0}\mathscr Z_n)
    \oplus C^{0,\alpha}(\omega_\Sigma\otimes\psi^*V\mathscr Z_n).
\end{align*}
We define $L_\psi$ to be the second component of this differential.
In holomorphic target coordinates and a holomorphic frame of
$V\mathscr Z_n$, write $Q$ for the matrix of the quotient map.
Then, in a source coordinate $z$,
\begin{align}\label{Equation: intrinsic local linearization}
    L_\psi\xi
    =\bigl(Q(\psi)\partial_z\xi
       +dQ_\psi[\xi]\partial_z\psi\bigr)dz,
    \qquad
    D_\psi\xi=(\bar\partial_{\psi^*T^{1,0}\mathscr Z_n}\xi,L_\psi\xi).
\end{align}
These formulas apply to all $C^{1,\alpha}$ sections $\xi$ along $\psi$.
They are independent of the target coordinates and quotient frame:
the additional terms in changing a trivialization of the nonlinear section
are proportional to its value at $\psi$, which is zero.
In particular, $L_\psi$ maps holomorphic sections to holomorphic sections.
We do not need to choose a vertical linearization at nonzero values of $F$.
For the torsion-free tangent connection fixed in Remark~\ref{Remark: degree and area}, the same operator has the covariant expression
\begin{align*}
    L_\psi\xi=\Pi_{\operatorname V}(\nabla^{1,0}\xi)+(\nabla_\xi q)(\partial\psi).
\end{align*}
Here $\nabla q$ is computed using the connection on $\operatorname{Hom}(T^{1,0}\mathscr Z_n,V\mathscr Z_n)$ induced by the Chern connections on its two factors, and is evaluated along $\psi$.
\begin{remark}
    Note that $\psi\in\HH_d(\Sigma,\mathscr{Z}_n)$ if and only if $F(\psi)=0$. Moreover, for every $\psi\in\HH_d(\Sigma,\mathscr{Z}_n)$ we have
    \begin{align*}
        \ker dF(\psi)=\ker D_\psi=\big\{\xi\in H^0(\psi^*T^{1,0}\mathscr{Z}_n) \mbox{ : } L_{\psi}\xi=0\big\}\subset H^0(\psi^*T^{1,0}\mathscr{Z}_n).
    \end{align*}
\end{remark}
Thus, at a holomorphic horizontal map $\psi$, the infinitesimal deformation problem reduces to the restriction
\begin{align*}
    L_\psi:
    H^0\big(\Sigma,\psi^*T^{1,0}\mathscr Z_n\big)
    \rightarrow
    H^0\big(
        \Sigma,
        \omega_\Sigma\otimes\psi^*V\mathscr Z_n
    \big).
\end{align*}
The central objective of the present section is to prove that this map is surjective when $\psi$ is a balanced Bor{\accent23 u}vka cover of sufficiently large degree. The $\operatorname{SO}(3)$-homogeneous geometry of the Bor{\accent23 u}vka spheres also underlies the work of Ball--Madnick~\cite[Section~2]{BallMadnick2026Morse}, who compute their Morse index and nullity and show that every normal Jacobi field along these spheres is induced by a twistor deformation.
\begin{lemma}\label{Lemma: pullback universal bundles}
    Let $\varphi_n:\mathbb S^2\rightarrow\mathbb S^{2n}$ be the superminimal immersion of constant Gaussian curvature corresponding to a special orbit of the irreducible representation $\operatorname{SO}(3)\rightarrow\operatorname{SO}(2n+1)$ and let $\tilde\varphi_n:\mathbb{CP}^1\rightarrow\mathscr Z_n$ be its unique twistor lift. Then
    \begin{align*}
        \tilde\varphi_n^*\mathscr S_n
        &\cong
        \mathcal O_{\cp^1}(-n-1)^{\oplus n},
        \\
        \tilde\varphi_n^*\mathscr L_n
        &\cong
        \mathcal O_{\cp^1}.
    \end{align*}
\end{lemma}

\begin{proof}
    Since the Bor{\accent23 u}vka--Veronese immersion is unique up to
    congruence and the twistor construction is
    $\operatorname{SO}(2n+1)$-equivariant, it suffices to prove the
    statement for the standard Veronese representative. Indeed, every
    $A\in\operatorname{SO}(2n+1)$ induces a holomorphic automorphism
    \begin{align*}
        \widehat A:\mathscr Z_n\rightarrow\mathscr Z_n,
        \qquad
        P\mapsto A_{\mathbb C}P,
    \end{align*}
    which preserves the universal bundles $\mathscr S_n$ and
    $\mathscr L_n$. Therefore, if
    $\varphi_n'=A\circ\varphi_n$, then its twistor lift satisfies
    \begin{align*}
        \tilde\varphi_n'
        =
        \widehat A\circ\tilde\varphi_n,
    \end{align*}
    and the corresponding pullback bundles are naturally isomorphic.
    Let $F_n:\cp^1\to\cp^{2n}$ be the directrix curve of $\varphi_n$. This curve is holomorphic and linearly full. Moreover, the directrix construction commutes with ambient rotations, so $F_n$ is $\operatorname{SO}(3)$-equivariant. We regard these actions as $\operatorname{SU}(2)$-actions through the standard double covering $\operatorname{SU}(2)\to\operatorname{SO}(3)$.
    Write
    \begin{align*}
        F_n^*\mathcal O_{\cp^{2n}}(1)\cong\mathcal O_{\cp^1}(m).
    \end{align*}
    Pullback of linear forms gives an $\operatorname{SU}(2)$-equivariant map
    \begin{align*}
        (\c^{2n+1})^\vee\rightarrow H^0\bigl(\mathcal O_{\cp^1}(m)\bigr),
    \end{align*}
    which is injective because $F_n$ is linearly full. The target is the irreducible representation $\operatorname{Sym}^m((\c^2)^\vee)$. The image is therefore the entire target, and comparison of dimensions gives $m+1=2n+1$. Thus $m=2n$, and $F_n$ is induced by the complete linear system $\lvert\mathcal O_{\cp^1}(2n)\rvert$. In particular, it is a rational normal curve of degree $2n$.
    For every $k=0,\ldots,n-1$, let $E_k\to\cp^1$ denote the
    $k$-th osculating bundle of $F_n$. Thus, if $z$ is a local
    holomorphic coordinate and
    \begin{align*}
        \xi:U\rightarrow\c^{2n+1}\smallsetminus\{0\}
    \end{align*}
    is a local holomorphic lift of $F_n$, then
    \begin{align*}
        (E_k)_x
        =
        \operatorname{span}_{\c}
        \left\{
            \xi(x),
            \partial_z\xi(x),
            \ldots,
            \partial_z^k\xi(x)
        \right\}
        \qquad
        \forall\,x\in U.
    \end{align*}
    This definition is independent of the choices of $z$ and $\xi$.
    By the definition of the twistor lift,
    \begin{align*}
        \tilde\varphi_n^*\mathscr S_n
        =
        E_{n-1}.
    \end{align*}

    Let
    \begin{align*}
        \mathcal P^k
        \bigl(
            \mathcal O_{\cp^1}(2n)
        \bigr)
    \end{align*}
    denote the holomorphic bundle of principal parts of order $k$.
    Since $F_n$ is induced by the complete linear system
    $\lvert\mathcal O_{\cp^1}(2n)\rvert$, the jet-evaluation morphism
    \begin{align*}
        H^0\bigl(\cp^1,\mathcal O_{\cp^1}(2n)\bigr)
        \otimes\mathcal O_{\cp^1}
        \rightarrow
        \mathcal P^k\bigl(\mathcal O_{\cp^1}(2n)\bigr)
    \end{align*}
    is surjective for $k\leq2n$. Its dual is therefore an injective
    morphism
    \begin{align*}
        \mathcal P^k\bigl(\mathcal O_{\cp^1}(2n)\bigr)^\vee
        \rightarrow
        H^0\bigl(\cp^1,\mathcal O_{\cp^1}(2n)\bigr)^\vee
        \otimes\mathcal O_{\cp^1},
    \end{align*}
    whose image is precisely the $k$-th osculating bundle $E_k$. Indeed, choose a basis of $H^0(\mathcal O_{\cp^1}(2n))$ and write its local representatives as $f_0,\ldots,f_{2n}$. The corresponding local lift of $F_n$ is $\xi=(f_0,\ldots,f_{2n})$. The dual jet-evaluation map sends the functionals extracting the Taylor coefficients of orders $0,\ldots,k$ to $\xi,\partial_z\xi,\ldots,\partial_z^k\xi$, up to nonzero factorial factors. Consequently,
    \begin{align*}
        E_k
        \cong
        \mathcal P^k
        \bigl(\mathcal O_{\cp^1}(2n)\bigr)^\vee.
    \end{align*}
    The standard splitting formula for principal-parts bundles on
    $\cp^1$ (see e.g. \cite{Maakestad2008}) gives
    \begin{align*}
        \mathcal P^k
        \bigl(
            \mathcal O_{\cp^1}(2n)
        \bigr)
        \cong
        \mathcal O_{\cp^1}(2n-k)^{\oplus(k+1)}.
    \end{align*}
    Taking $k=n-1$, we obtain
    \begin{align*}
        \tilde\varphi_n^*\mathscr S_n
        =
        E_{n-1}
        &\cong
        \mathcal O_{\cp^1}
        \bigl(
            -2n+n-1
        \bigr)^{\oplus n}
        \\
        &=
        \mathcal O_{\cp^1}(-n-1)^{\oplus n}.
    \end{align*}
    Finally, the map
    \begin{align*}
        x\mapsto\bigl(x,[\varphi_n(x)]\bigr)
    \end{align*}
    defines a nowhere-vanishing smooth section of $\tilde\varphi_n^*\mathscr L_n$. Here $[\varphi_n(x)]$ denotes its class in $\tilde\varphi_n(x)^\perp/\tilde\varphi_n(x)$. This class is nonzero: the plane $\tilde\varphi_n(x)$ is isotropic, whereas $(\varphi_n(x),\varphi_n(x))=1$. Consequently,
    \begin{align*}
        c_1
        \bigl(
            \tilde\varphi_n^*\mathscr L_n
        \bigr)
        =
        0.
    \end{align*}
    Since every degree-zero holomorphic line bundle on $\cp^1$ is
    holomorphically trivial, it follows that
    \begin{align*}
        \tilde\varphi_n^*\mathscr L_n
        \cong
        \mathcal O_{\cp^1}.
    \end{align*}
    This concludes the proof.
\end{proof}
\begin{corollary}\label{Corollary: splitting along balanced Boruvka covers}
    Let $\varphi_n:\mathbb S^2\rightarrow\mathbb S^{2n}$ be the superminimal immersion of constant Gaussian curvature corresponding to a special orbit of the irreducible representation $\operatorname{SO}(3)\rightarrow\operatorname{SO}(2n+1)$ and let $\tilde\varphi_n:\mathbb{CP}^1\rightarrow\mathscr Z_n$ be its unique twistor lift. Then there exist integers
    \begin{align*}
        a_1,\ldots,a_{\frac{n(n+1)}{2}}\geq n+1,
    \end{align*}
    depending only on $n$, such that
    \begin{align}\label{Equation: Boruvka tangent vertical splitting}
    \begin{split}
        \tilde\varphi_n^*T^{1,0}\mathscr Z_n
        &\cong
        \bigoplus_{i=1}^{\frac{n(n+1)}{2}}
        \mathcal O_{\cp^1}(a_i),
        \\
        \tilde\varphi_n^*V\mathscr Z_n
        &\cong
        \mathcal O_{\cp^1}(2n+2)^{
            \oplus\frac{n(n-1)}{2}
        }.
    \end{split}
    \end{align}
    Consequently, if $f:\Sigma\rightarrow\cp^1$ is a balanced holomorphic cover of degree $\ell\in\n\smallsetminus\{0\}$ and $\psi:=\tilde\varphi_n\circ f$, we have
    \begin{align}\label{Equation: balanced Boruvka bundle splitting}
    \begin{split}
        \psi^*T^{1,0}\mathscr Z_n
        &\cong
        \bigoplus_{i=1}^{\frac{n(n+1)}{2}}
        L_{i,\ell},
        \\
        \psi^*V\mathscr Z_n
        &\cong
        \tilde L_{\ell}^{
            \oplus\frac{n(n-1)}{2}
        },
    \end{split}
    \end{align}
    where
    \begin{align*}
        L_{i,\ell}
        &:=
        f^*\mathcal O_{\cp^1}(a_i),
        \\
        \tilde L_{\ell}
        &:=
        f^*\mathcal O_{\cp^1}(2n+2).
    \end{align*}
    In particular,
    \begin{align}\label{Equation: balanced Boruvka summand degrees}
        \deg(L_{i,\ell})
        &=
        a_i\ell
        \geq
        (n+1)\ell,
        \qquad
        i=1,\ldots,\frac{n(n+1)}{2},
        \\
        \deg(\tilde L_{\ell})
        &=
        (2n+2)\ell.
    \end{align}
\end{corollary}
\begin{proof}
    By Lemma~\ref{Lemma: pullback universal bundles}, we have
    \begin{align}\label{Equation: Boruvka universal splitting recalled}
        \tilde\varphi_n^*\mathscr S_n
        \cong
        \mathcal O_{\cp^1}(-n-1)^{\oplus n},
        \quad
        \tilde\varphi_n^*\mathscr L_n
        \cong
        \mathcal O_{\cp^1}.
    \end{align}
    Recall that the holomorphic tangent bundle of $\mathscr Z_n$ fits
    into the short exact sequence
    \begin{align}\label{Equation: holomorphic tangent sequence twistor}
        0
        \rightarrow
        \mathscr S_n^\vee\otimes\mathscr L_n
        \rightarrow
        T^{1,0}\mathscr Z_n
        \rightarrow
        \wedge^2\mathscr S_n^\vee
        \rightarrow
        0.
    \end{align}
    Pulling back by $\tilde\varphi_n$ and using
    \eqref{Equation: Boruvka universal splitting recalled}, we obtain
    \begin{align}\label{Equation: Boruvka pulled tangent exact sequence}
        0
        \rightarrow
        \mathcal O_{\cp^1}(n+1)^{\oplus n}
        \rightarrow
        \tilde\varphi_n^*T^{1,0}\mathscr{Z}_n
        \rightarrow
        \mathcal O_{\cp^1}(2n+2)^{
            \oplus\frac{n(n-1)}{2}
        }
        \rightarrow
        0.
    \end{align}
    In particular, using the identification
    \begin{align*}
        V\mathscr Z_n
        \cong
        \wedge^2\mathscr S_n^\vee,
    \end{align*}
    we immediately obtain
    \begin{align*}
        \tilde\varphi_n^*V\mathscr Z_n
        \cong
        \wedge^2
        \left(
            \tilde\varphi_n^*\mathscr S_n^\vee
        \right)
        \cong
        \wedge^2
        \left(
            \mathcal O_{\cp^1}(n+1)^{\oplus n}
        \right)
        \cong
        \mathcal O_{\cp^1}(2n+2)^{
            \oplus\frac{n(n-1)}{2}
        }.
    \end{align*}
    By the Birkhoff--Grothendieck splitting theorem, there exist integers
    $a_1,\ldots,a_{\frac{n(n+1)}{2}}$ such that
    \begin{align}\label{Equation: BG splitting E0}
        \tilde\varphi_n^*T^{1,0}\mathscr{Z}_n
        \cong
        \bigoplus_{i=1}^{\frac{n(n+1)}{2}}
        \mathcal O_{\cp^1}(a_i).
    \end{align}
    We claim that
    \begin{align*}
        a_i\geq n+1
        \qquad
        \forall\,
        i=1,\ldots,\frac{n(n+1)}{2}.
    \end{align*}
    Indeed, tensoring
    \eqref{Equation: Boruvka pulled tangent exact sequence} with
    $\mathcal O_{\cp^1}(-n-1)$ gives
    \begin{align}\label{Equation: twisted Boruvka tangent exact sequence}
        0
        \rightarrow
        \mathcal O_{\cp^1}^{\oplus n}
        \rightarrow
        \tilde\varphi_n^*T^{1,0}\mathscr{Z}_n\otimes\mathcal O_{\cp^1}(-n-1)
        \rightarrow
        \mathcal O_{\cp^1}(n+1)^{
            \oplus\frac{n(n-1)}{2}
        }
        \rightarrow
        0.
    \end{align}
    Since
    \begin{align*}
        H^1(\cp^1,\mathcal O_{\cp^1})=0,
    \end{align*}
    the induced map on global sections
    \begin{align*}
        H^0\left(
            \cp^1,
            \tilde\varphi_n^*T^{1,0}\mathscr{Z}_n\otimes\mathcal O_{\cp^1}(-n-1)
        \right)
        \rightarrow
        H^0\big(
            \cp^1,
            \mathcal O_{\cp^1}(n+1)^{
                \oplus\frac{n(n-1)}{2}
            }
        \big)
    \end{align*}
    is surjective. Since both
    $\mathcal O_{\cp^1}^{\oplus n}$ and
    $\mathcal O_{\cp^1}(n+1)^{
        \oplus\frac{n(n-1)}{2}
    }$
    are globally generated, it follows from
    \eqref{Equation: twisted Boruvka tangent exact sequence} that
    \begin{align*}
        \tilde\varphi_n^*T^{1,0}\mathscr{Z}_n\otimes\mathcal O_{\cp^1}(-n-1)
    \end{align*}
    is globally generated. On the other hand, by
    \eqref{Equation: BG splitting E0},
    \begin{align*}
        \tilde\varphi_n^*T^{1,0}\mathscr Z_n\otimes\mathcal O_{\cp^1}(-n-1)
        \cong
        \bigoplus_{i=1}^{\frac{n(n+1)}{2}}
        \mathcal O_{\cp^1}(a_i-n-1).
    \end{align*}
    A line bundle $\mathcal O_{\cp^1}(k)$ is globally generated if and
    only if $k\geq0$. Hence
    \begin{align*}
        a_i-n-1\geq0
    \end{align*}
    for every $i$, proving the claim.
    Finally, since $\psi=\tilde\varphi_n\circ f$, pulling back \eqref{Equation: Boruvka tangent vertical splitting} by $f$ yields
    \begin{align*}
        \psi^*T^{1,0}\mathscr Z_n
        &=
        f^*\tilde\varphi_n^*T^{1,0}\mathscr{Z}_n
        \cong
        \bigoplus_{i=1}^{\frac{n(n+1)}{2}}
        f^*\mathcal O_{\cp^1}(a_i),
        \\
        \psi^*V\mathscr Z_n
        &=
        f^*\tilde\varphi_n^*V\mathscr Z_n
        \cong
        f^*\mathcal O_{\cp^1}(2n+2)^{
            \oplus\frac{n(n-1)}{2}
        }.
    \end{align*}
    Since $\deg(f)=\ell$, we have
    \begin{align*}
        \deg
        \big(
            f^*\mathcal O_{\cp^1}(k)
        \big)
        =
        k\ell
    \end{align*}
    for every $k\in\mathbb Z$, and
    \eqref{Equation: balanced Boruvka summand degrees} follows.
\end{proof}
\begin{remark}\label{Remark: satisfying assumptions for holomorphic peak sections}
    The holomorphic splittings above hold for every holomorphic cover $f$ of degree $\ell$. For the peak-section construction, equip $\mathcal O_{\cp^1}(a)$ with its Fubini--Study metric $h_a$ and $f^*\mathcal O_{\cp^1}(a)$ with $f^*h_a$, where $a$ is one of the fixed integers $a_i$ or $2n+2$. These are auxiliary metrics on the summands. The balanced estimates become essential when the summands $L_{i,\ell}$ are equipped with their natural pullback Hermitian metrics. Their Chern curvatures are then proportional to $f^*\omega_{\cp^1}$ and hence to $|df|_{g_0}^2\omega_0$. The averaged lower bound in the definition of a balanced cover will therefore provide precisely the averaged curvature positivity required in Proposition~\ref{Proposition: holomorphic peak sections}. The natural-scale curvature bounds required there follow from Remark~\ref{Remark: balanced pullback bounded geometry}.
\end{remark}
\subsection{A polynomially bounded holomorphic right inverse}
We establish the regularity of balanced Bor{\accent23 u}vka covers by constructing a polynomially bounded holomorphic right inverse for $L_\psi$. A homogeneous filtration reduces the construction to a finite collection of jet equations. We solve these equations across the branch locus by prescribing scalar jets and applying a global $\bar\partial$-correction, then lift the solutions through the filtration.
\begin{lemma}[A filtration of linearized horizontality]
\label{Lemma: filtration of linearized horizontality}
    Let $n\geq2$, and set
    \begin{align*}
        E_0:=\tilde\varphi_n^*T^{1,0}\mathscr Z_n,
        \qquad
        V_0:=\tilde\varphi_n^*V\mathscr Z_n.
    \end{align*}
    Let $q_0:E_0\to V_0$ be the pulled-back holomorphic quotient.
    There exist $\operatorname{SL}(2,\c)$-equivariant holomorphic
    filtrations
    \begin{align*}
        0=\mathcal E_0\subset\mathcal E_1\subset\cdots
        \subset\mathcal E_n=E_0,
        \qquad
        0=\mathcal V_0=\mathcal V_1\subset\cdots
        \subset\mathcal V_n=V_0
    \end{align*}
    satisfying $q_0(\mathcal E_j)=\mathcal V_j$ and
    $\mathcal E_1\cong\mathcal O_{\cp^1}(2)$.
    For every $2\leq j\leq n$, there are holomorphic bundle
    isomorphisms
    \begin{align*}
        \mathcal E_j/\mathcal E_{j-1}
        &\cong J^{j-1}\mathcal O_{\cp^1}(4j-2),\\
        \mathcal V_j/\mathcal V_{j-1}
        &\cong J^{j-2}\mathcal O_{\cp^1}(4j-2),
    \end{align*}
    under which the map induced by $q_0$ is jet truncation.
    Here $J^mL$ denotes the holomorphic jet bundle of order $m$
    of a line bundle $L$ on $\cp^1$.

    Moreover, let $f:\Sigma\to\cp^1$ be a holomorphic map
    from a Riemann surface, and put
    $\psi:=\tilde\varphi_n\circ f$.
    The operator $L_\psi$ maps local holomorphic sections of
    $f^*\mathcal E_j$ to local holomorphic sections of
    $\omega_\Sigma\otimes f^*\mathcal V_j$.
    For $j=m+1$, where $1\leq m\leq n-1$, its induced quotient
    operator is
    \begin{align*}
        D_{m,f}(u_0,\ldots,u_m)
        =\bigl(du_a-u_{a+1}\,d(w\circ f)\bigr)_{a=0}^{m-1}
    \end{align*}
    under the pullbacks of the preceding isomorphisms.
    Here $w$ is a local holomorphic coordinate on $\cp^1$,
    and the $u_a$ are the components in the pulled-back
    derivative coordinates of the jet bundle, associated with
    $w$ and a local holomorphic frame of
    $\mathcal O_{\cp^1}(4m+2)$.
\end{lemma}
\begin{proof}
    \phantom{.}

    \medskip
    \noindent
    \textbf{Step 1: a filtration of the homogeneous tangent sequence}.
    We first construct filtrations of $E_0$ and $V_0$ for which the
    induced maps
    \begin{align*}
        \overline q_j:
        \mathcal E_j/\mathcal E_{j-1}
        \longrightarrow
        \mathcal V_j/\mathcal V_{j-1}
    \end{align*}
    have source rank $j$, target rank $j-1$, and one-dimensional
    kernel. We also compute the action of the stabilizer of
    $[1:0]$ on these quotients and identify the kernel explicitly.
    These data will allow us, in Step~2, to identify the quotient
    bundles with jet bundles and the maps $\overline q_j$ with
    jet truncation. We construct the filtrations using the
    holomorphic sections induced by infinitesimal orthogonal
    transformations, and compute their ranks by evaluating
    these sections at $[1:0]$. As in the proof of Lemma~\ref{Lemma: pullback universal bundles}, we identify the complexified irreducible representation defining the Bor{\accent23 u}vka sphere with $\operatorname{Sym}^{2n}\c^2$ by an $\operatorname{SU}(2)$-equivariant isomorphism. Its monomial basis is
    \begin{align*}
        b_i:=e_0^{2n-i}e_1^i
        \qquad\forall\,i=0,\ldots,2n,
    \end{align*}
    where $\{e_0,e_1\}$ is the standard basis of $\c^2$. The identification is chosen so that the standard complex bilinear pairing on $\c^{2n+1}$ has, in this basis, the expression
    \begin{align*}
        (b_i,b_j)=\frac{(-1)^i}{\binom{2n}{i}}\delta_{i+j,2n}
        \qquad\forall\,i,j=0,\ldots,2n.
    \end{align*}
    The displayed form is $\operatorname{SL}(2,\c)$-invariant.
    By irreducibility, the standard complex bilinear pairing
    transported by the equivariant identification is a nonzero
    scalar multiple of this form. Rescaling the identification
    therefore gives the stated normalization while preserving
    equivariance. Define the linear endomorphisms $X,Y,H\in\mathfrak{so}(\operatorname{Sym}^{2n}\c^2)=\mathfrak{so}(2n+1,\c)$ by
    \begin{align*}
        Xb_i:=(2n-i)b_{i+1}, \qquad Yb_i:=ib_{i-1}, \qquad Hb_i:=(2n-2i)b_i.
    \end{align*}
    Here $Xb_{2n}=Yb_0=0$. These operators satisfy $[Y,X]=H$, $[H,X]=-2X$, and $[H,Y]=2Y$. The directrix of $\tilde\varphi_n$ is the rational normal curve $[1:w]\mapsto[(e_0+we_1)^{2n}]$. Its first $n$ osculating vectors span $e^{wX}P_0$, where $P_0:=\operatorname{span}_\c\{b_0,\ldots,b_{n-1}\}$. Hence
    \begin{align*}
        \tilde\varphi_n([1:w])=e^{wX}P_0.
    \end{align*}
    For endomorphisms, write $(\operatorname{ad}A)(B):=[A,B]$.
    We organize the infinitesimal orthogonal transformations into
    subspaces invariant under the $\operatorname{SL}(2,\c)$ action,
    so that their evaluated images give equivariant subbundles.
    We start with odd powers of $X$, which remain skew-adjoint. Thus, for odd $s\in\{1,3,\ldots,2n-1\}$, we let $A_{s,a}\in\mathfrak{so}(2n+1,\c)$ be given by
    \begin{align*}
        A_{s,a}:=(\operatorname{ad}Y)^aX^s,
    \end{align*}
    and we define
    \begin{align*}
        U_s:=\operatorname{span}_\c\{A_{s,a} \mbox{ : } 0\leq a\leq 2s\}\subset\mathfrak{so}(2n+1,\c).
    \end{align*}
    Since
    \begin{align*}
        e^{tY}X^se^{-tY}=(X+tH-t^2Y)^s,
    \end{align*}
    we have $A_{s,2s}=(-1)^s(2s)!Y^s\neq0$ and $A_{s,2s+1}=0$. The commutator identities give
    \begin{align}
        \label{Equation: right inverse representation identities}
        [H,A_{s,a}]&=2(a-s)A_{s,a},\\
        \nonumber
        [Y,A_{s,a}]&=A_{s,a+1}, \hspace{2.75cm} (a\leq 2s-1)\\
        [X,A_{s,a}]&=a(2s-a+1)A_{s,a-1} \qquad(a\geq1),
        \nonumber
    \end{align}
    while $[X,A_{s,0}]=0$. The non-vanishing of $A_{s,2s}$ shows that every earlier $A_{s,a}$ is non-zero. Their distinct eigenvalues for $\operatorname{ad}H$ show
    that they are linearly independent. The commutator identities
    also show that $U_s$ is an irreducible $\mathfrak{sl}(2,\c)$-module
    for the action by commutators with $X,Y,H$, of dimension $2s+1$. The Casimir operator
    \begin{align*}
        (\operatorname{ad}H)^2+2(\operatorname{ad}X)(\operatorname{ad}Y)+2(\operatorname{ad}Y)(\operatorname{ad}X)\in\operatorname{End}(\mathfrak{so}(2n+1,\c))
    \end{align*}
    acts on $U_s$ by multiplication by $4s(s+1)$, as follows from \eqref{Equation: right inverse representation identities}. These eigenvalues are distinct, so the sum of the $U_s$ is direct. Since
    \begin{align*}
        \sum_{j=1}^n(4j-1)=n(2n+1)=\dim_\c\mathfrak{so}(2n+1,\c),
    \end{align*}
    they exhaust $\mathfrak{so}(2n+1,\c)$.
    For $A\in\mathfrak{so}(2n+1,\c)$, its fundamental section is the infinitesimal deformation induced by the orthogonal action:
    \begin{align*}
        \sigma_A(w):=\left.\frac{d}{dt}\right|_{t=0}e^{tA}\tilde\varphi_n(w).
    \end{align*}
    It is a holomorphic section of $E_0$. Set $\mathcal E_0=0$ and define the image of the holomorphic evaluation map by
    \begin{align*}
        (\mathcal E_j)_w&:=\operatorname{span}_\c\{\sigma_A(w) \mbox{ : } A\in U_1\oplus U_3\oplus\cdots\oplus U_{2j-1}\},\\
        \mathcal V_j&:=q_0(\mathcal E_j).
    \end{align*}
    A filtration here means these nested families of subbundles; the quotient $\mathcal E_j/\mathcal E_{j-1}$ records the new directions added at stage $j$. By $\operatorname{SL}(2,\c)$-equivariance, the evaluation maps at $p\in\cp^1$ and $g\cdot p$ are related by linear isomorphisms of their domains and target fibres, and therefore have the same rank. Since $\operatorname{SL}(2,\c)$ acts transitively on $\cp^1$, this rank is independent of $p$. The quotient map $q_0$ is also equivariant, so the same argument applies to its compositions with the evaluation maps. These are holomorphic bundle maps of constant rank. Hence their images $\mathcal E_j$ and $\mathcal V_j$ are holomorphic subbundles.
    
    \smallskip
    \noindent
    We now compute the ranks of the holomorphic subbundles $\mathcal E_j$ and $\mathcal V_j$. Clearly, it suffices to work at $[1:0]\in\cp^1$, where $\tilde\varphi_n([1:0])=P_0$. 
    Under the usual identification of the tangent space of the Grassmannian with
    $\operatorname{Hom}(P_0,\c^{2n+1}/P_0)$, the tangent vector induced by an endomorphism $A$ is
    \begin{align*}
        \operatorname{ev}_0(A):P_0&\longrightarrow\c^{2n+1}/P_0,\\
        v&\longmapsto [Av],
    \end{align*}
    where ``$[\,\cdot\,]$'' stands for the equivalence class of a vector modulo $P_0$. 
    Let $r\geq1$ be an integer. We say that an endomorphism $A$ has $H$-weight $-2r$ if $[H,A]=-2rA$. We use the same terminology for the induced
    action on $T_{P_0}^{1,0}\mathscr Z_n$. More explicitly, viewing a tangent vector to $\mathscr Z_n$ as a map $T:P_0\to\c^{2n+1}/P_0$, this action is
    \begin{align*}
        H\cdot T:=\overline H\circ T-T\circ(H|_{P_0}),
    \end{align*}
    where $\overline H$ is the endomorphism induced by $H$ on $\c^{2n+1}/P_0$. We denote its weight-$-2r$ subspace by
    \begin{align*}
        \bigl(T_{P_0}^{1,0}\mathscr Z_n\bigr)_{-2r}:=\left\{T\in T_{P_0}^{1,0}\mathscr Z_n:
        H\cdot T=-2rT\right\}.
    \end{align*}
    Evaluation preserves weights, since $H\cdot\operatorname{ev}_0(A)=\operatorname{ev}_0([H,A])$. Since
    \begin{align*}
        [H,A_{s,a}]=2(a-s)A_{s,a},
    \end{align*}
    the subspace of $U_s$ consisting of the endomorphisms of weight $-2r$ is zero when $s<r$ and is spanned by $A_{s,s-r}$ when $s\geq r$.
    To determine whether these endomorphisms give independent
    tangent vectors after evaluation, we express their coefficients
    in terms of polynomials. We claim that there is a polynomial
    $p_{s,r}$ such that
    \begin{align}
        \label{Equation: right inverse polynomial weight vectors}
        A_{s,s-r}=X^r p_{s,r}(H),\qquad\deg p_{s,r}=s-r.
    \end{align}
    To prove this claim, first note that
    \begin{align*}
        [Y,X^r]&=rX^{r-1}(H-r+1),\\
        [Y,p(H)]&=Y\bigl(p(H)-p(H+2)\bigr),\\
        XY&=\frac{(2n-H)(2n+H+2)}4.
    \end{align*}
    The first identity follows by expanding the commutator of $Y$ with $X^r$, the second follows from $HY=Y(H+2)$, and the third can be checked on each basis vector $b_i$. Consequently,
    \begin{align*}
        [Y,X^rp(H)]=X^{r-1}\Bigg(r(H-r+1)p(H)+\frac{(2n-H)(2n+H+2)}4\bigl(p(H)-p(H+2)\bigr)\Bigg).
    \end{align*}
    If $p$ has degree $a$ and leading coefficient $c_a\neq0$, the two terms in parentheses contribute respectively $rc_a$ and $\frac a2c_a$ to the coefficient of $H^{a+1}$. Thus the polynomial in parentheses has degree $a+1$ and leading coefficient
    \begin{align*}
        \left(r+\frac a2\right)c_a\neq0.
    \end{align*}
    Starting from $A_{s,0}=X^s$, each successive commutator with $Y$ therefore decreases the power of $X$ by one and increases the polynomial degree by one, as long as the resulting power of $X$ remains positive. After $s-r$ commutators, this gives \eqref{Equation: right inverse polynomial weight vectors}. In particular, for $A=X^rp(H)$,
    \begin{align*}
        \operatorname{ev}_0(A)(b_i)=p(2n-2i)\bigg[\frac{(2n-i)!}{(2n-i-r)!}\,b_{i+r}\bigg]
    \end{align*}
    whenever $0\leq i\leq n-1$ and $i+r\leq2n$. The right-hand side of the above equation is zero when $i+r<n$, while $X^rb_i=0$ when $i+r>2n$. Thus the only indices that contribute are
    \begin{align*}
        I_r:=
        \left\{i\in\n\cup\{0\}:
        \max(0,n-r)\leq i\leq\min(n-1,2n-r)\right\}.
    \end{align*}
    For $1\leq r\leq2n-1$, 
    \begin{align*}
        |I_r|=
        \begin{cases}
            r,&1\leq r\leq n,\\
            2n-r+1,&n<r\leq2n-1.
        \end{cases}
    \end{align*}
    All the factorial coefficients above are nonzero for $i\in I_r$. Consequently,
    \begin{align*}
        \operatorname{ev}_0\bigl(X^rp(H)\bigr)=0
        \quad\Leftrightarrow\quad
        p(2n-2i)=0
        \qquad\forall\,i\in I_r.
    \end{align*}
    Therefore, if $\deg p<|I_r|$ and $\operatorname{ev}_0(X^rp(H))=0$, then $p$ has more distinct roots than its degree and must be the zero polynomial. It follows that linearly independent polynomials of degree
    less than $|I_r|$ give linearly independent elements of
    $\operatorname{Hom}(P_0,\c^{2n+1}/P_0)$ under the assignment
    \begin{align*}
        p\mapsto\operatorname{ev}_0(X^rp(H)).
    \end{align*}
    For the polynomials $p_{s,r}$, these maps belong to
    $T_{P_0}^{1,0}\mathscr Z_n$, since the corresponding
    endomorphisms are skew-adjoint. We will apply this observation
    to count the independent tangent directions contributed by
    the modules $U_s$.

    \smallskip
    \noindent
    To determine when the preceding modules already span an
    entire tangent weight space, we compute its dimension for
    $1\leq r\leq n$. A map sending $b_i$ to $[b_k]$ has
    $H$-weight $(2n-2k)-(2n-2i)=2(i-k)$.
    Hence a tangent vector $T$ of weight $-2r$ has the form
    \begin{align*}
        T(b_i)=
        \begin{cases}
            t_i b_{i+r}\mod P_0,&n-r\leq i\leq n-1,\\
            0,&0\leq i<n-r.
        \end{cases}
    \end{align*}
    Differentiating the isotropy condition for $P_0$ gives
    \begin{align*}
        (T(b_i),b_{i'})+(b_i,T(b_{i'}))=0
        \qquad\forall\,i,i'=0,\ldots,n-1,
    \end{align*}
    where the pairings are independent of the representatives
    chosen modulo $P_0$.

    \smallskip
    \noindent
    The coefficient $t_{n-r}$ sends $b_{n-r}$ to $b_n$ and is
    unconstrained, since $b_n$ is orthogonal to $P_0$.
    For the remaining indices $n-r+1,\ldots,n-1$, the only
    nontrivial equations occur when
    \begin{align*}
        i+i'=2n-r.
    \end{align*}
    Each pair of distinct indices satisfying this relation contributes
    two coefficients subject to one nontrivial linear equation, and
    therefore one independent parameter. If $i=i'$, the corresponding
    equation forces $t_i=0$.
    The involution $i\mapsto2n-r-i$ has
    $\lfloor(r-1)/2\rfloor$ pairs of distinct indices. Hence
    \begin{align*}
        \dim_\c
        \bigl(T_{P_0}^{1,0}\mathscr Z_n\bigr)_{-2r}
        =
        1+\left\lfloor\frac{r-1}{2}\right\rfloor
        =
        \left\lceil\frac r2\right\rceil.
    \end{align*}
    Moreover,
    \begin{align*}
        P_0^\perp
        =\operatorname{span}_\c\{b_0,\ldots,b_n\},
        \qquad
        (\ker q_0)_{[1:0]}
        =\operatorname{Hom}(P_0,P_0^\perp/P_0).
    \end{align*}
    Thus the horizontal part of this weight space is precisely
    the one-dimensional subspace for which only $t_{n-r}$ may
    be nonzero. Its maps send $b_{n-r}$ to a multiple of $b_n$
    and all other basis vectors to zero modulo $P_0$.
    There are no horizontal weight spaces with $r>n$.

    \smallskip
    \noindent
    Fix now $j\in\{1,\ldots,n\}$. We determine which tangent
    directions are added when passing from $\mathcal E_{j-1}$
    to $\mathcal E_j$.
    For $j\leq r\leq2j-1$, the weight-$-2r$ endomorphisms
    in $U_1\oplus U_3\oplus\cdots\oplus U_{2j-1}$ have the form
    $X^rp(H)$ with
    \begin{align*}
        \deg p\leq2j-1-r.
    \end{align*}
    This degree is strictly smaller than $|I_r|$. Indeed,
    \begin{align*}
        2j-1-r&\leq r-1<r=|I_r|
        &&\mbox{if }r\leq n,\\
        2j-1-r&\leq2n-1-r<2n-r+1=|I_r|
        &&\mbox{if }r>n.
    \end{align*}
    Evaluation is therefore injective on this polynomial space.
    The polynomials $p_{s,r}$ belonging to different odd indices
    $s$ have distinct degrees $s-r$, so they are linearly independent.
    In particular, $p_{2j-1,r}$ has larger degree than every polynomial
    contributed by the preceding modules. Adding $U_{2j-1}$ therefore
    adds exactly one tangent direction of each weight $-2r$ with
    $j\leq r\leq2j-1$.

    \smallskip
    \noindent
    For the remaining positive shifts, we use the following observation.
    For any $1\leq r\leq n$, the modules among
    $U_1,U_3,\ldots,U_{2r-1}$ that contribute to weight $-2r$
    are exactly those with odd index
    \begin{align*}
        r\leq s\leq2r-1.
    \end{align*}
    There are $\lceil r/2\rceil$ such indices. Their polynomials
    $p_{s,r}$ have distinct degrees at most $r-1$, so evaluation
    at the $r$ distinct points indexed by $I_r$ is injective.
    Their images are therefore $\lceil r/2\rceil$ independent
    vectors in a tangent weight space of that same dimension.
    Thus $\mathcal E_r$ contains the entire tangent weight
    space of weight $-2r$.
    In particular, when $r<j$, that weight space is already
    contained in $\mathcal E_{j-1}$.

    \smallskip
    \noindent
    Finally, an endomorphism of nonnegative $H$-weight preserves
    $P_0$ and therefore evaluates to zero.
    There are no weights $-2r$ with $r>2j-1$ in the first $j$
    modules. It follows that the fibre of
    $\mathcal E_j/\mathcal E_{j-1}$ at $[1:0]$ has precisely
    the following $j$ weights, each with multiplicity one:
    \begin{align}
        \label{Equation: right inverse quotient weights}
        -2(2j-1),-2(2j-2),\ldots,-2j.
    \end{align}
    A basis of this fibre is given by the classes
    \begin{align*}
        v_{j,a}:=
        \bigl[\operatorname{ev}_0(A_{2j-1,a})\bigr],
        \qquad 0\leq a\leq j-1.
    \end{align*}
    Since evaluation at $[1:0]$ is equivariant under its stabilizer,
    the identity $[Y,A_{s,a}]=A_{s,a+1}$ gives
    \begin{align*}
        Y\cdot v_{j,a}
        =
        \begin{cases}
            v_{j,a+1},&0\leq a<j-1,\\
            0,&a=j-1.
        \end{cases}
    \end{align*}
    For the last equality, the next endomorphism has weight
    $-2(j-1)$, whose tangent weight space already lies in
    $\mathcal E_{j-1}$ when $j\geq2$. When $j=1$, the next
    endomorphism is $A_{1,1}=H$, which preserves $P_0$ and
    evaluates to zero.
    
    \smallskip
    \noindent
    Since $\mathcal V_j=q_0(\mathcal E_j)$, its rank is the
    rank of $\mathcal E_j$ minus the dimension of its horizontal
    part. We therefore determine $\mathcal E_j\cap\ker q_0$.
    The preceding dimension argument shows that, for every
    $1\leq r\leq j$, $\mathcal E_j$ contains the horizontal
    line of weight $-2r$, since it already belongs to $\mathcal E_r$.
    Suppose instead that $j<r\leq\min\{n,2j-1\}$.
    Every weight-$-2r$ vector in $(\mathcal E_j)_{[1:0]}$
    is represented by $X^rp(H)$ with
    \begin{align*}
        \deg p\leq2j-1-r.
    \end{align*}
    If this vector is horizontal, its coefficients must vanish
    at all indices in $I_r$ except $i=n-r$. Thus $p$ vanishes
    at $r-1$ distinct points. Since
    \begin{align*}
        2j-1-r<r-1,
    \end{align*}
    this forces $p=0$. If $r>2j-1$, there is no contribution
    from these modules in the first place.
    As there are no horizontal weights with $r>n$, we conclude
    that
    \begin{align*}
        \mathcal K_j:=\mathcal E_j\cap\ker q_0
    \end{align*}
    has, at $[1:0]$, exactly the weights
    $-2,-4,\ldots,-2j$, each with multiplicity one.
    Equivariance gives $\operatorname{rank}\mathcal K_j=j$
    at every point.

    \smallskip
    \noindent
    The successive quotients of $\mathcal E_j$ have ranks
    $1,\ldots,j$, and $\mathcal V_j=q_0(\mathcal E_j)$.
    Therefore
    \begin{align}
        \label{Equation: right inverse filtration ranks}
        \operatorname{rank}\mathcal E_j
        &=\sum_{k=1}^j k=\frac{j(j+1)}2,\\
        \operatorname{rank}\mathcal V_j
        &=\operatorname{rank}\mathcal E_j
          -\operatorname{rank}\mathcal K_j
          =\frac{j(j-1)}2.
        \nonumber
    \end{align}
    For $j=n$, these are the ranks of $E_0$ and $V_0$,
    respectively. Hence the filtrations are
    \begin{align*}
        0=\mathcal E_0\subset\mathcal E_1\subset\cdots
        \subset\mathcal E_n=E_0,\qquad
        0=\mathcal V_0=\mathcal V_1\subset\cdots
        \subset\mathcal V_n=V_0.
    \end{align*}

    \smallskip
    \noindent
    Finally, $q_0$ induces a surjective holomorphic bundle map
    \begin{align*}
        \overline q_j:
        \mathcal E_j/\mathcal E_{j-1}
        \longrightarrow
        \mathcal V_j/\mathcal V_{j-1}.
    \end{align*}
    A class $[v]$ belongs to its kernel precisely when
    $q_0(v)\in\mathcal V_{j-1}$.
    Since $\mathcal V_{j-1}=q_0(\mathcal E_{j-1})$, we can
    then choose $v'\in\mathcal E_{j-1}$ with
    $q_0(v')=q_0(v)$. The difference $v-v'$ is horizontal
    and represents the same class as $v$. Consequently,
    \begin{align*}
        \ker\overline q_j
        =
        \frac{\mathcal E_{j-1}+\mathcal K_j}{\mathcal E_{j-1}}
        \cong
        \frac{\mathcal K_j}{\mathcal K_{j-1}}.
    \end{align*}
    At $[1:0]$, this quotient is the one-dimensional space
    of weight $-2j$. Thus the fibre of $\ker\overline q_j$
    at $[1:0]$ is precisely the final weight space in
    \eqref{Equation: right inverse quotient weights},
    spanned by $v_{j,j-1}$.

    \medskip
    \noindent
    \textbf{Step 2: identification of the quotient operators}.
    Denote by $J^mL$ the holomorphic jet bundle of order $m$ of a
    line bundle $L$ on $\cp^1$. Its local coordinates are derivatives,
    not factorial-normalized Taylor coefficients. These jet bundles
    are on $\cp^1$; below they are pulled back as bundles, not
    replaced by jet bundles on $\Sigma$.
    Fix $j\geq2$, put $s=2j-1$ and $m=j-1$, and let
    $\mathcal Q_j:=\mathcal E_j/\mathcal E_{j-1}$.
    The stabilizer of $[1:0]$ in $\operatorname{SL}(2,\c)$ consists
    of upper triangular matrices. Its diagonal subgroup is
    $\operatorname{diag}(t,t^{-1})$, and its unipotent subgroup is
    generated by $Y$. In the fibre of $\mathcal Q_j$ at $[1:0]$,
    the span of all but the first weight is invariant under this
    stabilizer: $Y$ moves each vector to the next. Its quotient
    therefore defines a homogeneous line bundle $L_j$, meaning a
    line bundle with the induced $\operatorname{SL}(2,\c)$ action.
    The stabilizer acts on its fibre by $t^{-2s}$.
    Hence $L_j\cong\mathcal O_{\cp^1}(2s)$.
    Indeed, the character $t$ is that of $\mathcal O_{\cp^1}(-1)$.
    For $A\in U_s$, write $\lambda_A$ for the image of $\sigma_A$
    in $H^0(L_j)$. Choose the frame of $L_j$ on the affine
    chart obtained by transporting the image of $v_{j,0}$ at
    $[1:0]$ by $e^{wX}$. In this frame,
    \eqref{Equation: right inverse representation identities} gives
    \begin{align}
        \label{Equation: right inverse scalar fundamental sections}
        \lambda_{A_{s,a}}(w)
        =(-1)^a\frac{(2s)!}{(2s-a)!}w^a,
        \qquad 0\leq a\leq2s.
    \end{align}
    This follows by expanding $e^{-w\operatorname{ad}X}A_{s,a}$
    and reading its $A_{s,0}$ coefficient.
    At $w=0$, the maps $U_s\to(\mathcal Q_j)_0$ and
    $U_s\to(J^mL_j)_0$, given respectively by evaluation and
    $A\mapsto j^m\lambda_A(0)$, have the same kernel,
    $\operatorname{span}\{A_{s,a}:a\geq m+1\}$.
    At $[1:0]$, both maps are equivariant under the stabilizer and
    surjective. Their common kernel therefore identifies the two
    quotient representations of the stabilizer. Transporting this
    identification by $\operatorname{SL}(2,\c)$ gives a well-defined
    holomorphic bundle isomorphism $\mathcal Q_j\cong J^mL_j$.
    The final weight is the kernel of
    truncation to $J^{m-1}L_j$, so, compatibly with $q_0$,
    \begin{align}
        \label{Equation: right inverse graded jet bundles}
        \mathcal E_j/\mathcal E_{j-1}
        &\cong J^{j-1}\mathcal O_{\cp^1}(4j-2),\\
        \mathcal V_j/\mathcal V_{j-1}
        &\cong J^{j-2}\mathcal O_{\cp^1}(4j-2).
        \nonumber
    \end{align}
    Similarly, $\mathcal E_1\cong\mathcal O_{\cp^1}(2)$.
    The quotient $q$ is defined by the invariant bilinear form,
    so the complex orthogonal action preserves $\ker q$.
    Thus $e^{tA}\circ\tilde\varphi_n$ is horizontal for every $t$, and
    differentiating gives $L_{\tilde\varphi_n}\sigma_A=0$.
    The Leibniz rule for $L_{\tilde\varphi_n}$ is
    \begin{align*}
        L_{\tilde\varphi_n}(h\xi)=hL_{\tilde\varphi_n}\xi+\partial h\otimes q_0(\xi).
    \end{align*}
    Since the fundamental sections locally generate $\mathcal E_j$,
    the operator maps its holomorphic sections into
    $\omega_{\cp^1}\otimes\mathcal V_j$.
    On the $j$-th quotient, multiplication by a holomorphic
    scalar $h$ therefore produces the extra term
    $\partial h\otimes\operatorname{tr}_{m-1}(u)$, where
    $\operatorname{tr}_{m-1}:J^mL_j\to J^{m-1}L_j$ drops the last
    derivative. This is what it means for the symbol to be jet
    truncation. A general local section $(u_0,\ldots,u_m)$ of
    $J^mL_j$ need not satisfy $u_a=\partial_w^a u_0$.
    A section of the form $j^m\lambda=(\lambda,\lambda',\ldots,
    \lambda^{(m)})$ is called holonomic.
    In derivative coordinates the first-order operator
    \begin{align*}
        (u_0,\ldots,u_m)\mapsto
        (du_0-u_1\,dw,\ldots,du_{m-1}-u_m\,dw)
    \end{align*}
    has the same symbol and kills every holonomic jet $j^m\lambda_A$.
    This formula is intrinsic: the product and chain rules for a
    change of line-bundle frame and base coordinate give exactly
    the transition rule of $\omega_{\cp^1}\otimes J^{m-1}L_j$.
    The difference from the induced operator is therefore a
    zeroth-order bundle map. It vanishes on the jets of the sections
    in \eqref{Equation: right inverse scalar fundamental sections},
    which span every fibre, so the difference is zero.
    Now let $f:\Sigma\to\cp^1$ and $\psi=\tilde\varphi_n\circ f$ be as in the statement. The same argument applies along $\psi$.
    The sections $f^*\sigma_A$ locally generate $f^*\mathcal E_j$
    and satisfy $L_\psi(f^*\sigma_A)=0$, because they are the
    infinitesimal deformations $e^{tA}\circ\psi$. The Leibniz rule
    shows that $L_\psi$ maps local holomorphic sections of
    $f^*\mathcal E_j$ into
    $\omega_\Sigma\otimes f^*\mathcal V_j$.
    Its induced quotient operator has symbol the pulled-back
    truncation and kills $f^*j^m\lambda_A$.
    These sections span every fibre, including above branch
    points. Comparing the symbol and these sections identifies the
    quotient operator as
    \begin{align}
        \label{Equation: right inverse quotient operator}
        D_{m,f}(u_0,\ldots,u_m)
        =\bigl(du_a-u_{a+1}\,df\bigr)_{a=0}^{m-1},
    \end{align}
    from $f^*J^m\mathcal O_{\cp^1}(4m+2)$ to
    $\omega_\Sigma\otimes f^*J^{m-1}\mathcal O_{\cp^1}(4m+2)$.
    In this formula $df$ is the differential of a local base
    coordinate composed with $f$. The formula agrees on overlaps
    by the same chain rule, including where $df=0$.
\end{proof}

\begin{lemma}[Polynomially bounded holomorphic right inverse]
\label{Lemma: polynomial holomorphic right inverse}
    Fix $n\geq2$. There exist $C>0$, $N\in\n$, and $\ell_0\in\n$, depending only on $n$, the fixed geometry, and the fixed constants in the definition of balancedness, with the following property. Let $f:\Sigma\rightarrow\cp^1$ be a balanced holomorphic cover of degree $\ell\geq\ell_0$, and set $\psi:=\tilde\varphi_n\circ f$. Then there exists a complex-linear map
    \begin{align*}
        S_\psi:H^0(\omega_\Sigma\otimes\psi^*V\mathscr Z_n)\rightarrow H^0(\psi^*T^{1,0}\mathscr Z_n)
    \end{align*}
    satisfying
    \begin{align}\label{Equation: polynomial holomorphic right inverse}
        L_\psi S_\psi&=\operatorname{Id},\\
        \|S_\psi\beta\|_{L^2(\Sigma)}
        &\leq C\ell^N\|\beta\|_{L^2(\Sigma)}
        \qquad\forall\,\beta\in
        H^0(\omega_\Sigma\otimes\psi^*V\mathscr Z_n).
        \nonumber
    \end{align}
\end{lemma}
\begin{proof}
    Let $E_0,V_0$ and the filtrations $\mathcal E_j,\mathcal V_j$
    be those of Lemma~\ref{Lemma: filtration of linearized horizontality},
    and set
    \begin{align*}
        E:=f^*E_0,
        \qquad
        F:=\omega_\Sigma\otimes f^*V_0.
    \end{align*}
    By that lemma, $L_\psi$ preserves the pulled-back filtrations,
    with quotient operators $D_{m,f}$ given by
    \eqref{Equation: right inverse quotient operator}.
    We first solve these quotient equations with polynomial norm
    bounds, and then lift the solutions through the filtrations.
    Throughout the proof, constants and exponents may increase,
    but depend only on the data listed in the statement. We retain
    $n_0$ for the fixed exponent in
    Definition~\ref{Definition: m-balanced cover}(iii).
    Only finitely many derivative estimates are used, with orders
    bounded in terms of $n$. Choose fixed smooth Hermitian metrics
    on the bundles and subquotients over $\cp^1$ in the preceding
    lemma and below. Comparisons between these metrics remain
    uniform after pullback by $f$.

    \medskip
    \noindent
    \textbf{Step 1: quantitative local control of the cover}.
    The upper bound for $df$ implies that a disk of radius
    $c\ell^{-\frac12}$ on $\Sigma$ maps into a fixed-size target
    coordinate disk, for a sufficiently small fixed $c>0$.
    Using uniformly controlled source coordinates and target
    coordinates centered by Fubini--Study isometries, the Cauchy
    estimates give
    \begin{align}
        \label{Equation: right inverse Cauchy estimates}
        |\partial_z^kf|&\leq C_k\ell^{k/2}\qquad(k\geq1),\\
        |\partial_z^ku|&\leq C_k\ell^{(k+1)/2}\|u\|_{L^2(\Sigma)}
        \qquad(k\geq0)
        \nonumber
    \end{align}
    at the center of the disk. In the second estimate $u$ denotes
    the coefficient vector of a holomorphic section of a fixed
    pullback bundle, possibly tensored with a fixed power of
    $\omega_\Sigma$. Its metric is uniformly equivalent to the
    coefficient norm on these disks. All fixed-order covariant
    estimates obtained from these inequalities also have polynomial
    constants. Squared local $L^2$ Cauchy estimates, summed over a
    bounded-overlap natural-scale cover, give in particular
    \begin{align}
        \label{Equation: right inverse holomorphic derivative bound}
        \|L_\psi u\|_{L^2(\Sigma)}
        \leq C\ell^{\frac12}\|u\|_{L^2(\Sigma)}
        \qquad\forall\,u\in H^0(E).
    \end{align}
    Here the first-order coefficients are fixed base coefficients,
    while the zeroth-order coefficients contain one factor of $df$.
    At a branch point $p$, condition~(iii) in
    Definition~\ref{Definition: m-balanced cover} gives
    \begin{align*}
        |\nabla^2f(p)|_{g_0}\geq c\ell^{(1-n_0)/2}>0.
    \end{align*}
    Since $df(p)=0$, the connection terms in the coordinate
    expression for $\nabla^2f(p)$ vanish. Thus $f''(p)\neq0$
    in holomorphic coordinates, so every zero of $df$ is simple.
    Let $R_f$ be the reduced branch divisor, and write
    $g$ for the genus of $\Sigma$.
    Since $df$ is a holomorphic section of
    $\omega_\Sigma\otimes f^*T^{1,0}\cp^1$,
    \begin{align}
        \label{Equation: right inverse branch identity}
        \mathcal O(R_f)\cong\omega_\Sigma\otimes f^*\mathcal O_{\cp^1}(2),
        \qquad \deg R_f=2\ell+2g-2.
    \end{align}
    We require polynomial lower bounds for $df$ outside small branch
    disks. In a normalized source coordinate $z_p$ centered at $p$,
    write $D_p(t):=\{|z_p|<t\}$.
    There exist fixed $B,M>0$ and $\delta=c_0\ell^{-M}$ such
    that coordinate disks $D_p(4\delta)$ centered at the branch
    points are pairwise disjoint and
    \begin{align}
        \label{Equation: right inverse away from branch}
        |df|_{g_0}\geq c\ell^{-B}\delta
        \quad\mbox{on }\Sigma\smallsetminus
        \bigcup_{p\in\operatorname{spt}(R_f)}D_p(\delta).
    \end{align}
    We justify this assertion. Choose $B$ sufficiently large in terms
    of $n_0$, for instance $B=n_0+1$. At a branch point,
    normalized coordinates and condition~(iii) of balancedness give
    $|f''(0)|\geq c\ell^{-B}$. On natural-scale coordinate disks,
    \eqref{Equation: right inverse Cauchy estimates} gives
    $|f'''|\leq C\ell^{3/2}$. Choose $M>B+3$ and then $c_0>0$
    sufficiently small. Taylor's formula for $a=f'$ gives
    \begin{align*}
        |a(z)-a'(0)z|\leq C\ell^{3/2}|z|^2,
        \qquad
        |a(z)|\geq\tfrac12|a'(0)||z|
    \end{align*}
    on disks of radius any prescribed fixed multiple of $\delta$.
    In particular, no second zero lies in those disks. Uniform
    comparison of the coordinates with $g_0$, and an adjustment of
    $c_0$, give the stated pairwise disjointness.
    To prove \eqref{Equation: right inverse away from branch},
    suppose a point $x$ outside the smaller disks has
    $|df(x)|<\varepsilon\ell^{-B}\delta$, where $\varepsilon>0$
    will be fixed sufficiently small. Choose normalized coordinates
    centered at $x$ and $f(x)$ and write $a=f'$.
    Condition~(iii) of balancedness implies
    $|a'(0)|\geq c\ell^{-B}$; the connection terms relating
    the covariant and coordinate second derivatives contain $df$
    and are absorbed at this threshold. Also $a(0)\neq0$.
    On the circle of radius $r=2|a(0)|/|a'(0)|$,
    \begin{align*}
        |a(0)+a'(0)z|&\geq|a(0)|,\\
        |a(z)-a(0)-a'(0)z|
        &\leq C\ell^{3/2}r^2
        <|a(0)|.
    \end{align*}
    The strict inequality holds for large $\ell$, since
    $\ell^{B+3/2}\delta\to0$. Rouch\'e's theorem gives a zero
    within distance $C\varepsilon\delta$ of $x$. Uniform coordinate
    comparisons then place $x$ inside the smaller coordinate disk
    of this zero, a contradiction if $\varepsilon$ is sufficiently
    small. This proves \eqref{Equation: right inverse away from branch}.
    In particular, the local reciprocal $1/f'$ and each fixed-order
    derivative of $1/f'$ are bounded by a power of $\ell$ on this
    complement: differentiation produces products of derivatives
    of $f$ divided by powers of $f'$.

    \medskip
    \noindent
    \textbf{Step 2: a quantitative $\bar\partial$-correction}.
    For fixed integers $a>0$ and $b$, set
    \begin{align*}
        B_{a,b}(f):=f^*\mathcal O_{\cp^1}(a)\otimes\omega_\Sigma^b,
    \end{align*}
    equipped with the product of the pullback Fubini--Study metric
    and the metric induced by $g_0$ on $\omega_\Sigma^b$.
    For the Fubini--Study metric $h_a$ on $\mathcal O_{\cp^1}(a)$,
    there is a fixed $\kappa_{\mathrm{FS}}>0$ such that
    \begin{align*}
        \frac{i}{2\pi}F_{f^*h_a}
        =a\kappa_{\mathrm{FS}}|df|_{g_0}^2\omega_0.
    \end{align*}
    Apply the spectral-gap argument in Step~1 of the proof of
    Proposition~\ref{Proposition: holomorphic peak sections}, taking
    \begin{align*}
        m=0,\qquad d=a\ell,\qquad r=\sqrt a,
        \qquad \rho:=\kappa_{\mathrm{FS}}\ell^{-1}|df|_{g_0}^2.
    \end{align*}
    Conditions~(i)--(ii) in
    Definition~\ref{Definition: m-balanced cover} give
    \begin{align*}
        0\leq\rho&\leq C,\\
        |d\rho|_{g_0}
        &\leq C\ell^{-1}|df|_{g_0}|\nabla^2f|_{g_0}
        \leq C\ell^{1/2}\leq C_a\sqrt d,\\
        \fint_{B_{r d^{-1/2}}(x)}\rho\,d\vol_{g_0}
        &=\fint_{B_{\ell^{-1/2}}(x)}\rho\,d\vol_{g_0}\geq c.
    \end{align*}
    Thus the averaging radius is exactly the one in balancedness.
    The spectral-gap argument remains valid for the present $g_0$:
    a fixed lower bound $K_{g_0}\geq-C_0$ contributes only a bounded
    zeroth-order term, absorbed for sufficiently large $\ell$.
    The fixed-twist
    estimate \eqref{Equation: fixed twist adjoint coercivity},
    proved in Step 1 of the proof of
    Lemma~\ref{Lemma: del bar solution with weighted norm},
    then gives, for all sufficiently large $\ell$,
    \begin{align*}
        \|\bar\partial^*\gamma\|_{L^2}^2
        \geq c_{a,b}\ell\|\gamma\|_{L^2}^2
        \qquad\forall\,\gamma\in\Omega^{0,1}(B_{a,b}(f)).
    \end{align*}
    Since $\Sigma$ is a complex curve, the elliptic operator
    $\bar\partial\bar\partial^*$ is therefore invertible, and
    \begin{align*}
        T_{a,b,f}:=
        \bar\partial^*(\bar\partial\bar\partial^*)^{-1}
    \end{align*}
    is a complex-linear right inverse satisfying
    \begin{align}\label{Equation: right inverse delbar solution estimate}
        \|T_{a,b,f}\alpha\|_{L^2}
        \leq C_{a,b}\ell^{-\frac12}\|\alpha\|_{L^2}.
    \end{align}
    Smooth data give smooth solutions by elliptic regularity.
    A fixed holomorphic splitting gives the same conclusion on
    $f^*B$ whenever $B\to\cp^1$ splits into line bundles of
    positive degree; comparisons with other fixed smooth bundle
    metrics remain uniform after pullback.

    \medskip
    \noindent
    \textbf{Step 3: exact cancellation of poles in a quotient equation}.
    Fix $1\leq m\leq n-1$ and set
    $L_{m+1}:=\mathcal O_{\cp^1}(4m+2)$.
    Let
    \begin{align*}
        \beta\in H^0\bigl(\omega_\Sigma\otimes f^*J^{m-1}L_{m+1}\bigr).
    \end{align*}
    Choose local holomorphic coordinates $z$ on $\Sigma$ and
    $w$ on $\cp^1$, with $f$ mapping the source chart into the
    target chart. Use the derivative coordinates on the jet bundles
    determined by $w$ and a local holomorphic frame of $L_{m+1}$,
    and write the coefficients of $\beta$ as
    $(b_0,\ldots,b_{m-1})\,dz$. In the following formulas,
    $u_a':=\partial_z u_a$ and $f':=\partial_z(w\circ f)$.
    For a prescribed global scalar
    section $u_0=s\in H^0(f^*L_{m+1})$, the only possible solution of
    $D_{m,f}u=\beta$ away from $R_f$ is
    \begin{align}
        \label{Equation: right inverse scalar recursion}
        u_{a+1}=\frac{u_a'-b_a}{f'},
        \qquad 0\leq a<m.
    \end{align}
    These local solutions agree on overlaps: the operator is intrinsic,
    and a solution with prescribed scalar component $u_0$ is unique
    wherever $df\neq0$. Thus they give a global meromorphic section
    of $f^*J^mL_{m+1}$.
    At a simple branch point, the holomorphic local normal
    form provides source and target coordinates in which $f(z)=z^2$.
    We use this normal form only to count the jet conditions;
    all quantitative estimates below use the uniformly controlled
    coordinates of Step~1. Writing $s=\sum_{k\geq0}c_kz^k$, the first
    quotient is holomorphic if and only if $s'(0)=b_0(0)$.
    Once $u_1,\ldots,u_{a-1}$ are holomorphic, the next condition is
    $u_{a-1}'(0)=b_{a-1}(0)$. The coefficient of $c_{2a-1}$ in
    $u_{a-1}'(0)$ is
    \begin{align}
        \label{Equation: right inverse branch triangular coefficient}
        \frac{(2a-1)(2a-3)\cdots3}{2^{a-1}}\neq0,
        \qquad 1\leq a\leq m,
    \end{align}
    Here the empty product equals one. Indeed, applying
    $(2z)^{-1}\partial_z$ exactly $a-1$ times to
    $c_{2a-1}z^{2a-1}$ leaves
    $c_{2a-1}(2a-1)(2a-3)\cdots3\,z/2^{a-1}$.
    Scalar coefficients of degree greater than $2a-1$ leave terms
    of order at least two and do not contribute to the derivative
    at zero. Consequently these are
    $m$ triangular linear conditions on the $(2m-1)$-jet of $s$.
    One may set its even coefficients to zero and solve successively
    for its odd coefficients. All dependence on $\beta$ is linear.
    For quantitative purposes use the normalized source coordinates
    from Step 1, without rescaling them to make $f=z^2$.
    If $f'(z)=a_1z+O(z^2)$, the denominator $2^{a-1}$ in
    \eqref{Equation: right inverse branch triangular coefficient}
    is replaced by $a_1^{a-1}$. The other coefficients involve only
    finitely many derivatives of $f$ and $\beta$, and powers of
    $a_1^{-1}$. This follows by expanding the recurrence
    \eqref{Equation: right inverse scalar recursion} in Taylor series.
    The lower bound for $|a_1|$ and
    \eqref{Equation: right inverse Cauchy estimates} therefore show
    that every prescribed scalar coefficient is bounded by
    $C\ell^P\|\beta\|_{L^2}$ for a fixed exponent $P$.
    Matching these jets removes the poles exactly. At each stage,
    the numerator is holomorphic and vanishes at the branch point,
    while $f'$ has a simple zero. Its quotient is therefore
    holomorphic. Induction proves this for $u_1,\ldots,u_m$.

    \medskip
    \noindent
    \textbf{Step 4: a polynomially bounded solution of each quotient}.
    Represent the jets just obtained by polynomials $P_p$ of degree
    at most $2m-1$ in the local pulled-back frames of $f^*L_{m+1}$.
    Choose cutoffs $\chi_p$ supported in $D_p(2\delta)$ and equal
    to one on $D_p(\delta)$, with
    $|\nabla^k\chi_p|\leq C_k\delta^{-k}$. Their supports are
    disjoint by Step 1. Extending each term by zero, define
    \begin{align*}
        s_{\mathrm{app}}:=\sum_{p\in\operatorname{spt}(R_f)}\chi_pP_p.
    \end{align*}
    This is a global smooth section of $f^*L_{m+1}$, holomorphic near
    every branch point and realizing all prescribed jets.
    The number of terms is $2\ell+2g-2$, and every cutoff derivative
    is bounded by a fixed power of $\ell$. The coefficient estimates
    from Step 3 imply
    \begin{align}
        \label{Equation: right inverse approximate potential bound}
        \|s_{\mathrm{app}}\|_{L^2}
        +\|\bar\partial s_{\mathrm{app}}\|_{L^2}
        \leq C\ell^P\|\beta\|_{L^2},
    \end{align}
    after increasing $P$.
    Fix a holomorphic identification $T^{1,0}\cp^1\cong\mathcal O_{\cp^1}(2)$
    and write
    \begin{align*}
        \tau:=df\in H^0\bigl(\omega_\Sigma\otimes f^*\mathcal O_{\cp^1}(2)\bigr).
    \end{align*}
    The correction must preserve the prescribed jets through order
    $2m-1$, so it must vanish to order at least $2m$ at each point
    of $R_f$. Since $\tau$ has precisely these simple zeros,
    multiplication by $\tau^{2m}$ realizes the natural inclusion
    of the divisor twist, and identifies that twist as
    \begin{align}
        \label{Equation: right inverse residual positive bundle}
        f^*L_{m+1}(-2mR_f)
        &\cong f^*L_{m+1}\otimes
        \bigl(\omega_\Sigma\otimes f^*\mathcal O_{\cp^1}(2)\bigr)^{-2m}\\
        &\cong f^*\mathcal O_{\cp^1}(2)\otimes\omega_\Sigma^{-2m}
        =:\mathcal L_{m,f}.
        \nonumber
    \end{align}
    This is the positivity needed for interpolation: imposing the
    jets removes the factor $f^*\mathcal O_{\cp^1}(4m)$, but leaves
    the positive factor $f^*\mathcal O_{\cp^1}(2)$ and only a fixed
    twist on $\Sigma$.
    The form
    \begin{align*}
        \alpha:=\tau^{-2m}\bar\partial s_{\mathrm{app}}
    \end{align*}
    extends by zero near $R_f$ to a smooth
    $\mathcal L_{m,f}$-valued $(0,1)$-form.
    Its support is contained in the cutoff annuli, where
    \eqref{Equation: right inverse away from branch} applies.
    Hence
    \begin{align*}
        \|\alpha\|_{L^2}\leq C\ell^P\|\beta\|_{L^2}
    \end{align*}
    for a possibly larger fixed $P$. Apply Step 2 with $a=2$ and
    $b=-2m$ to solve $\bar\partial v=\alpha$ linearly with
    $\|v\|_{L^2}\leq C\ell^{-\frac12}\|\alpha\|_{L^2}$, and set
    \begin{align}
        \label{Equation: right inverse corrected scalar potential}
        s:=s_{\mathrm{app}}-\tau^{2m}v.
    \end{align}
    Then $s$ is holomorphic. Near each branch point, $v$ is
    holomorphic, and $\tau^{2m}$ vanishes to order $2m$.
    Thus the prescribed jets through order $2m-1$ are unchanged.
    Since $\|\tau\|_{L^\infty}\leq C\ell^{\frac12}$, we also have
    \begin{align}
        \label{Equation: right inverse scalar potential bound}
        \|s\|_{L^2}\leq C\ell^P\|\beta\|_{L^2}.
    \end{align}
    Cauchy estimates now bound all required derivatives of the
    holomorphic $s$ by a further fixed power of $\ell$.
    Apply \eqref{Equation: right inverse scalar recursion} to this
    scalar potential. The resulting meromorphic section has no poles
    by Step 3, hence is a global holomorphic section $u$ of $f^*J^mL_{m+1}$.
    Outside the disks $D_p(\delta)$, the recurrence, the Cauchy
    estimates, and the polynomial lower bound for $df$ give
    polynomial pointwise bounds for all its components.
    Inside a disk, use the holomorphic pulled-back jet frame defined
    on $D_p(2\delta)$. The components of $u$ are holomorphic and
    their values on $\partial D_p(\delta)$ satisfy the same bounds.
    The maximum modulus principle extends these bounds across the disk.
    Metrics of these fixed base frames are uniformly comparable,
    so
    \begin{align}
        \label{Equation: right inverse quotient solution bound}
        D_{m,f}u=\beta,
        \qquad\|u\|_{L^2}\leq C_m\ell^{N_m}\|\beta\|_{L^2}.
    \end{align}
    For fixed $f$, all coordinates, frames, and cutoffs have been
    chosen independently of $\beta$. The branch-jet system, the
    $\bar\partial$-solution, and the recurrence are complex linear
    in $\beta$. We have therefore constructed a complex-linear
    right inverse $R_{m,f}$ satisfying
    \eqref{Equation: right inverse quotient solution bound}.

    \medskip
    \noindent
    \textbf{Step 5: holomorphic lifting and the full right inverse}.
    Each fixed bundle $\mathcal E_j$ splits on $\cp^1$ into line
    bundles of degree at least two. To verify this, the standard jet
    filtration of $J^{j-1}\mathcal O_{\cp^1}(4j-2)$ has line
    quotients
    \begin{align*}
        \mathcal O_{\cp^1}(4j-2-2a),
        \qquad 0\leq a\leq j-1,
    \end{align*}
    of degrees at least $2j$.
    Refining the filtration of $\mathcal E_j$ by these jet filtrations
    and twisting by $\mathcal O_{\cp^1}(-2)$ gives nonnegative-degree
    line quotients. Induction on the resulting exact sequences shows
    that $H^1$ vanishes and all fibres are generated by global sections:
    quotient sections lift because the preceding $H^1$ vanishes.
    Birkhoff--Grothendieck therefore implies that every summand of
    $\mathcal E_j$ has degree at least two.
    Put
    \begin{align*}
        E^{(j)}:=f^*\mathcal E_j,
        \qquad F^{(j)}:=\omega_\Sigma\otimes f^*\mathcal V_j.
    \end{align*}
    For $j\geq2$, choose a fixed smooth right inverse of the base
    quotient $\mathcal E_j\rightarrow\mathcal E_j/\mathcal E_{j-1}$.
    Pulling it back lifts a holomorphic quotient section $u$ to a
    smooth section $\widetilde u$ of $E^{(j)}$ with
    \begin{align*}
        \|\widetilde u\|_{L^2}&\leq C\|u\|_{L^2},\\
        \bar\partial\widetilde u&\in\Omega^{0,1}(E^{(j-1)}),
        \qquad
        \|\bar\partial\widetilde u\|_{L^2}
        \leq C\ell^{\frac12}\|u\|_{L^2}.
    \end{align*}
    The derivative estimate follows from the chain rule for the
    fixed smooth splitting and $|df|\leq C\ell^{\frac12}$; there is
    no derivative of $u$ because it is holomorphic.
    By the positivity just proved and Step 2, the error has a
    solution in $E^{(j-1)}$ of norm at most $C\|u\|_{L^2}$.
    Subtraction gives a complex-linear holomorphic lifting operator
    \begin{align}
        \label{Equation: right inverse holomorphic lifting}
        H_{j,f}:H^0(E^{(j)}/E^{(j-1)})\rightarrow H^0(E^{(j)}),
        \qquad \|H_{j,f}u\|_{L^2}\leq C\|u\|_{L^2}.
    \end{align}
    We now solve the full equation by induction on the filtration.
    A solution of the $j$-th quotient equation leaves an error in
    $F^{(j-1)}$, which can be removed using the preceding stage.
    This does not require a holomorphic splitting of either
    filtration. For $j=1$, the target $F^{(1)}$ is zero, so take
    $S^{(1)}=0$.
    Suppose $S^{(j-1)}:H^0(F^{(j-1)})\rightarrow H^0(E^{(j-1)})$
    is a polynomially bounded right inverse.
    Let $\rho_j:F^{(j)}\to F^{(j)}/F^{(j-1)}$ be the quotient.
    Using \eqref{Equation: right inverse graded jet bundles}, set
    \begin{align*}
        u&:=R_{j-1,f}(\rho_j\beta),
        \qquad\eta:=H_{j,f}u,\\
        S^{(j)}\beta
        &:=\eta+S^{(j-1)}(\beta-L_\psi\eta)
        \qquad\bigl(\beta\in H^0(F^{(j)})\bigr).
    \end{align*}
    Compatibility of the quotient operators shows that
    $\beta-L_\psi\eta\in H^0(F^{(j-1)})$, and therefore
    $L_\psi S^{(j)}\beta=\beta$.
    The norm bounds
    \eqref{Equation: right inverse quotient solution bound},
    \eqref{Equation: right inverse holomorphic lifting}, and
    \eqref{Equation: right inverse holomorphic derivative bound}
    give, for fixed exponents,
    \begin{align*}
        \|\eta\|_{L^2}&\leq C\ell^{N_{j-1}}\|\beta\|_{L^2},\\
        \|\beta-L_\psi\eta\|_{L^2}
        &\leq\|\beta\|_{L^2}+C\ell^{\frac12}\|\eta\|_{L^2}.
    \end{align*}
    Thus $S^{(j)}$ is again polynomially bounded and complex linear.
    There are only $n$ steps. Since $E^{(n)}=E$ and $F^{(n)}=F$,
    the map $S_\psi:=S^{(n)}$ satisfies
    \eqref{Equation: polynomial holomorphic right inverse}, after
    increasing the exponent to an integer.
\end{proof}

The right inverse constructed above gives the projected spectral gap directly by duality. In the following statement, $L_\psi^*$ denotes the formal $L^2$-adjoint of the intrinsic operator $L_\psi$.
\begin{proposition}[Projected polynomial spectral gap]\label{Proposition: projected polynomial spectral gap for L psi}
    Let $\varphi_n:\mathbb S^2\rightarrow\mathbb S^{2n}$ be a linearly full superminimal immersion of constant Gaussian curvature corresponding to a special orbit of the irreducible representation $\operatorname{SO}(3)\rightarrow\operatorname{SO}(2n+1)$ and let $\tilde\varphi_n:\mathbb{CP}^1\rightarrow\mathscr Z_n$ be its unique twistor lift, of degree $d_n$. Let $f:\Sigma\rightarrow\mathbb{CP}^1$ be a balanced holomorphic cover of degree $\ell\in\mathbb N\smallsetminus\{0\}$, and set
    \begin{align*}
        \psi
        :=
        \tilde\varphi_n\circ f
        \in
        \HH_d^f(\Sigma,\mathscr Z_n),
        \qquad
        d:=d_n\ell.
    \end{align*}
    Define
    \begin{align*}
        E
        :=
        \psi^*T^{1,0}\mathscr Z_n,
        \qquad
        V
        :=
        \psi^*V\mathscr Z_n,
    \end{align*}
    and let
    \begin{align*}
        Q_E:
        L^2(E)
        \rightarrow
        H^0(E)
    \end{align*}
    denote the $L^2$-orthogonal projection onto the finite-dimensional linear subspace $H^0(E)\subset L^2(E)$. Then there exist constants $c>0$, $N\in\n$, and $\ell_0\in\n$, depending only on $n$, the geometry of $(\Sigma,g_0)$, and the fixed constants in the definition of balancedness, such that if $\ell\geq\ell_0$ we have
    \begin{align}\label{Equation: final projected polynomial spectral gap}
        \big\|
            Q_E(L_\psi^*\xi)
        \big\|_{L^2(\Sigma)}
        \geq
        c\,d^{-N}
        \|\xi\|_{L^2(\Sigma)}
        \qquad
        \forall\,
        \xi\in H^0(\omega_\Sigma\otimes V).
    \end{align}
    In particular, since 
    \begin{align*}
        \big(L_{\psi}|_{H^0(E)}\big)^*=Q_E\circ\big(L_{\psi}^*|_{H^0(\omega_{\Sigma}\otimes V)}\big),
    \end{align*}
    we have that $L_{\psi}:H^0(E)\to H^0(\omega_{\Sigma}\otimes V)$ is surjective.
\end{proposition}
\begin{proof}
    By Lemma~\ref{Lemma: polynomial holomorphic right inverse},
    for all sufficiently large $\ell$ there exists a complex-linear
    map
    \begin{align*}
        S_\psi:H^0(\omega_\Sigma\otimes V)\rightarrow H^0(E)
    \end{align*}
    such that $L_\psi S_\psi=\operatorname{Id}$ and
    $\|S_\psi\xi\|_{L^2}\leq C\ell^N\|\xi\|_{L^2}$.
    Since $S_\psi\xi$ is holomorphic, for every
    $\xi\in H^0(\omega_\Sigma\otimes V)$ we have
    \begin{align*}
        \|\xi\|_{L^2}^2
        &=\langle L_\psi S_\psi\xi,\xi\rangle_{L^2}\\
        &=\langle S_\psi\xi,Q_E L_\psi^*\xi\rangle_{L^2}\\
        &\leq C\ell^N\|\xi\|_{L^2}
            \|Q_E L_\psi^*\xi\|_{L^2}.
    \end{align*}
    Dividing by $\|\xi\|_{L^2}$ when $\xi\neq0$ gives
    \begin{align*}
        \|Q_E L_\psi^*\xi\|_{L^2}
        \geq C^{-1}\ell^{-N}\|\xi\|_{L^2}.
    \end{align*}
    Since $d=d_n\ell$ and $d_n$ depends only on $n$, this is
    \eqref{Equation: final projected polynomial spectral gap}.
    Surjectivity also follows directly from the existence of $S_\psi$.
\end{proof}
\subsection{Proof of the main statement}
We now combine the construction of balanced covers with the deformation theory developed above.
\begin{proof}[Proof of Theorem~\ref{Theorem: main statement 1}]
    Fix $n\geq2$ and let $\varphi_n:\mathbb S^2\rightarrow\mathbb S^{2n}$ be the superminimal immersion of constant Gaussian curvature corresponding to a special orbit of the irreducible representation $\operatorname{SO}(3)\rightarrow\operatorname{SO}(2n+1)$ and let $\tilde\varphi_n:\mathbb{CP}^1\rightarrow\mathscr Z_n$ be its unique twistor lift. We denote its degree by
    \begin{align*}
        d_n:=\deg(\tilde\varphi_n).
    \end{align*}
    By Proposition~\ref{Proposition: existence of balanced covers},
    we can find $\ell_0\in\n$ such that, for every $\ell\geq\ell_0$, there exists a balanced holomorphic cover $f_\ell:\Sigma\rightarrow\cp^1$ of degree $\ell$. Set
    \begin{align}\label{Equation: balanced Boruvka sequence main proof}
        \psi_\ell
        :=
        \tilde\varphi_n\circ f_\ell,
        \qquad
        d_\ell
        :=
        d_n\ell.
    \end{align}
    Since $f_\ell$ is non-constant, it is surjective. Hence
    \begin{align*}
        \psi_\ell(\Sigma)
        =
        \tilde\varphi_n(\cp^1),
    \end{align*}
    and therefore $\psi_\ell$ is linearly full. Moreover,
    $\psi_\ell$ is holomorphic and horizontal, and
    \begin{align*}
        \deg(\psi_\ell)
        =
        \ell\,\deg(\tilde\varphi_n)
        =
        d_\ell.
    \end{align*}
    Thus, $\psi_\ell\in\HH_{d_\ell}^{f}(\Sigma,\mathscr Z_n)$. Set
    \begin{align*}
        E_\ell
        :=
        \psi_\ell^*T^{1,0}\mathscr Z_n,
        \qquad
        V_\ell
        :=
        \psi_\ell^*V\mathscr Z_n.
    \end{align*}
    By
    Corollary~\ref{Corollary: splitting along balanced Boruvka covers},
    there exist integers
    \begin{align*}
        a_1,\ldots,a_{\frac{n(n+1)}{2}}\geq n+1,
    \end{align*}
    depending only on $n$, such that
    \begin{align*}
        E_\ell
        &\cong
        \bigoplus_{i=1}^{\frac{n(n+1)}{2}}
        f_\ell^*\mathcal O_{\cp^1}(a_i),
        \\
        V_\ell
        &\cong
        f_\ell^*\mathcal O_{\cp^1}(2n+2)^{
            \oplus\frac{n(n-1)}{2}
        }.
    \end{align*}
    In particular,
    \begin{align*}
        \deg
        \big(
            f_\ell^*\mathcal O_{\cp^1}(a_i)
        \big)
        =
        a_i\ell
        \rightarrow+\infty
    \end{align*}
    as $\ell\to+\infty$. Hence, after increasing $\ell_0$ if
    necessary,
    \begin{align}\label{Equation: cohomology vanishing main proof}
        H^1(E_\ell)
        =
        0
        \qquad
        \forall\,\ell\geq\ell_0.
    \end{align}
    Moreover, by Serre duality and the splitting of $V_\ell$, we have
    \begin{align}\label{Equation: vertical cohomology vanishing main proof}
        H^1(\omega_\Sigma\otimes V_\ell)
        \cong
        H^0(V_\ell^\vee)^\vee
        =0 
        \qquad
        \forall\,\ell\geq\ell_0.
    \end{align}
    By
    \eqref{Equation: cohomology vanishing main proof} and the standard
    deformation theory of holomorphic maps, the space
    \begin{align*}
        \operatorname{H}_{d_\ell}(\Sigma,\mathscr Z_n)
    \end{align*}
    of holomorphic maps of degree $d_\ell$ is, in a $C^1$-neighbourhood
    $\mathcal U_\ell$ of $\psi_\ell$, a finite-dimensional complex manifold whose tangent space at $\psi_\ell$ is naturally identified with
    \begin{align}\label{Equation: tangent holomorphic maps main proof}
        T_{\psi_\ell}\mathcal U_\ell
        \cong
        H^0(\Sigma,E_\ell).
    \end{align}
    By upper semicontinuity, after shrinking $\mathcal U_\ell$ if
    necessary,
    \begin{align*}
        H^1
        \big(
            \Sigma,
            \omega_\Sigma\otimes
            \eta^*V\mathscr Z_n
        \big)
        =
        0
        \qquad
        \forall\,\eta\in\mathcal U_\ell.
    \end{align*}
    Consequently, the spaces
    \begin{align*}
        H^0
        \big(
            \Sigma,
            \omega_\Sigma\otimes
            \eta^*V\mathscr Z_n
        \big),
        \qquad
        \eta\in\mathcal U_\ell,
    \end{align*}
    form the fibres of a holomorphic vector bundle
    \begin{align*}
        \mathcal G_\ell
        \rightarrow
        \mathcal U_\ell.
    \end{align*}
    Since $H\mathscr Z_n$ is a holomorphic subbundle and $\Pi_{\operatorname V}:T^{1,0}\mathscr Z_n\to V\mathscr Z_n$ is a holomorphic bundle morphism, the vertical component of the differential defines a holomorphic section
    \begin{align*}
        \mathfrak h_\ell:
        \mathcal U_\ell
        &\rightarrow
        \mathcal G_\ell,
        \\
        \eta
        &\mapsto
        \Pi_{\operatorname V}(\partial\eta).
    \end{align*}
    Its zero set is precisely the set of holomorphic horizontal maps
    contained in $\mathcal U_\ell$.
    By the definition of the linearized horizontality operator and
    \eqref{Equation: tangent holomorphic maps main proof}, the
    differential of $\mathfrak h_\ell$ at $\psi_\ell$ is
    \begin{align}\label{Equation: derivative horizontality main proof}
        d\mathfrak h_\ell(\psi_\ell)
        =
        L_{\psi_\ell}\big|_{H^0(E_\ell)}
        :
        H^0(E_\ell)
        \rightarrow
        H^0(\omega_\Sigma\otimes V_\ell).
    \end{align}
    Since $f_\ell$ is balanced, Lemma~\ref{Lemma: polynomial holomorphic right inverse} supplies a holomorphic right inverse for $L_{\psi_\ell}$, after increasing $\ell_0$ if necessary. Consequently,
    \begin{align*}
        L_{\psi_\ell}:
        H^0(E_\ell)
        \rightarrow
        H^0(\omega_\Sigma\otimes V_\ell)
    \end{align*}
    is surjective for every $\ell\geq\ell_0$.
    The holomorphic implicit-function theorem applied to
    $\mathfrak h_\ell$ therefore shows that, after possibly shrinking
    $\mathcal U_\ell$, the zero set
    \begin{align*}
        Y_\ell
        :=
        \mathfrak h_\ell^{-1}(0)
    \end{align*}
    is a complex submanifold of $\mathcal U_\ell$, containing
    $\psi_\ell$, with
    \begin{align}\label{Equation: tangent Y main proof}
        T_{\psi_\ell}Y_\ell
        =
        \ker
        \big(
            L_{\psi_\ell}
            \big|_{H^0(E_\ell)}
        \big).
    \end{align}
    Since $\psi_\ell$ is linearly full and linear fullness is an open
    condition, we may shrink $Y_\ell$ once more so that
    \begin{align}\label{Equation: Y linearly full main proof}
        Y_\ell
        \subset
        \HH_{d_\ell}^{f}(\Sigma,\mathscr Z_n).
    \end{align}
    We now compute the complex dimension of $Y_\ell$. Recall
    that
    \begin{align*}
        \operatorname{rk}_{\c}(E_\ell)
        &=
        \frac{n(n+1)}{2},
        \\
        \operatorname{rk}_{\c}(V_\ell)
        &=
        \frac{n(n-1)}{2}.
    \end{align*}
    Moreover, by
    Remark~\ref{Remark: degrees of pullback bundles},
    \begin{align*}
        \deg(E_\ell)
        &=
        2nd_\ell,
        \\
        \deg(V_\ell)
        &=
        2(n-1)d_\ell.
    \end{align*}
    Therefore, by Riemann--Roch and
    \eqref{Equation: cohomology vanishing main proof},
    \begin{align}\label{Equation: h0 E main proof}
        h^0(E_\ell)
        =
        2nd_\ell
        +
        \frac{n(n+1)}{2}(1-g).
    \end{align}
    Similarly,
    \begin{align*}
        \deg(\omega_\Sigma\otimes V_\ell)
        =
        2(n-1)d_\ell
        +
        \frac{n(n-1)}{2}(2g-2),
    \end{align*}
    and hence, by
    \eqref{Equation: vertical cohomology vanishing main proof},
    \begin{align}\label{Equation: h0 omega V main proof}
        h^0(\omega_\Sigma\otimes V_\ell)
        =
        2(n-1)d_\ell
        +
        \frac{n(n-1)}{2}(g-1).
    \end{align}
    Since
    \eqref{Equation: derivative horizontality main proof} is
    surjective, we obtain from
    \eqref{Equation: tangent Y main proof},
    \eqref{Equation: h0 E main proof}, and
    \eqref{Equation: h0 omega V main proof} that
    \begin{align*}
        \dim_{\c}Y_\ell
        &=
        h^0(E_\ell)
        -
        h^0(\omega_\Sigma\otimes V_\ell)
        \\
        &=
        2d_\ell
        +
        \left(
            \frac{n(n+1)}{2}
            +
            \frac{n(n-1)}{2}
        \right)
        (1-g)
        \\
        &=
        2d_\ell+n^2(1-g).
    \end{align*}
    Finally, define
    \begin{align*}
        X_{\Sigma,d_\ell}
        :=
        \big\{
            \pi\circ\eta\, \mbox{ : }
            \eta\in Y_\ell
        \big\}.
    \end{align*}
    By Proposition
    \ref{Proposition: correspondence holo+hor with supermin} and
    \eqref{Equation: Y linearly full main proof}, every element of
    $X_{\Sigma,d_\ell}$ is a linearly full branched superminimal
    immersion of degree $d_\ell$. Hence
    \begin{align*}
        X_{\Sigma,d_\ell}
        \subset
        \operatorname{SM}_{d_\ell}^{f}(\Sigma,\s^{2n}).
    \end{align*}
    Moreover, uniqueness of the twistor lift implies that the map
    \begin{align*}
        Y_\ell
        &\rightarrow
        X_{\Sigma,d_\ell},
        \\
        \eta
        &\mapsto
        \pi\circ\eta
    \end{align*}
    is bijective. We equip $X_{\Sigma,d_\ell}$ with the complex
    manifold structure transported through this bijection. It follows
    that
    \begin{align}\label{Equation: dimension X main proof}
        \dim_{\c}X_{\Sigma,d_\ell}
        =
        2d_\ell+n^2(1-g)
    \end{align}
    for every $\ell\geq\ell_0$.
    Finally, by \eqref{Equation: balanced Boruvka sequence main proof},
    \begin{align*}
        d_\ell
        =
        d_n\ell
        \rightarrow+\infty
        \qquad
        \mbox{ as }\,
        \ell\to+\infty.
    \end{align*}
    Discarding the finitely many initial indices
    $\ell<\ell_0$ and relabelling the resulting sequence, we obtain the
    sequence and the complex manifolds required in
    Theorem~\ref{Theorem: main statement 1}. This concludes the proof.
\end{proof}
\section{A degree bound under nonpositive Gaussian curvature}
\label{Section: degree bound under nonpositive curvature}

We establish a restriction on the negatively curved members of the families
constructed in Theorem~\ref{Theorem: main statement 1}. More generally,
nonpositive Gaussian curvature gives an upper bound on the degree of any
linearly full branched superminimal immersion into a fixed even-dimensional
sphere in terms of the genus. The estimate is independent of the conformal
structure and does not require the curvature to be bounded away from zero.
\begin{proposition}\label{Proposition: degree bound under nonpositive curvature}
    For every $n\geq2$, there exists $C_n>0$ with the following property.
    Let $\Sigma$ be a closed Riemann surface of genus $g$, and let
    $\varphi:\Sigma\to\s^{2n}$ be a linearly full branched superminimal
    immersion of degree $d$ satisfying
    \begin{align*}
        K_\varphi\leq0
        \qquad\mbox{on }\,\Sigma\smallsetminus\mathcal B_\varphi.
    \end{align*}
    Then $g\geq2$ and, denoting by $h_\Sigma$ the metric of constant Gaussian
    curvature $-1$ in the conformal class of $\Sigma$, we have
    \begin{align}\label{Equation: hyperbolic metric bound under nonpositive curvature}
        \varphi^*g_{\s^{2n}}\leq C_n h_\Sigma.
    \end{align}
    In particular,
    \begin{align}\label{Equation: degree bound under nonpositive curvature}
        d\leq C_n(g-1).
    \end{align}
\end{proposition}
\begin{proof}
    We first show that there is no flat superminimal immersion into
    $\s^{2n}$, even locally and without assuming linear fullness.
    Suppose that $F:U\to\s^{2n}$ is such an immersion. After restricting
    $U$ and choosing a flat conformal coordinate $z$, we may assume that
    \begin{align*}
        F^*g_{\s^{2n}}=|dz|^2.
    \end{align*}
    The harmonic-map equation reads
    \begin{align*}
        \partial_z\partial_{\bar z}F=-\frac12 F.
    \end{align*}
    Set $V_r:=\partial_z^rF$ for $r\geq0$. By superminimality,
    \begin{align*}
        (V_r,V_s)=0
        \qquad\mbox{for every }r,s\geq1,
    \end{align*}
    where $(\,\cdot\,,\,\cdot\,)$ denotes the complex-bilinear extension
    of the Euclidean inner product. Differentiating $(F,F)=1$ and then
    using these identities gives
    \begin{align*}
        (V_r,F)=0
        \qquad\mbox{for every }r\geq1.
    \end{align*}
    Define $G_{rs}:=(V_r,\overline{V_s})$. Thus $G_{00}=1$ and
    $G_{r0}=0$ for $r\geq1$. Since
    $\partial_{\bar z}V_r=-\frac12V_{r-1}$ for $r\geq1$, we have
    \begin{align*}
        \partial_{\bar z}G_{r,s-1}
        =-\frac12G_{r-1,s-1}+G_{rs}
        \qquad\mbox{for }r\geq s\geq1.
    \end{align*}
    An induction on $s$, followed by Hermitian symmetry, therefore yields
    \begin{align*}
        G_{rs}=2^{-r}\delta_{rs}
        \qquad\mbox{for every }r,s\geq0.
    \end{align*}
    Hence $V_0,\ldots,V_{2n+1}$ are $2n+2$ mutually orthogonal nonzero
    vectors in $\c^{2n+1}$, a contradiction.
    Returning to $\varphi$, its Gaussian curvature cannot vanish
    identically on the regular locus by the preceding argument.
    If $b_\varphi$ denotes its total branching order, the branched
    Gauss--Bonnet formula consequently gives
    \begin{align*}
        2\pi(2-2g+b_\varphi)
        =
        \int_{\Sigma\smallsetminus\mathcal B_\varphi}
        K_\varphi\,dA_\varphi
        <0.
    \end{align*}
    In particular, $g\geq2$, so $h_\Sigma$ is well defined.

    \smallskip
    \noindent
    We now prove
    \eqref{Equation: hyperbolic metric bound under nonpositive curvature}
    by contradiction. If no such constant existed for a fixed $n$, there
    would be closed Riemann surfaces $\Sigma_j$, with hyperbolic metrics
    $h_j$, and linearly full branched superminimal immersions
    $\varphi_j:\Sigma_j\to\s^{2n}$ satisfying the curvature assumption,
    such that
    \begin{align*}
        \varphi_j^*g_{\s^{2n}}=u_jh_j,
        \qquad
        M_j:=\max_{\Sigma_j}u_j\rightarrow+\infty.
    \end{align*}
    Here $u_j$ is smooth and nonnegative, and vanishes precisely at the
    branch points. Choose universal covering maps
    $p_j:\D\to\Sigma_j$ such that $u_j(p_j(0))=M_j$. Then
    \begin{align*}
        p_j^*h_j=\frac{4|dz|^2}{(1-|z|^2)^2}.
    \end{align*}
    Writing $\D_R:=\{w\in\c:|w|<R\}$, define
    \begin{align*}
        F_j:\D_{2\sqrt{M_j}}&\rightarrow\s^{2n},\\
        F_j(w)&:=
        \varphi_j\left(p_j\left(\frac{w}{2\sqrt{M_j}}\right)\right).
    \end{align*}
    These maps are conformal and harmonic. Their induced metrics are
    $F_j^*g_{\s^{2n}}=\lambda_j|dw|^2$, where
    \begin{align*}
        \lambda_j(w)
        =
        \frac{u_j\left(p_j\left(w/(2\sqrt{M_j})\right)\right)}{M_j}
        \left(1-\frac{|w|^2}{4M_j}\right)^{-2}.
    \end{align*}
    Consequently,
    \begin{align*}
        \lambda_j(0)=1,
        \qquad
        0\leq\lambda_j(w)
        \leq
        \left(1-\frac{|w|^2}{4M_j}\right)^{-2}.
    \end{align*}
    These inequalities give uniform bounds for the first derivatives
    on every fixed disk. The harmonic-map equation
    \begin{align*}
        \Delta F_j+|\nabla F_j|^2F_j=0
    \end{align*}
    and interior elliptic estimates then give, after passing to a
    subsequence, convergence in $C^\infty_{\operatorname{loc}}(\c)$
    to a conformal harmonic map
    $F_\infty:\c\to\s^{2n}$. The isotropy identities defining
    superminimality extend across branch points by continuity and
    pass to the smooth limit.
    Writing $F_\infty^*g_{\s^{2n}}=\lambda_\infty|dw|^2$, we have
    \begin{align*}
        0\leq\lambda_\infty\leq1,
        \qquad
        \lambda_\infty(0)=1.
    \end{align*}
    Choose a small connected neighbourhood $U$ of $0$ on which
    $\lambda_\infty>0$. Smooth convergence implies that the induced
    Gaussian curvature $K_\infty$ is nonpositive on $U$. Hence
    \begin{align*}
        \Delta\log\lambda_\infty
        =-2K_\infty\lambda_\infty\geq0
        \qquad\mbox{on }\,U.
    \end{align*}
    Since $\log\lambda_\infty\leq0$ and
    $\log\lambda_\infty(0)=0$, the strong maximum principle gives
    $\lambda_\infty\equiv1$ on $U$. Thus $F_\infty|_U$ is a flat
    superminimal immersion, contradicting the first part of the proof.
    This proves
    \eqref{Equation: hyperbolic metric bound under nonpositive curvature}.

    Finally, Remark~\ref{Remark: degree and area} and the Gauss--Bonnet
    theorem for $h_\Sigma$ give
    \begin{align*}
        4\pi d
        =
        \operatorname{Area}(\varphi)
        \leq
        C_n\operatorname{Area}(\Sigma,h_\Sigma)
        =
        4\pi C_n(g-1),
    \end{align*}
    which proves
    \eqref{Equation: degree bound under nonpositive curvature}.
\end{proof}
\begin{remark}
    In particular, for fixed $n$, any sequence of linearly full branched superminimal immersions 
    \begin{align*}
        \varphi_\ell:\Sigma_\ell\to\s^{2n} 
    \end{align*}
    of degrees $d_\ell\to+\infty$ with strictly negative Gaussian curvature away from their branch loci must satisfy
    \begin{align*}
        g(\Sigma_\ell)\geq1+\frac{d_\ell}{C_n}\rightarrow+\infty.
    \end{align*}
    Thus the simultaneous growth of the genus in Conjecture~\ref{Conjecture: negative curvature} is necessary.
\end{remark}
%
%
\bibliographystyle{amsalpha} 
\bibliography{main} 
\end{document}